%% file: Casson-Lin_Arxiv_Accepted.tex
\documentclass[11pt,a4paper]{amsart}
\usepackage{a4wide}
\usepackage[utf8]{inputenc}
\usepackage[english]{babel}

\usepackage{amsfonts,amssymb,amsmath}
\usepackage[color=green!30]{todonotes}

\usepackage{xcolor}
\usepackage[colorlinks,
    linkcolor={red!50!black},
    citecolor={blue!50!black},
    urlcolor={blue!80!black}]{hyperref}

\usepackage{tikz}
\usetikzlibrary{matrix}
\usepackage[all]{xy}
\usetikzlibrary{calc, matrix, arrows, cd}

\def\Z{\mathbb{Z}}

\def\R{\mathbb{R}}

\def\N{\mathbb{N}}
\def\C{\mathbb{C}}

\newcommand{\mS}{\mathbb{S}}
\newcommand{\mC}{\mathbb{C}}
\newcommand{\mR}{\mathbb{R}}

\newcommand{\mZ}{\mathbb{Z}}
\newcommand{\ph}{\varphi}

\newcommand{\bm}{\begin{pmatrix}}
\newcommand{\ema}{\end{pmatrix}}
\newcommand{\bsm}{\left(\begin{smallmatrix}}
\newcommand{\esm}{\end{smallmatrix}\right)}

\newcommand{\Hom}{\operatorname{Hom}}

\newcommand{\Id}{\operatorname{Id}}
\newcommand{\id}{\operatorname{id}}

\newcommand{\Tr}{\operatorname{Tr}}

\newcommand{\Sldeux}{\operatorname{SL}_2}
\newcommand{\SU}{\operatorname{SU}}
\newcommand{\SO}{\operatorname{SO}}

\newcommand{\sudeux}{\mathfrak{su}_2}

\newcommand{\al}{\alpha}

\newcommand{\g}{\gamma}

\newcommand{\de}{\delta}

\newcommand{\la}{\lambda}

\newcommand{\benu}{\begin{enumerate}}
\newcommand{\eenu}{\end{enumerate}}

\newcommand{\T}{\mathbb{T}}
\newcommand{\su}{\operatorname{\mathfrak{su}}}
\tikzset{node distance=2cm, auto}

\newtheorem{theorem}{Theorem}[section]

\newtheorem{corollary}[theorem]{Corollary}
\newtheorem{lemma}[theorem]{Lemma}
\newtheorem{proposition}[theorem]{Proposition}

\newtheorem*{claim}{Claim}

\theoremstyle{definition}
\newtheorem{definition}[theorem]{Definition}

\newtheorem{example}[theorem]{Example}

\newtheorem{remark}[theorem]{Remark}

\DeclareMathOperator{\tr}{tr}

\begin{document}
\title{A multivariable Casson-Lin type invariant}
\begin{abstract}
We introduce a multivariable Casson-Lin type invariant for links in $S^3$. This invariant is defined as a signed count of irreducible $\SU(2)$ representations of the link group with fixed meridional traces. For 2-component links with linking number one, the invariant is shown to be a sum of multivariable signatures. We also obtain some results concerning deformations of $\SU(2)$ representations of link groups.
\end{abstract}
\author{Leo Benard}
\address{Institut Math\'ematique de Jussieu-Paris Rive Gauche, Sorbonne Universit\'e, France}
\email{leo.benard@imj-prg.fr}
\author{Anthony Conway}
\address{Department of Mathematics, Durham University, United Kingdom}
\email{anthonyyconway@gmail.com}
\maketitle


\section{Introduction}

The Casson-Lin invariant $h(K)$ of a knot $K$ was originally defined by Lin as a signed count of conjugacy classes of trace-free irreducible $\SU(2)$ representations of $\pi_1(S^3 \setminus K)$~\cite{Lin}. Lin proved furthermore that $h(K)$ equals half the (Murasugi) signature of $K$. Allowing for more general trace conditions, this result was later generalized by Herald~\cite{Herald} and Heusener-Kroll~\cite{HeusenerKroll} who defined an invariant $h_K(\alpha)$ for those $\alpha \in (0,\pi)$ which satisfy~$\Delta_K(e^{2i\alpha}) \neq 0$. They also related their invariant to the Levine-Tristram signature $\sigma_K$ by showing that
$$h_K(\alpha) =\frac{1}{2} \sigma_K(e^{2i \alpha}).$$
Similar invariants have been constructed for links: Harper-Saveliev~\cite{HarperSaveliev} defined a signed count of a certain type of \emph{projective} $\SU(2)$ representations for $2$-component links $L=K_1 \cup K_2$ and showed that their invariant coincides with the linking number $\pm \ell k(K_1,K_2)$. The sign was later determined by Boden-Herald~\cite{BodenHerald} and the construction was extended to $n$-component links by Boden-Harper~\cite{BodenHarper}. We also refer to~\cite{BodenHarper} for a construction involving the group $\SU(n)$ and to~\cite{HarperSavelievFloer, CollinSteer} for further gauge theoretic developments.

\medbreak

The first aim of this article is to produce a multivariable generalization of the Casson-Lin invariant. Namely, building on the approach of Lin~\cite{Lin} and Heusener-Kroll~\cite{HeusenerKroll}, we consider conjugacy classes of $\SU(2)$ representations with fixed meridional traces. More precisely, given an $n$-component ordered link $L$ and an $n$-tuple $(\alpha_1,\ldots,\alpha_n) \in (0,\pi)^n$ such that the multivariable Alexander polynomial $\Delta_L(t_1,\ldots,t_n)$ does not vanish on $\lbrace (e^{\varepsilon_1 2i\alpha_1},\ldots,e^{\varepsilon_n 2i\alpha_n}) \ \vert \ \varepsilon_j=\pm 1 \rbrace$, we define a \emph{multivariable Casson-Lin invariant} 
$$h_L(\alpha_1,\ldots,\alpha_n).$$
Generalizing the aforementioned authors' approach, this invariant is defined using (colored) braids. The invariance of $h_L$ is then proved by showing independence under the two colored Markov moves (Propositions~\ref{prop:InvarianceMarkov1} and~\ref{prop:InvarianceMarkov2}). By construction, $h_L$ recovers the invariant of Heusener-Kroll~\cite{HeusenerKroll} if $L$ is a knot, while Proposition~\ref{prop:locallyconst} shows that $h_L$ is locally constant. Note that since we are counting $\SU(2)$ representations and not projective $\SU(2)$ representations, our invariant $h_L$ is distinct from the link invariant constructed by Harper-Saveliev~\cite{HarperSaveliev} and Boden-Harper~\cite{BodenHarper}. The following paragraphs shall make this difference more concrete.
\medbreak

In~\cite{Lin, HeusenerKroll}, the invariants under consideration were related to signature invariants by studying the effect of crossing changes and computing the invariant on a ``base case", namely the unknot. In our setting, this task is complicated by the following fact: if $L$ and $L'$ are related by a crossing change and $\Delta_L$ is not identically zero, it might well be that $\Delta_{L'} \equiv 0$, and in this case, $h_{L'}$ is not defined. Furthermore, since the Alexander polynomial of the $n$-component unlink is trivial (for $n \geq 2$), there is no obvious ``base case".

While we have not managed to circumvent this issue in general, we nevertheless provide a formula relating $h_L(\alpha_1,\ldots,\alpha_n)$ to $h_{L'}(\alpha_1,\ldots,\alpha_n)$ whenever $L$ and $L'$ differ by a crossing change \emph{within} a component of $L$. In particular, for $2$-component links with linking number~$1$, we are then able to relate $h_L$ to the multivariable signature $\sigma_L$ of Cimasoni-Florens~\cite{CimasoniFlorens}: for this class of links, the Hopf link can be used as the ``base case". 
Here, the multivariable signature is a function on $\mathbb{T}^n_*:=(S^1 \setminus \lbrace 1 \rbrace )^n$ that generalizes the classical Levine-Tristram signature.

\begin{theorem}
\label{thm:Intro}
Let $L=K_1 \cup K_2$ be a $2$-component ordered link with $\ell k(K_1,K_2)=~1$, let~$(\alpha_1, \alpha_2) \in (0, \pi)^2$, and set $(\omega_1,\omega_2)=(e^{2i\alpha_1},e^{2i\alpha_2})$. 
If the multivariable Alexander polynomial satisfies $\Delta_L(\omega_1^{\varepsilon_1}, \omega_2^{\varepsilon_2}) \neq 0$ for all~$(\varepsilon_1,\varepsilon_2) \in \lbrace \pm 1 \rbrace^2$, then the following equality holds:
\begin{equation}
\label{eq:WantedTheorem} h_L(\alpha_1,\alpha_2)=\frac{-1}{2}(\sigma_L(\omega_1,\omega_2)+\sigma_L(\omega_1,\omega_2^{-1})).
\end{equation}
\end{theorem}

First, we observe that despite appearances, the formula displayed in~\eqref{eq:WantedTheorem} is in fact symmetric.
\begin{remark}
\label{rem:FormulaIsSymmetric}
Since the multivariable signature is known to satisfy $\sigma_L(\omega_1^{-1},\omega_2^{-1})=\sigma_L(\omega_1,\omega_2)$, the conclusion of Theorem~\ref{thm:Intro} can be rewritten as
\begin{equation}
\label{eq:WantedTheoremSymmetric} 
h_L(\alpha_1,\alpha_2)=\frac{-1}{4}(\sigma_L(\omega_1,\omega_2)+\sigma_L(\omega_1,\omega_2^{-1})+\sigma_L(\omega_1^{-1},\omega_2)+\sigma_L(\omega_1^{-1},\omega_2^{-1})).
\end{equation}
\end{remark}

In fact, throughout this article, we work with colored links: an $n$-component oriented link~$L$ is \emph{$\mu$-colored} if its components are partitioned into sublinks $L_1 \cup \ldots \cup L_\mu$. For instance, taking $\mu=n$, a $\mu$-colored link is an ordered link, while a $1$-colored link is simply an oriented link. In particular, in this latter case, our construction defines a one variable Casson-Lin invariant which reduces to Heusener-Kroll's invariant if $L$ is a knot. 

\begin{remark}
\label{rem:1Colored}
Theorem~\ref{thm:Intro} (as stated in~\eqref{eq:WantedTheoremSymmetric}) does not hold for 1-colored links with more than one component. 
Reformulating, if $\alpha \in (0,\pi)$ and $L$ is an oriented link with at least two components, then~$-\frac{1}{4}\left(\sigma_L(e^{2 i \alpha})+\sigma_L(e^{-2 i \alpha}) \right)=-\frac{1}{2}\sigma_L(e^{2i\alpha})$ need not equal $h_L(\alpha)$, even under the assumptions of the theorem
(for knots $h_L(\alpha)=-\frac{1}{2}\sigma_L(e^{2i\alpha})$ holds by Heusener and Kroll's work; the sign difference is discussed in Remark~\ref{rem:CrossingChangeSignature}).
Indeed, the equality is  false for the (one-colored) Hopf link:
regardless of the number of colors, $h_J$ vanishes for the Hopf link~$J$ (since~$\pi_1(S^3 \setminus J)$ is abelian), while the 1-variable signature of $J$ (i.e. the Levine-Tristram signature) is equal to~$1$ or $-1$ depending on the orientation.
\end{remark}

Even though we ignore whether the linking number hypothesis is necessary in Theorem~\ref{thm:Intro}, some comments can be made nonetheless.

\begin{remark}
\label{rem:Formula1Colored}
We claim that if Theorem~\ref{thm:Intro} is true for an arbitrary $2$-component ordered link $L=K_1 \cup K_2$, then the equality  $h_L(\frac{\pi}{2})=- \sigma_L(-1)-\ell k(K_1,K_2)$ holds for the underlying oriented link (which we also denote by $L$) whenever $\Delta_L(-1) \neq 0$.
In particular, if~$\ell k(K_1,K_2)=1$, then Theorem~\ref{thm:Intro} implies the following equality (which sheds further light on Remark~\ref{rem:1Colored}):
$$  h_L\left(\frac{\pi}{2}\right)=-\sigma_L(-1)-1.$$
This contrasts with the knot case, where Lin showed that $h_K(\frac{\pi}{2})=\frac{1}{2}\sigma_K(-1)$~\cite{Lin}.
The claim is established by noting that for an $n$-component ordered link $L=K_1 \cup \ldots \cup~K_n$, the multivariable Casson-Lin invariant satisfies $h_L(\alpha,\ldots,\alpha)=h_L(\alpha)$ (Remark~\ref{rem:Diagonal} below),
while the multivariable signature satisfies $\sigma_L(\omega,\ldots,\omega)=\sigma_L(\omega)+\sum_{i<j} \ell k(K_i,K_j)$~\cite[Proposition~2.5]{CimasoniFlorens}, and $\Delta_L(t,\ldots,t)(t-1)=\Delta_L(t)$.
\end{remark}

Summarizing, the multivariable Casson-Lin invariant is related to the multivariable signature for knots and 2-component ordered links with linking number $1$, but is \emph{a priori} a new invariant in general. Note that the resemblance between (abelian invariants of) 2-component links with linking number $1$ and knots was already observed and exploited in~\cite{FriedlPowellConcordanceHopf}.

\begin{remark}
\label{rem:NotHomotopy}
It should now be clear that our multivariable Casson-Lin invariant $h_L$ differs from the invariant of Harper-Saveliev~\cite{HarperSaveliev} and Boden-Harper~\cite{BodenHarper}. As an additional remark in this direction, it is interesting to note that this latter count of projective representations might be a link homotopy invariant~\cite[discussion following Conjecture 4.7]{BodenHarper}, while this seems unlikely for our~$h_L$: the statement is already incorrect for 2-component links with linking number $1$ since the multivariable signature is not a link homotopy invariant.
\end{remark}

The second aim of this paper is to provide some results on deformations of $\SU(2)$ representations of link groups. In other words, we study whether an abelian $\SU(2)$ representation of a link group is a limit point of irreducible representations. Before providing some history and stating our results, we introduce some notation. Given an $n$-component ordered link $L=K_1 \cup \ldots \cup K_n$ (whose exterior in $S^3$ is denoted by $M_L$) and $\omega=(\omega_1,\ldots,\omega_n) \in \T_*^n$, we consider the abelian representation  
$$  \rho_\omega \colon \pi_1(M_L) \to \SU(2), \ \ \ \ \rho_\omega(\gamma) = \bm \prod\limits_{i=1}^n \omega_i^{\ell k(\gamma,K_i)} & 0 \\ 0 & \prod\limits_{i=1}^n \omega_i^{-\ell k(\gamma,K_i)}\ema.$$
In the knot case (i.e. $n=1$), it is known since Burde~\cite{Burde} and de Rham~\cite{DR67} that if~$\rho_\omega$ is a limit of irreducible $\SU(2)$ representations, then $\Delta_K(\omega^2)=0$. Frohman and Klassen have shown that the converse holds if $\omega$ is a \emph{simple} root of $\Delta_K(t)$~\cite{FrohmanKlassen}. This result was generalized by Herald~\cite{Herald} and Heusener-Kroll~\cite{HeusenerKroll}: these authors used Casson-Lin type invariants to show that if $\omega$ is a root of $\Delta_K(t)$ and if the Levine-Tristram signature $\sigma_K$ changes value at~$\omega$, then~$\rho_\omega$ is a limit of irreducible representations. We refer to~\cite{BenAbdelghani} for other results in this direction and to~\cite{HeusenerPortiSuarez,HeusenerPorti} (and references therein) for deformations of $\operatorname{SL}_n(\C)$ representations.

In the case of links, these questions seemed to have received less attention. Our first result in this context is a multivariable generalization of the theorem of Burde and de Rham. While our results hold for colored links and also concern $\operatorname{SL}_2(\C)$ representations (Theorems~\ref{thm:AlexReducible} and~\ref{thm:Alexclosed}), we only state the following result on $\SU(2)$ representations, see Corollary~\ref{cor:AlexReducibleSU2}:
 
\begin{proposition}
\label{prop:MultivarBurdedeRhamIntro}
Let $L$ be an $n$-component link and let $\omega=(\omega_1,\ldots,\omega_n) \in \T_*^n$. If the abelian representation~$\rho_\omega$ is a limit of irreducible $\SU(2)$ representations, then $\Delta_L(\omega_1^2,\ldots,\omega_n^2)=0$.
\end{proposition}

Just as in the knot case, one might now wonder about the converse of Proposition~\ref{prop:MultivarBurdedeRhamIntro}. Our final result uses Theorem~\ref{thm:Intro} to provide a partial converse for 2-component links with linking number~$1$ (in the spirit of Herald's and Heusener-Kroll's result~\cite{Herald, HeusenerKroll} which involved the Levine-Tristram signature). To state our result, we use $V(\Delta_L) \subset \T^n_*$ to denote the variety described by the intersection of $ \T^n_*$ with the zero-locus of the multivariable Alexander polynomial of an $n$-component link $L$.

\begin{theorem}
\label{thm:DeformationIntro}
Let $L$ be a 2-component ordered link with linking number $1$.
Let $(\omega_1,\omega_2) \in~\T_*^2$ be such that $\Delta_L(\omega_1,\omega_2)=0$ and $\Delta_L(\omega_1,\omega_2^{-1}) \neq 0$. 
 Assume that for any open subset $U \subset~\T_*^2$ containing~$(\omega_1,\omega_2)$, the multivariable signature $\sigma_L$ is not constant on~$U \setminus (V(\Delta_L) \cap U)$.
 Then the abelian representation $\rho_{(\omega_1,\omega_2)}$  is a limit of irreducible representations.
\end{theorem}

This paper is organized as follows. In Section~\ref{sec:AlexanderAbelian}, we review some facts about representation spaces and prove Proposition~\ref{prop:MultivarBurdedeRhamIntro}. In Section~\ref{sec:MultivariableCassonLin}, we define the multivariable Casson-Lin invariant~$h_L$. In Section~\ref{sec:ColoredGassner}, we review some facts about the colored Gassner representation and the multivariable potential function. In Section~\ref{sec:CrossingChange}, we study the effect of crossing changes on~$h_L$ and, in Section~\ref{sec:Signature}, we prove Theorems~\ref{thm:Intro} and~\ref{thm:DeformationIntro}..

\subsection*{Acknowledgments}

We are indebted to Michael Heusener for sharing his knowledge of the Casson-Lin invariant and making available some of his computations. We also thank him and an anonymous referee for pointing out a mistake in a prior version of this work. We also wish to thank Solenn Estier for helpful discussions. We thank the University of Geneva and the Institut Math\'ematique de Jussieu - Paris Rive Gauche at which part of this work was conducted. 
LB is supported by the NCCR SwissMap, funded by the Swiss National Science Fondation.
AC thanks Durham University for its hospitality and was supported by an early Postdoc.Mobility fellowship, also funded by the Swiss National Science Fondation.

\section{$\Sldeux(\C)$ representations and the multivariable Alexander polynomial}
\label{sec:AlexanderAbelian}

The aim of this section is to obtain a necessary condition for the existence of reducible non-abelian $\Sldeux(\C)$ representations of link groups and prove Proposition~\ref{prop:MultivarBurdedeRhamIntro} from the introduction. In Subsection~\ref{sub:RepresentationSpaces}, we review representation spaces, in Subsection~\ref{sub:MultivarAlex}, we recall some facts about the multivariable Alexander polynomial and, in Subsection~\ref{sub:ProofReducible}, we state and prove our results on reducible representations of link groups. Note that while most of this paper deals with $\SU(2)$ representations, we hope that the more general~$\Sldeux(\C)$ statements of Theorems~\ref{thm:AlexReducible} and~\ref{thm:Alexclosed} might be of independent interest.

\subsection{Representation spaces}
\label{sub:RepresentationSpaces}
In this subsection, we review some basics facts and notations about representation spaces. References include~\cite{Klassen, AkbulutMcCarthy, Porti97}.
\medbreak

Let $\pi$ be a finitely generated group and let $G$ be either $\SU(2)$ or $\Sldeux(\C)$. The \emph{representation space} of $\pi$ is the set $R_G(\pi):=\Hom(\pi, G)$ endowed with the compact open topology. Choosing a set $\lbrace x_1,\ldots,x_n \rbrace$ of generators of $\pi$, the map  $R_G(\pi) \to G^n, \rho \mapsto \left(\rho(x_1), \ldots, \rho(x_n)\right) $ realizes~$R_G(\pi)$ as an algebraic subset of $G^n$. A representation $\rho \in R_G(\pi)$ is \emph{abelian} if its image is an abelian subgroup of $G$. The closed subset of abelian representations will be denoted by~$S(\pi)$. A representation is \emph{reducible} if it admits a non-trivial invariant subspace and \emph{irreducible} otherwise.

\begin{remark}
\label{rem:Reducible}
For $\SU(2)$ representation spaces, a representation $\rho$ is reducible if and only if it is abelian. This is well known not to be the case for other Lie groups such as $\Sldeux(\C)$.
\end{remark}

When $G=\SU(2)$, we write $R(\pi)$ instead of $R_{\operatorname{SU}(2)}(\pi)$. The group $\operatorname{SU}(2)$ acts on $R(\pi)$ by conjugation and the \emph{$\SU(2)$-character variety} of $\pi$ consists of the quotient $X(\pi) = R(\pi)/\SU(2)$.
After removing abelian representations, $\SO(3)=\SU(2)/\pm \Id$ acts freely and properly on $R(\pi) \setminus S(\pi)$. In practice, we shall frequently consider (subspaces of) the set of conjugacy classes of non abelian representations:
$$ \widehat{R}(\pi)=\left( R(\pi) \setminus S(\pi) \right) /\SO(3).$$
When $G=\Sldeux(\mC)$, the quotient $R_{\operatorname{SL}_2}(\pi)/\operatorname{SL}_2$ is not Hausdorff in general. In this case, the \emph{$\Sldeux$-character variety} of $\pi$ is the algebro-geometric quotient 
$$X_{\Sldeux}(\pi) = R_{\Sldeux}(\pi)/\!/ \Sldeux(\mC).$$
While references for the algebro-geometric quotient include~\cite{LM85,CS83, Porti97}, all we need in the sequel is the following fact: two representations $\rho$ and $\rho'$ are identified in $X_{\Sldeux}(\pi)$ if their traces are equal, i.e. if for all $\g \in \pi$ one has $\Tr \rho(\g) = \Tr \rho'(\g)$. As a consequence, this quotient is the usual one when restricted to the set of irreducible representations (see for instance \cite[Proposition 1.5.2]{CS83}).

Finally, when $\pi$ is the fundamental group of a manifold $M$, we write $R(M)$ instead of~$R(\pi)$. In fact, we are particularly interested in the case where $M$ is a link exterior.

\subsection{The multivariable Alexander polynomial and reducible $\Sldeux(\C)$ representations}
\label{sub:MultivarAlex}

In this subsection, we first briefly review the multivariable Alexander polynomial before stating a criterion for the existence of reducible non-abelian $\Sldeux(\C)$-representations. We also prove Proposition~\ref{prop:MultivarBurdedeRhamIntro} from the introduction.
\medbreak

A~\emph{$\mu$-colored link} is an oriented link~$L$ in~$S^3$ whose components are partitioned into~$\mu$ sublinks~$L_1\cup\dots\cup L_\mu$. Given an $n$-component $\mu$-colored link~$L$, we let~$M_L$ denote its exterior, we consider the homomorphism $\ph \colon \pi_1(M_L) \to \Z^\mu, \gamma \mapsto (\ell k(L_1,\gamma),\ldots,\ell k(L_\mu,\gamma))$ and use~$m_1,\ldots,m_n$ to denote the meridians of $L$.

\begin{remark}
\label{rem:ReducibleRepresentation}
Any reducible representation $\rho \colon \pi_1(M_L) \to \Sldeux(\mC)$ is conjugated to one which satisfies $\rho(m_i) = \left( \begin{smallmatrix}  \la_i & \ast\\0 & \la_i^{-1} \end{smallmatrix} \right)$ for some $\lambda_i \in \C$ and for $i=1,\ldots, n$. Using $K_1,\ldots,K_n$ to denote the connected components of $L$, observe that  for $\gamma \in \pi_1(M_L)$, this representation satisfies
\begin{equation}
\label{eq:Representation}
\rho(\g) = \bm \prod\limits_{i=1}^n \la_i^{\ell k(\gamma,K_i)} & \ast \\ 0 & \prod\limits_{i=1}^n \la_i^{-\ell k(\gamma,K_i)}\ema.
\end{equation}
Next we bring the colors into play. Assume that $\la=(\la_1,\ldots, \la_\mu)$ lies in $(\mC^*)^\mu$ and let $\rho_\la \colon \pi_1(M_L) \to \Sldeux(\C)$ be the representation which maps $m_j$ to $\left( \begin{smallmatrix}  \la_i & \ast\\0 & \la_i^{-1} \end{smallmatrix} \right)$ if $m_j$ belongs to the sublink $L_i$. Note that if $\mu=n$, this recovers the representation described in~(\ref{eq:Representation}). Consider the composition $\varphi_\la \colon \pi_1(M_L) \stackrel{\varphi}{\to} \Z^\mu \to  \C$, where the second map sends the canonical basis element $e_i$ to $\la_i$. If $\gamma$ lies in $\pi_1(M_L)$, then $\rho_\lambda(\gamma)$ can be written explicitly as
\begin{equation}
\label{eq:RepresentationColored}
\rho_\la(\gamma)=\left( \begin{smallmatrix} \varphi_\la(\gamma) & * \\ 0 & \varphi_\la(\gamma)^{-1}  \end{smallmatrix} \right)= \bm \prod\limits_{i=1}^\mu \la_i^{\ell k(\gamma,L_i)} & \ast \\ 0 & \prod\limits_{i=1}^\mu \la_i^{-\ell k(\gamma,L_i)}\ema. 
\end{equation}
\end{remark}

As observed in Remark~\ref{rem:Reducible}, reducible $\Sldeux(\C)$ representations need not to be abelian. In order to describe this situation in more details, we recall the definition of the Alexander polynomial of a colored link. The previously described epimorphism~$\varphi \colon \pi_1(M_L) \to \Z^\mu$ induces a regular~$\Z^\mu$-covering~$\widehat{M}_L \rightarrow M_L$. The homology of~$\widehat{M}_L$ is a module over~$\Lambda_\mu:=\Z[t_1^{\pm 1},\dots,t_\mu^{\pm 1}]$, and the~$\Lambda_\mu$-module~$H_1(\widehat{M}_L)$ is called the {\em Alexander module\/} of the colored link~$L$. 

\begin{definition}
\label{def:AlexanderColored}
The \textit{Alexander polynomial} $\Delta_L(t_1,\ldots,t_\mu)$ of a $\mu$-colored link $L$ is the order of its Alexander module.
\end{definition}

The Alexander polynomial is only well defined up to units of $\Lambda_\mu$, that is, up to multiplication by powers of $\pm t_i$. We refer to~\cite{KawauchiBook, CimasoniPotential} for further information on $\Delta_L$.
The main theorem of this section is the following.

\begin{theorem}
\label{thm:AlexReducible}
Let $L$ be a $\mu$-colored link and let $\la=(\la_1,\ldots,\la_\mu)$ lie in $(\mC^* \setminus \lbrace 1 \rbrace)^\mu$. There exists a reducible, non abelian $\Sldeux(\C)$-representation of the form $\rho_\la$ if and only if $\Delta_L(\la^2)=0$.
\end{theorem}

We delay the proof of Theorem~\ref{thm:AlexReducible} to Subsection~\ref{sub:ProofReducible} and give some applications instead. 

\begin{theorem}\label{thm:Alexclosed}
Let $L$ be a $\mu$-colored link and let $\la=(\la_1,\ldots,\la_\mu)$ lie in $(\mC^* \setminus \lbrace 1 \rbrace)^\mu$.  If $\Delta_L(\la^2)\neq 0$, then a sufficiently small neighborhood of the representation $\rho_\la$ in $R_{\Sldeux(\mC)}(M_L)$ consists entirely of reducible representations.
\end{theorem}

\begin{proof}
The strategy of the proof follows~\cite[Lemma 3.9, (iii)]{Porti97}).
A representation $\rho \colon \pi \to \Sldeux(\mC)$ is reducible if and only if for any $\gamma, \delta \in \pi$, one has $\Tr \rho(\gamma\delta\gamma^{-1}\delta^{-1}) = 2$~\cite[Lemma 1.2.1]{CS83}. Consequently, reducibility is well defined at the level of character varieties, and the set of irreducible characters is open in both the representation variety $R_{\mathrm{SL}_2}(M_L)$ and in the character variety $X_{\mathrm{SL}_2}(M_L)$. 

Since $\Delta_L(\la^2)\neq 0$, Theorem \ref{thm:AlexReducible} implies that $\rho_\la$ is abelian. In fact, we claim that every representation $\rho'$ with the same character as $\rho_\lambda$ is abelian and is conjugated to $\rho_\la$. To see this, first note that since $\rho'$ has the same character as $\rho_\lambda$, the previous paragraph implies that~$\rho'$ is reducible. Using Theorem \ref{thm:AlexReducible}, $\rho'$ must in fact be abelian. Since $\rho_\lambda$ and $\rho'$ are abelian and have the same character, they must be conjugated, concluding the proof of the claim.

By way of contradiction, assume that the representation $\rho_\la$ has irreducible representations in anyone of its neighborhoods. Since we argued that the set of irreducible characters is open in the character variety, the character of the representation $\rho_\la$ lies in an irreducible component $X\subset X_{\mathrm{SL}_2}(M_L)$ that contains the character of an irreducible representation.

Next, consider the quotient map $t \colon R_{\mathrm{SL}_2}(M_L)  \to X_{\mathrm{SL}_2}(M_L)$. If $\chi \in  X$ is the character of an irreducible representation, then the fiber $t^{-1}(\lbrace \chi \rbrace)$ is homeomorphic to $\operatorname{PSL}(2,\C)$ and in particular has dimension $3$. Since irreducible characters form an open dense subset of~$X$ and since the dimension of the fiber $t^{-1}(\lbrace \chi \rbrace)$ is upper semi-continuous on $X$ for any character~$\chi$ in~$X$, the dimension of $t^{-1}(\lbrace \chi \rbrace)$ is at least 3.

Set $\chi_\lambda:=t(\rho_\lambda)$. Since $\rho_\la$ is abelian, the claim implies that $t^{-1}(\lbrace \chi_\lambda \rbrace)$ is isomorphic to $\operatorname{SL}_2(\C)/G_{\rho_\lambda}$, where $G_{\rho_\la} \leq \Sldeux(\mC)$ is the stabilizer of $\rho_\lambda$ and has positive dimension (since $\rho_\lambda$ is abelian). Therefore the fiber $t^{-1}(\lbrace \chi_\la \rbrace)$ has dimension strictly less than $3$ which contradicts the previous paragraph.
\end{proof}

In the next sections, our interest will lie in $\SU(2)$ representations. In this case, as recalled in Remark~\ref{rem:Reducible}, every reducible representation is abelian and the resulting eigenvalues lie on the unit circle. Using $\T_*^\mu$ to denote $(S^1\setminus \lbrace 1 \rbrace)^\mu$, we obtain the following result which generalizes a theorem of Burde~\cite{Burde} and de Rham~\cite{DR67}. This proves Proposition~\ref{prop:MultivarBurdedeRhamIntro} from the introduction.

\begin{corollary}
\label{cor:AlexReducibleSU2}
Let $L$ be a $\mu$-colored link and let $\omega$ lie in $\T_*^\mu$.  If $\Delta_L(\omega^2)\neq 0$, then a sufficiently small neighborhood of $\rho_\omega$ in $R(M_L)$ consists entirely of abelian representations.
\end{corollary}
\begin{proof}
This follows directly from Theorem \ref{thm:Alexclosed} and the observation that~$R(M_L)$ embeds in~$R_{\mathrm{SL}_2}(M_L)$: any $\SU(2)$ representation is also an $\Sldeux(\C)$ representation.
\end{proof}

\subsection{Proof of Theorem~\ref{thm:AlexReducible}}
\label{sub:ProofReducible}

The map $\varphi_\la \colon \pi_1(M_L) \to \C$ described in Remark~\ref{rem:ReducibleRepresentation} endows~$\mC$ with a left $\mZ[\pi_1(M_L)]$-module structure;
we write $\C_\la$ for emphasis. Consider the twisted cochain complex $C^*(\pi_1(M_L), \mC_\la)$ and recall that a 1-cocycle $u\in Z^1(\pi_1(M_L), \mC_\la)$ is a map $u \colon \pi_1(M_L) \to \mC$ that satisfies 
\begin{equation}\label{cocycle}
 u(\g\de) = u(\g) +  \varphi_\la(\gamma) u(\de)
\end{equation}
for every $\g, \de$ in $\pi_1(M_L)$. The following lemma provides a cohomological obstruction for a reducible representation to be abelian. 

\begin{lemma}
\label{lem:Cocycle}
Given $\la \in (\mC^* \setminus \lbrace 1 \rbrace)^\mu$, the following assertions hold:
\begin{enumerate}
\item The representation $\rho_\la$ gives rise to a cocycle $u \in Z^1(M_L, \mC_{\la^2})$.
\item The representation $\rho_\la$ is abelian if and only if $[u]=0 \in H^1(M_L, \mC_{\la^2})$.
\end{enumerate}
\end{lemma}

\begin{proof}
Using the definition of $\rho_\lambda$, we may write $\rho_\la(\g)$ as $\left( \begin{smallmatrix} \varphi_\la(\gamma) &\varphi_\la(\gamma)^{-1} u(\g)  \\ 0 & \varphi_\la(\gamma)^{-1} \end{smallmatrix} \right)$ for each $\gamma$ in~$\pi_1(M_L)$ and this gives rise to a map $u \colon \pi_1(M_L) \to \mC$. Given~$\gamma$ and $\delta$ in $\pi_1(M_L)$, the equality $\rho(\g\de) = \rho(\g)\rho(\de)$ then shows that~$u$ satisfies the following relation:
$$\varphi_\la(\g\de)^{-1} u(\g\de) =\varphi_\la(\g\de^{-1}) u(\de) + \varphi_\la(\gamma \delta)^{-1} u(\g).$$ 
Multiplying this equation by $\varphi_\la(\gamma \delta)$, we deduce that $u$ must satisfy $ u(\g\de) = u(\g) + \varphi_\la(\gamma^2) u(\de)$ which is the cocycle condition from~(\ref{cocycle}). Thus $u$ is a cocycle and the first assertion is proved.

 To prove the second assertion, we must show that the reducible representation $\rho_\la$ is abelian if and only if the cocycle $u$ is a coboundary, that is if there exists a $z \in \mC$ such that for all~$\g \in \pi_1(M_L)$, one has
\begin{equation}
\label{eq:Coboundary}
u(\g)~=~(\varphi_\la(\gamma^2)-1)z.
\end{equation}
First, observe that $\rho_\la$ is abelian if and only if for each $\gamma \in \pi_1(M_L)$ there exists an invertible matrix $A = \left( \begin{smallmatrix} a&b\\ c&d \end{smallmatrix} \right)$ such that $DA=A \rho_\la(\gamma)$, where $D$ denotes the diagonal matrix $ \left( \begin{smallmatrix} \varphi_\la(\gamma) & 0 \\ 0 & \varphi_\la(\gamma)^{-1} \end{smallmatrix} \right) $. Writing out this equation coordinate by coordinate, one deduces that $\rho_\lambda$ is abelian if and only if the three following equations hold:
$$
\begin{cases}
b \varphi_\la(\gamma)&=  \ a  \varphi_\la(\gamma)^{-1} u(\g)+ b \varphi_\la(\gamma)^{-1},\\
c \varphi_\la(\gamma)^{-1}&= \ c  \varphi_\la(\gamma),\\
d  \varphi_\la(\gamma)^{-1}&= \ c  \varphi_\la(\gamma)^{-1} u(\g) + d  \varphi_\la(\gamma)^{-1}.
\end{cases}$$
If $\varphi_\la(\gamma) = \pm 1$, the representation is abelian if and only if there exists $a,c$ such that $a u(\gamma)=0$ and $c u(\gamma)=0$. Since $A$ must be invertible, either $a$ or $c$ must be non-zero, and in this case~$u(\gamma)$ must vanish for each $\gamma$. In particular $[u]$ vanishes in cohomology.

If $\varphi_\la(\gamma) \neq \pm 1$, the representation is abelian if and only if $c=0$ and $au(\g) =b(\varphi_\la(\gamma^2)-1)$. Since $A$ is invertible, we deduce that $a \neq 0$ and therefore $u(\g) = \frac{b}{a}(\varphi_\la(\gamma^2)-1)$. Consequently, looking back to~(\ref{eq:Coboundary}), we have obtained the defining equation for a coboundary. This concludes the proof of the second assertion and thus the proof of the lemma.
\end{proof}

As we shall see shortly, Theorem~\ref{thm:AlexReducible} will follow promptly from the following lemma.

\begin{lemma}
\label{lem:AlexanderVanish}
Let $\la$ lie in $(\mC^* \setminus \lbrace 1 \rbrace)^\mu$. The complex vector space $H^1(M_L;\C_\la)$ does not vanish if and only if $\la$ satisfies~$\Delta_L(\la) =0$.
\end{lemma}

\begin{proof}
First, since we are dealing with $\C$-vector spaces,
the universal coefficient theorem shows that the vanishing of $H^1(M_L;\C_\la)$ is equivalent to the vanishing of $H_1(M_L;\C_\la)$. 
To prove the lemma, we must show that $H_1(M_L;\C_\la)$ does not vanish if and only if~$\Delta_L(\la)$ vanishes. It is enough to show that  the order of this latter vector space is zero if and only if~$\Delta_L(\la)$ is zero. This will immediately follow if we prove that 
$$H_1(M_L;\C_\la) \cong H_1(M_L;\Lambda_\mu) \otimes_{\Lambda_\mu} \C_\la.$$
To prove this assertion, we will use (a particular case of) the universal coefficient spectral sequence (UCSS) whose second page is given by $E^2_{p,q}=\operatorname{Tor}_p^{\Lambda_\mu}(H_q(M_L;\Lambda_\mu),\C_\la)$ and which converges to $H_*(M_L;\C_\la)$, see~\cite[Chapter 2]{Hillman}. We start with the following claim.

\begin{claim}
 Endow $\Z=H_0(M_L;\Lambda_\mu)$ with the  $\Lambda_\mu$-module structure coming from the augmentation homomorphism $\Lambda_\mu \to \Z, t_i \mapsto 1$. If $\la$ lies in $(\mC^* \setminus \lbrace 1 \rbrace)^\mu$, then the complex vector space $\operatorname{Tor}_k^{\Lambda_\mu}(\Z,\C_\la)$ vanishes for $k=1,2$.
\end{claim}
\begin{proof}
Using the $\Lambda_\mu$-resolution of $\Z$ given by the chain complex for the universal cover of the torus~$\mathbb{T}^\mu$,
we have $\operatorname{Tor}_k^{\Lambda_\mu}(\Z,\C_\la)=H_k(\mathbb{T}^\mu;\C_\la)$. As the $\lambda_i$ are not equal to $1$, the claim now follows from considerations involving cellular homology, see~\cite{Nicolaescu} and~\cite[Lemma 2.2]{ConwayFriedlToffoli}.
\end{proof}

Using the claim, we know that $E^2_{2,0}=0$. The UCSS then gives $E_{2,0}^\infty=0$ and provides a filtration $0 \subset F_1^0 \subset F_1^1=H_1(M_L;\C_\la)$. As the UCSS also implies that  $F_1^0=E_{0,1}^\infty=H_1(M_L;\Lambda_\mu) \otimes_{\Lambda_\mu} \C_\la$ and~$E_{1,0}^\infty \cong F_1^1/F_1^0$, we obtain the following short exact sequence:
$$ 0 \to H_1(M_L;\Lambda_\mu) \otimes_{\Lambda_\mu} \C_\la \to H_1(M_L;\C_\la) \to \operatorname{Tor}_1^{\Lambda_\mu}(H_0(M_L;\Lambda_\mu),\C_\la) \to 0.$$
Since we showed in the claim that $\operatorname{Tor}_1^{\Lambda_\mu}(H_0(M_L;\Lambda_\mu),\C_\la)$ vanishes, the lemma follows. 
\end{proof}

Combining these two lemmas, we are now in position to conclude the proof of Theorem~\ref{thm:AlexReducible}.

\begin{proof}[Proof of Theorem \ref{thm:AlexReducible}]
Let $u_{\rho_\la}$ be the $1$-cocycle described in Lemma~\ref{lem:Cocycle}. Using the second point of this same lemma, the existence of a reducible non abelian representation~$\rho_\la$ is equivalent to the cohomology class $[u_{\rho_\la}]$ being non zero in $H^1(M_L,\C_{\la^2})$. Thus, if there exists a reducible non-abelian representation of the form $\rho_\la$, then $H^1(M_L,\C_{\la^2})$ is non-trivial and Lemma~\ref{lem:AlexanderVanish} implies that the multivariable Alexander polynomial $\Delta_L$ vanishes at $\la^2$. 

Conversely, if the multivariable Alexander polynomial vanishes at $\la^2$, then Lemma~\ref{lem:AlexanderVanish} implies that $H^1(M_L;\C_{\la^2})$ does not vanish. Since $H^1(M_L;\C_{\la^2})=H^1(\pi_1(M_L);\C_{\la^2})$, we deduce that there is a non-zero cocycle $u$ in $Z^1(\pi_1(M_L);\C_{\la^2})$. Defining a representation $\rho$ from~$u$ just as in the proof of Lemma~\ref{lem:Cocycle} produces the desired non-abelian representation.
\end{proof}

\section{The multivariable Casson-Lin invariant}
\label{sec:MultivariableCassonLin}

The goal of this section is to define the multivariable Casson-Lin invariant. More precisely, in Subsection~\ref{sub:ColoredBraids}, we review colored braids, in Subsection~\ref{sub:TheInvarianth}, we define our invariant on braids and in Subsection~\ref{sub:Markov} we verify its invariance under the colored Markov moves.

\subsection{Colored braids}
\label{sub:ColoredBraids}
In this subsection, we briefly review colored braids and discuss the action of the colored braid groups on $\SU(2)^n$. References for colored braids include~\cite{Murakami, ConwayTwistedBurau}, while discussions of the action of the braid group $B_n$ on $\SU(2)^n$ include~\cite{Lin, HeusenerKroll, HeusenerOrientation, Long}.
\medbreak

The braid group $B_n$ admits a presentation with $n-1$ generators $\sigma_1,\sigma_2, \dots, \sigma_{n-1}$ subject to the relations $\sigma_i \sigma_{i+1} \sigma_i=\sigma_{i+1} \sigma_i \sigma_{i+1}$ for each $i$, and $\sigma_i \sigma_j = \sigma_j \sigma_i$ if $|i-j|>2$. Topologically, the generator $\sigma_i$ is the braid whose $i$-th component passes over the $(i+1)$-th component. The \textit{closure} of a braid $\beta$ is the link $\widehat{\beta}$ obtained from $\beta$ by adding parallel strands in $S^3 \setminus (D^2 \times [0,1])$.

\begin{figure}[!htb]
\centering
\includegraphics[scale=0.5]{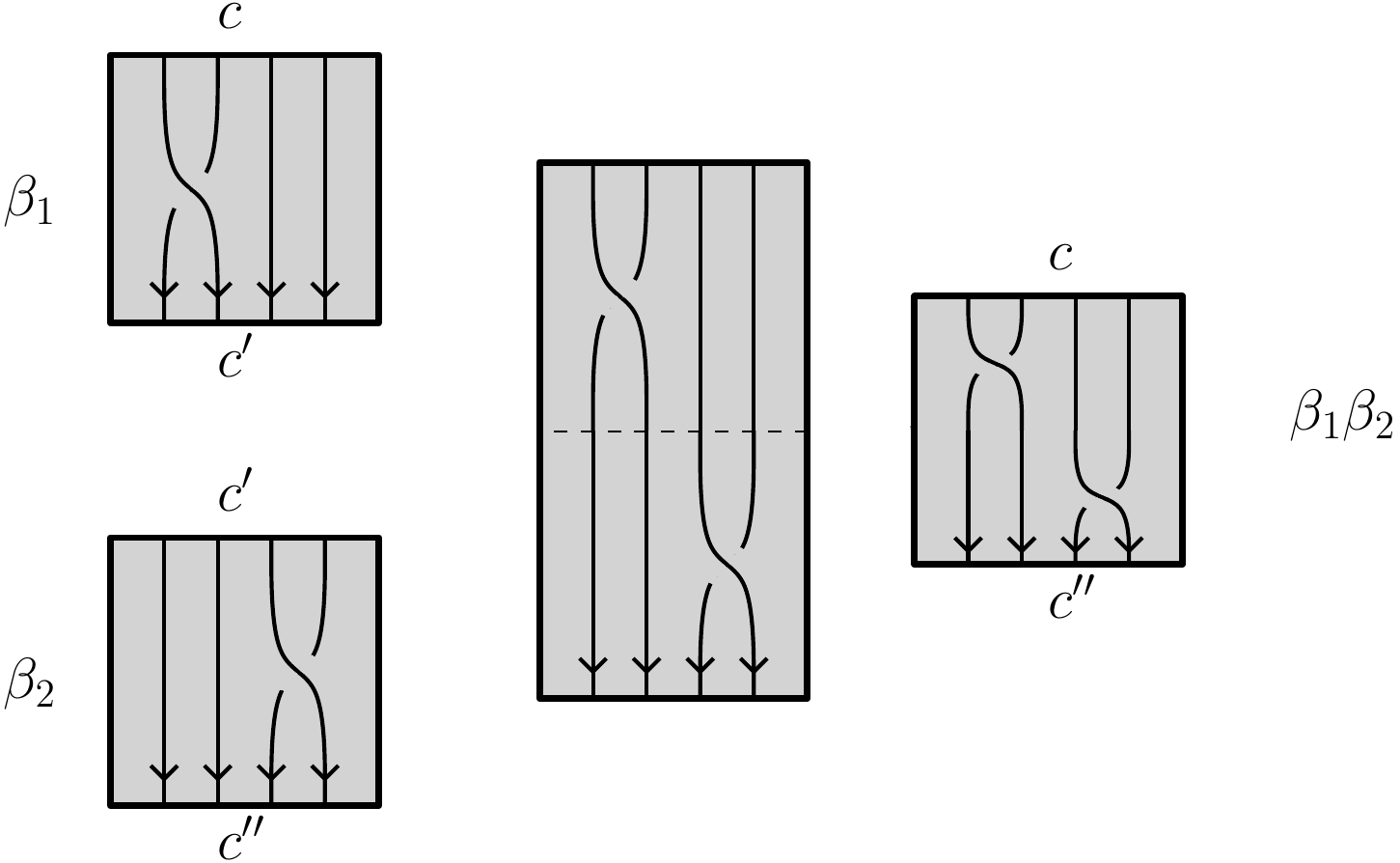}
\caption{A $(c,c')$-braid $\beta_1$, a $(c',c'')$-braid $\beta_2$ and their composition, the $(c,c'')$-braid $\beta_1\beta_2$. Here $\beta_1$ is the generator $\sigma_1$ of $B_4$, while $\beta_2$ is $\sigma_3$.}
\label{fig:CompositionThesisColor}
\end{figure}

A braid~$\beta$ is{\em~$\mu$-colored\/} if each of its components is assigned (via a surjective map) an integer in~$\{1,2,\dots,\mu\}$ (which we call a \emph{color}). A~$\mu$-colored braid induces a coloring on its top and bottom boundary components.  A~$\mu$-colored braid is then called a~\emph{$(c,c')$-braid}, where~$c$ and~$c'$ are the sequences of~$ 1, 2,\dots,\mu$ induced by the coloring of the braid (these sequence will be referred to as \emph{$\mu$-colorings}). We shall denote by~$\id_c$ the isotopy class of the trivial~$(c,c)$-braid.  The composition of a $(c,c')$-braid $\beta_1$ with a $(c',c'')$-braid $\beta_2$ is the $(c,c'')$-braid $\beta_1 \beta_2$ depicted in Figure~\ref{fig:CompositionThesisColor}. Thus, for any $c$, we obtain a \emph{colored braid group} $B_c$ which consists of isotopy classes of $(c,c)$-braids. For instance, if $\mu=1$ (so that $c=(1,\ldots,1)$), then $B_c$ is the braid group $B_n$, while if $\mu = n$ and $c_i=i$ for each $i$, then $B_c$ is the pure braid group $P_n$. We shall often use the map $i_{c_{n+1}} \colon B_c \hookrightarrow B_{(c_1, \ldots, c_n, c_{n+1})}$ which sends $\alpha$ to the disjoint union of $\alpha $ with a trivial strand of color $c_{n+1}$, see Figure~\ref{fig:Inclusion}. Here,~$c_{n+1}$ can be equal to one of the~$n$ first~$c_i$'s.

\begin{figure}[!htb]
\centering
\includegraphics[scale=0.5]{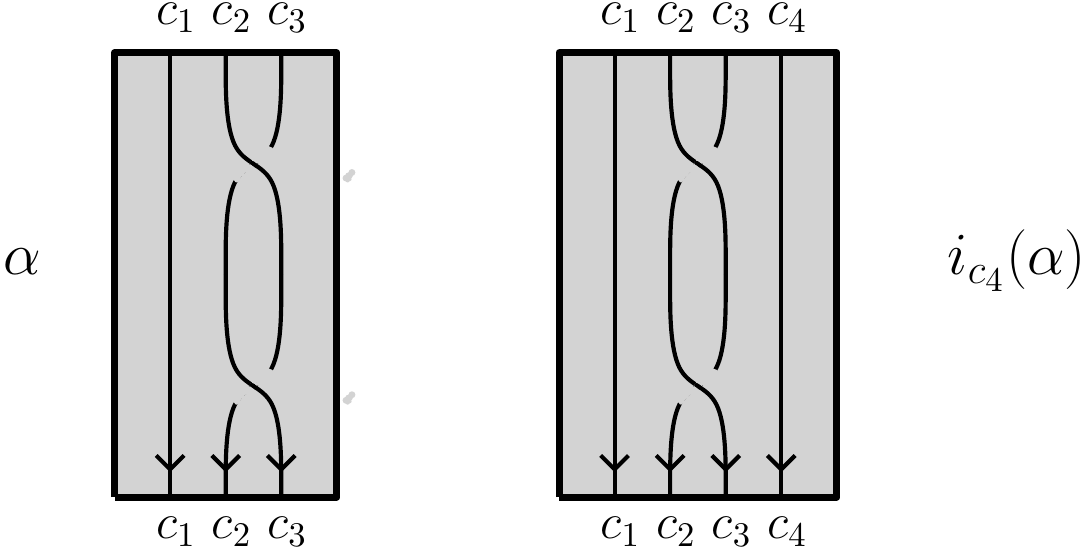}
\caption{An example of the inclusion map $i_{c_4}$.}
\label{fig:Inclusion}
\end{figure}

Finally, the closure of a $\mu$-colored braid $\beta \in B_c$ is the $\mu$-colored link $\widehat{\beta}$ obtained from $\beta$ by adding colored parallel strands in $S^3 \setminus (D^2 \times [0,1]).$ We refer to~\cite[Theorem 3.3]{Murakami} for the colored version of Alexander's theorem and instead focus on the colored version of Markov's theorem, referring to~\cite[Theorem 3.5]{Murakami} for the proof. 

\begin{proposition}
\label{prop:MarkovColored}
Two $(c,c)$-braids have isotopic closures if and only if they are related by a sequence of the following moves and their inverses: 
\begin{enumerate}
\item replace $\gamma\beta$ by $\beta \gamma$, where $\gamma$ is a $(c,c')$-braid and $\beta$ is a $(c',c)$-braid,
\item replace $\gamma$ by $\sigma^{\varepsilon}_n i_{c_n}(\gamma)$, where $\gamma$ is a $(c,c)$-braid with $n$ strands, $\sigma_n$ is viewed as a $((c_1, \ldots , c_n, c_n),(c_1, \ldots , c_n, c_n))$-braid, and $\varepsilon$ is equal to $\pm 1$.
\end{enumerate}
\end{proposition}

We conclude this subsection by discussing the action of (colored) braids on $\SU(2)^n$. Topologically, this action can be understood as follows. Any braid $\beta$ can be represented by a homeomorphism  of the punctured disk $D_n$ which fixes the boundary pointwise~\cite{Birman}. As a consequence, the braid group induces a \emph{right} action of $B_n$ on the free group $F_n=\pi_1(D_n).$ More explicitly, this action can be described on the generators $x_1,\ldots,x_n$ of $F_n$ as follows:
\begin{equation}
\label{eq:Action}
x_j \sigma_i =
\begin{cases}
x_i x_{i+1} x_i^{-1}  & \mbox{if }  j=i, \\
 x_i                            & \mbox{if }  j=i+1, \\ 
x_j                             & \mbox{otherwise. } \\ 
 \end{cases}
 \end{equation}
In particular, every braid $\beta$ induces a homeomorphism $R(D_n)  \to R(D_n)$, which we still denote by $\beta$. More concretely, identifying $R(D_n)$ with $\SU(2)^n$, this homeomorphism maps $(X_1,\ldots,X_n)$ to $ (X_1 \beta,\ldots,X_n\beta)$. So, for instance, the generator $\sigma_1 \in B_2$ acts as~$(X_1,X_2)\sigma_1=(X_1X_2X_1^{-1},X_1)$.  Note that we chose to follow Birman's conventions~\cite{Birman} and to think of~(\ref{eq:Action}) as a right action. In particular, we obtain a homomorphism $B_n \to \operatorname{Aut}(F_n)$. Working with left actions would lead to an \emph{anti}-homomorphism (see e.g.~\cite{CimasoniTuraev, ConwayEstier}).

\begin{remark}
\label{rem:ConventionsBraids}
Our conventions match those of Lin~\cite[page 339]{Lin}. On the other hand, given a braid $\beta \in B_n$ and $w \in F_n$, some authors, such as Long~\cite[page 539]{Long}, choose to define $\beta \cdot w$ as~$w \beta^{-1}$; however given an automorphism $\theta$ of $F_n$, Long then sets $\theta(X_1,\ldots,X_n):=(\theta^{-1} X_1,\ldots, \theta^{-1}X_n)$ and therefore obtains the same action of $B_n$ on $\SU(2)^n$ as we do~\cite[page 537]{Long}. On the other hand, Heusener-Kroll also use the action $\beta \cdot w=w \beta^{-1}$~\cite[Example 3.1]{HeusenerKroll} but define $\beta (X_1,\ldots,X_n)$ as $(\beta \cdot X_1,\ldots,\beta \cdot X_n)$~\cite[bottom of page 484]{HeusenerKroll}. 
\end{remark}

\begin{remark}
\label{rem:FixedPoint}
The fixed point set of the homeomorphism $\beta \colon \SU(2)^n \to \SU(2)^n$ induced by~$\beta$ can be identified with the representation space of $X_{\widehat{\beta}}$, see for instance~\cite[Lemma 1.2]{Lin}. Reformulating, $R(X_{\widehat{\beta}})$ is equal to the intersection of the diagonal~$\Lambda_n \subset \SU(2)^n \times \SU(2)^n$ with the graph $\Gamma_\beta \subset \SU(2)^n \times \SU(2)^n$ of the homeomorphism of $\SU(2)^n$ induced by $\beta$.
\end{remark}

Building on the work of Lin~\cite{Lin} and Heusener-Kroll~\cite{HeusenerKroll}, the invariant we shall define in Subsection~\ref{sub:TheInvarianth} makes crucial use of Remark~\ref{rem:FixedPoint}. Indeed we wish to ``count" (conjugacy classes of) irreducible representations in $R(X_{\widehat{\beta}})=\Lambda_n \cap \Gamma_\beta$ with certain traces fixed. For this reason, given a $\mu$-tuple $\alpha=(\alpha_1,\ldots,\alpha_\mu)$ of real numbers in $(0,\pi)^\mu$ and a coloring $c$, we shall frequently consider the following subspace of $\SU(2)^n$:
$$ R_n^{\alpha,c}=\lbrace  (X_1,\ldots,X_n) \in \SU(2)^n \vert \ \tr(X_i)=2\operatorname{cos}(\alpha_{c_i}) \text{ for } i=1,\ldots,n \rbrace. $$
In particular, observe that if $\beta$ is an $n$-stranded $(c,c')$-braid, then the aforementioned homeomorphism $\beta \colon \SU(2)^n \to \SU(2)^n$ descends to a well defined homeomorphism $\beta \colon R_n^{\alpha,c} \to R_n^{\alpha,c'}$.
 Of particular interest is the graph of this homeomorphism:
$$ \Gamma_\beta^{\alpha} = \lbrace (A_1, \ldots, A_n, A_1\beta, \ldots, A_n\beta) \ \vert \ (A_1,\ldots,A_n) \in R_n^{\alpha,c}  \rbrace \subset R_n^{\alpha,c}  \times R_n^{\alpha,c'}.$$
For instance, the trivial $(c,c)$-braid $\beta=\id_c$ induces the identity automorphism on the free group and thus on $R(F_n)=\SU(2)^n$. Thus the graph $\Gamma_{\id_c} \subset \SU(2)^n \times \SU(2)^n$ coincides with the diagonal~$\Lambda_n$. We use the following notation for the corresponding space of fixed traces:
$$\Lambda_n^{\alpha,c}= \lbrace (A_1, \ldots, A_n,A_1,\ldots, A_n)  \ \vert \ (A_1, \ldots, A_n) \in R_n^{\alpha,c} \rbrace.$$
As we alluded to above, our goal is to make sense of a signed count of conjugacy classes of irreducible representations $\rho \colon \pi_1(X_{\widehat{\beta}}) \to \SU(2)$ such that the trace of any meridian of the sublink $\widehat{\beta}_j$ of $\widehat{\beta}$ is equal to $2\cos(\al_{c_j})$. In other words, using Remark~\ref{rem:FixedPoint} and the notations of Subsection~\ref{sub:RepresentationSpaces}, we are trying to make sense of a signed count of the elements of $\widehat{\Lambda}_n^{\alpha,c} \cap \widehat{\Gamma}_\beta^{\alpha}$.

\subsection{Definition of the invariant}
\label{sub:TheInvarianth}
The goal of this subsection is to define the multivariable Casson-Lin invariant of a $(c,c)$-braid. Our approach builds on the work of Lin~\cite{Lin} and of Heusener-Kroll~\cite{HeusenerKroll}, see also~\cite{HeusenerOrientation} and~\cite{HarperSaveliev}.
\medbreak
Let $\beta$ be a $\mu$-colored $n$-stranded $(c,c)$-braid and let $\al = (\al_1, \ldots, \al_\mu)$ lie in $(0,\pi)^\mu$. The invariant that we shall consider requires us to make sense of the algebraic intersection of (quotients of)~$\Lambda_n^{\alpha,c}$ with $\Gamma_\beta^{\alpha}$ inside (a quotient of) the following space:
$$ H_n^{\alpha,c} = \lbrace(A_1, \ldots, A_n, B_1, \ldots, B_n) \in R_n^{\alpha,c} \times R_n^{\alpha,c}  \ \vert \ \prod_{i=1}^n A_i = \prod_{i=1}^n B_i \rbrace. $$
The inclusion $\Gamma_\beta^{\alpha} \subset H_n^{\alpha,c}$ holds because any braid $\xi \in B_n$ fixes $x_1\cdots x_n \in F_n$; see~\eqref{eq:Action}.
In order to count conjugacy classes of the aforementioned irreducible representations, we first need to avoid the abelian locus of the various representation varieties. For this reason, we consider the following set which should be understood (under the isomorphism~$R(F_n) \cong \SU(2)^n$) as the subspace of abelian representations of $R(F_n)$:
\begin{equation}
\label{eq:AbelianLocus}
S_n^{\alpha,c} = \lbrace(A_1, \ldots, A_n, B_1, \ldots, B_n) \in R_n^{\alpha,c} \times R_n^{\alpha,c} \ \vert \ A_iA_j = A_jA_i , A_iB_j = B_jA_i, B_iB_j=B_jB_i\rbrace.
\end{equation}
Slightly abusing notation, we shall write $S_n^{\alpha,c}$ instead of $S_n^{\alpha,c} \cap \Theta^{\alpha,c}$, where $\Theta^{\alpha,c}$ is any of the previously defined spaces $\Gamma_\beta^{\alpha}, \Lambda_n^{\alpha,c}$ or $H_n^{\alpha,c}$. As described in Subsection~\ref{sub:RepresentationSpaces}, $\SO(3)$ acts freely on the resulting sets of irreducible representations and we make the following definitions:
$$  \widehat{\Lambda}_n^{\alpha,c}=(\Lambda_n^{\alpha,c} \setminus S_n^{\alpha,c})/\SO(3), \ \ \ \ 
 \widehat{\Gamma}_\beta^\alpha=(\Gamma_\beta^\alpha \setminus S_n^{\alpha,c})/\SO(3),  \ \ \ \ 
  \widehat{H}_n^{\alpha,c}=(H_n^{\alpha,c} \setminus S_n^{\alpha,c})/\SO(3). 
$$
Observe that both $\widehat{\Lambda}_n^{\alpha,c}$ and $\widehat{\Gamma}_\beta^{\alpha}$ are smooth open $(2n-3)$-dimensional manifolds:~$\Lambda_n^{\alpha,c}$ and $\Gamma_\beta^{\alpha}$ are~$2n$ dimensional (the subspaces of matrices in $\SU(2)$ with fixed trace are 2-dimensional) and the 3-dimensional Lie group $\SO(3)$ acts freely and properly on the open manifolds $\Lambda_n^{\alpha,c} \setminus S_n^{\alpha,c}$ and $\Gamma_\beta^{\alpha} \setminus S_n^{\alpha,c}$. Recalling Remark~\ref{rem:FixedPoint}, the representations we wish to consider lie in the intersection~$\widehat{\Gamma}_\beta^{\alpha} \cap \widehat{\Lambda}_n^{\alpha,c}$, viewed as a subspace of $\widehat{H}_n^{\alpha,c}$. In order for a ``count" to make sense, we must now check that this intersection is compact and that $\widehat{\Gamma}_\beta^{\alpha}$ and $\widehat{\Lambda}_n^{\alpha,c}$ are half dimensional in $\widehat{H}_n^{\alpha,c}$. We start by proving the latter, namely we prove that $\widehat{H}_n^{\alpha,c}$ is $4n-6$ dimensional.

\begin{lemma}
\label{lem:DimensionIrred}
The space $H_n^{\alpha,c} \setminus S_n^{\alpha,c}$ is a smooth open $(4n-3)$-dimensional manifold. In particular $\widehat{H}_n^{\alpha,c}$ is $(4n-6)$ dimensional.
\end{lemma}
\begin{proof}
Consider the map $f_n \colon R_{n}^{\alpha,c} \times R_{n}^{\alpha,c}   \to\SU(2)$ defined~by $f_n(A_1, \ldots, A_n, B_1, \ldots, B_n) = A_1\cdots A_n B_n^{-1} \cdots B_1^{-1}.$ Observe that $H_n^{\alpha,c}=f_n^{-1}(\Id)$. The same arguments as in \cite[Lemma~1.5]{Lin} and~\cite[Lemma 3.3]{HeusenerKroll}) show that $f_n$ restricts to a submersion ${f_n}_|$ on $H_n^{\alpha,c} \setminus S_n^{\alpha,c}$. As a consequence, $H_n^{\alpha,c} \setminus S_n^{\alpha,c}={f_n}_|^{-1}(\Id)$ is a smooth manifold whose dimension is equal to $\dim(R_n^{\alpha,c} \times R_n^{\alpha,c})-\dim(\SU(2))=4n-3$. This concludes the proof of the lemma.
\end{proof}

Next, making use of Section~\ref{sec:AlexanderAbelian}, we show that the space $\widehat{\Gamma}_\beta^\alpha \cap \widehat{\Lambda}_n^{\alpha,c}$ is compact.
For a fixed $\alpha = (\al_1, \ldots, \al_\mu)$ in $(0,\pi)^\mu$, we consider the finite set 
$$S(\alpha) = \lbrace (e^{\varepsilon_1 2i\al_1}, \ldots, e^{\varepsilon_\mu 2i\al_\mu}) \ | \ \varepsilon_i = \pm 1 \text{ for } i=1,\ldots,\mu \rbrace .$$
The set $S(\alpha)$ contains $2^\mu$ elements, indexed by the $\varepsilon=(\varepsilon_1,\ldots,\varepsilon_\mu)$ in $\lbrace \pm 1 \rbrace^\mu$.
For this reason, we will sometimes write elements of $\mathcal{S}(\alpha)$ as $\omega_\varepsilon$, where $\varepsilon \in \lbrace \pm 1 \rbrace^\mu$.
We will only do this when no confusion occurs with the coordinates of $\omega$.

\begin{proposition}
\label{prop:Compact}
Let $\al = (\al_1, \ldots, \al_\mu)$ lie in $(0,\pi)^\mu$ and let $\beta$ be an $n$-stranded $\mu$-colored $(c,c)$-braid. If~$\Delta_{\widehat{\beta}}(\omega_\varepsilon) \neq 0$ for all $\omega_\varepsilon \in S(\alpha)$, then $\widehat{\Gamma}_\beta^\alpha \cap \widehat{\Lambda}_n^{\alpha,c}$ is compact.
\end{proposition}
\begin{proof}
Since $\SO(3)$ is compact, it is sufficient to prove that $(\Lambda_n^{\alpha,c} \setminus S_n^{\alpha,c}) \cap (\Gamma_\beta^\alpha \setminus S_n^{\alpha,c})$ is compact. As this set lies in the compact set $\SU(2)^{2n}$, we are reduced to proving that it is closed. Let~$(\rho_k)_{k \in \N}$ be a convergent sequence of representations in $(\Lambda_n^{\alpha,c} \setminus S_n^{\alpha,c}) \cap (\Gamma_\beta^\alpha \setminus S_n^{\alpha,c})$, with limit $\rho_\infty \in \SU(2)^{2n}$. Since $\Lambda_n^{\alpha,c}$ and $ \Gamma_\beta^\alpha$ are closed in $\SU(2)^{2n}$, it follows that $\rho_\infty$ lies in $\Lambda_n^{\alpha,c}\cap\Gamma_\beta^\alpha$. 
 By way of contradiction, assume that $\rho_\infty$ is abelian. The $n$ components of~$\rho_\infty \in \SU(2)^n$ are therefore simultaneously conjugated to~$\bsm \omega_{c_i}^{1/2}&0\\0&\omega_{c_i}^{-1/2}\esm$ for some $\omega:=\omega_\varepsilon$ in $S(\alpha)$. Since $\Delta_{\widehat{\beta}}(\omega) \neq 0$, Corollary~\ref{cor:AlexReducibleSU2} implies that $\rho_k$ is abelian for $k$ big enough, a contradiction. We therefore deduce that $\rho_\infty$ lies in~$(\Lambda_n^{\alpha,c} \setminus S_n^{\alpha,c}) \cap (\Gamma_\beta^\alpha \setminus S_n^{\alpha,c})$, concluding the proof of the proposition.
\end{proof}


Perturbing $\widehat{\Gamma}_\beta^\alpha$ if necessary, we can assume that it intersects transversally the diagonal~$\widehat{\Lambda}_n^{\alpha,c}$. Consequently, thanks to Proposition~\ref{prop:Compact}, we know that $\widehat{\Gamma}_\beta^\alpha \cap \widehat{\Lambda}_n^{\alpha,c}$ is a $0$-dimensional manifold. We now orient these manifolds. Use $\mathbb{S}_\theta$ to denote the set of matrices in $\SU(2)$ with trace~$2\operatorname{cos}(\theta)$. Orient this copy of $S^2$ in a fixed (but arbitrary) way. Since $R_n^{\alpha,c}$ consists of an $n$-fold product of $\mathbb{S}_{\alpha_{c_i}}$, we endow it with the product orientation. The diagonal $\Lambda_n^{\alpha,c}$ and the graph~$\Gamma_\beta^\alpha$ are naturally diffeomorphic to $R_n^{\alpha,c}$ via the projection on the first factor and they are given the induced orientations. Next, consider the map
$$ f_n \colon R_n^{\alpha,c} \times R_n^{\alpha,c} \to \SU(2)$$
which we encountered in the proof of Lemma~\ref{lem:DimensionIrred}. Using this map, we can pull back the orientation of $\SU(2)$ to obtain an orientation on $H_n^{\alpha,c} \setminus S_n^{\alpha,c}$. The adjoint action of $\SO(3)$ on~$\mathbb{S}_\theta$ is orientation preserving, hence the $\SO(3)$-quotients $\widehat{\Gamma}_\beta^\alpha, \widehat{\Lambda}_n^{\alpha,c}$ and $\widehat{H}_n^{\alpha,c}$ are orientable and we endow them with the quotient orientation.

\begin{definition}
\label{def:MultivariableCassonLin}
Let $\beta$ be a $\mu$-colored $(c,c)$-braid and let $\alpha \in (0,\pi)^\mu$. 
If $\Delta_{\widehat{\beta}}(\omega_\varepsilon) \neq 0$ for all~$\omega_\varepsilon \in S(\alpha)$, then the \emph{multivariable Casson-Lin invariant} of $\beta$ at $\alpha$ is defined as the algebraic intersection number of $\widehat{\Gamma}_\beta^\alpha$ and $\widehat{\Lambda}_n^{\alpha,c}$ inside $\widehat{H}_n^{\alpha,c}$:
$$ h^c_\beta(\alpha):=\langle \widehat{\Lambda}_n^{\alpha,c}, \widehat{\Gamma}_\beta^\alpha \rangle_{\widehat{H}_n^{\alpha,c}}.$$
\end{definition}

Given a $\mu$-colored link $L$, we wish to define $h_L$ as $h^c_\beta$, where $\beta$ is any $(c,c)$-braid whose closure is $L$. In order to obtain a well defined link invariant, we must check that $h^c_\beta$ is invariant under the colored Markov moves which were described in Proposition~\ref{prop:MarkovColored}.

\subsection{Invariance under Markov moves}
\label{sub:Markov}

In this subsection, we prove that $h^c_\beta(\alpha)$ is invariant under the two colored Markov moves described in Proposition~\ref{prop:MarkovColored}. Since the key ideas of the proofs are present in~\cite[Theorem 1.8]{Lin} and~\cite[Proposition 4.2 and Proposition 4.3]{HarperSaveliev}, we place emphasis on the role of the colors, while referring to the original references for details.
\medbreak
The invariance under the first Markov move will follow promptly from the following lemma.

\begin{lemma}
\label{lem:Markov1Prelim}
Let $\alpha$ lie in $(0,\pi)^\mu$ and let $c$ and $c'$ be $\mu$-colorings. Let $\xi_1$ be a $(c,c)$-braid, let~$\xi_2$ be a $(c,c')$-braid and view $\xi_2^{-1}$ as a $(c',c)$-braid. The multivariable Casson-Lin invariants of the $(c,c)$-braid $\xi_1$ and the $(c',c')$-braid $\xi_2^{-1}\xi_1 \xi_2 $ are related by the following equation:
$$ h^{c}_{\xi_1}(\alpha)=h^{c'}_{\xi_2^{-1}\xi_1 \xi_2}(\alpha).$$
\end{lemma}
\begin{proof}
Recalling Subsection~\ref{sub:ColoredBraids}, the $(c,c')$-braid $\xi_2$ gives rise to an orientation preserving homeomorphism $\xi_2 \colon R_n^{\alpha,c} \to R_n^{\alpha,c'}$. One can then argue that it induces a well defined orientation preserving homeomorphism $\xi_2 \times \xi_2 \colon \widehat{H}_n^{\alpha,c}  \to \widehat{H}_n^{\alpha,c'}$. A short computation (using right actions) shows that $(\xi_2 \times \xi_2)(\widehat{\Lambda}_n^{\alpha,c})=\widehat{\Lambda}_n^{\alpha,c'}$ and $(\xi_2 \times \xi_2)(\widehat{\Gamma}^\alpha_{\xi_1})=\widehat{\Gamma}^\alpha_{\xi_2^{-1}\xi_1 \xi_2 }$.
 The result now follows promptly, see~\cite[first part of the proof of Theorem 1.8]{Lin} and~\cite[proof~of~Proposition~4.2]{HarperSaveliev}.
\end{proof}

Using Lemma~\ref{lem:Markov1Prelim}, we can prove the invariance under the first colored Markov move.

\begin{proposition}
\label{prop:InvarianceMarkov1}
The multivariable Casson-Lin invariant is preserved under the first colored Markov move.
\end{proposition}
\begin{proof}
Let $\alpha$ lies in $(0,\pi)^\mu$, let $\xi$ be a $(c,c')$-braid and let $\eta$ be a $(c',c)$-braid. Applying Lemma~\ref{lem:Markov1Prelim} to the $(c,c)$-braid $\xi \eta$ and to the $(c,c')$-braid $\xi$, we obtain $h^{c}_{\xi\eta}(\alpha)=h^{c'}_{\xi^{-1}(\xi \eta) \xi}(\alpha)=h^{c'}_{\eta\xi}(\alpha)$. This concludes the proof of the proposition.
\end{proof}

\begin{proposition}
\label{prop:InvarianceMarkov2}
The multivariable Casson-Lin invariant is preserved under the second colored Markov move.
\end{proposition}
\begin{proof}
Fix $\alpha \in (0,\pi)^\mu$, a $\mu$-coloring $c$ and a $(c,c)$-braid $\beta$. For the sake of conciseness, we write~$c'$ instead of $(c_1,\ldots,c_n,c_n)$ and we recall from Subsection~\ref{sub:ColoredBraids} that $i_{c_n} \colon B_c \to B_{c'}$ denotes the natural inclusion which adds a trivial strand of color $c_n$ to a given $(c,c)$-braid. Viewing the generator~$\sigma_n \in B_{n+1}$ as a $(c',c')$-braid, our goal is to show that $h_{\sigma_n i_{c_n}(\beta)}^{c'}(\alpha)=h^c_\beta(\alpha)$. Using Lemma~\ref{lem:Markov1Prelim}, this is equivalent to showing that
\begin{equation}
\label{eq:WantMarkov2}
h_{i_{c_n}(\beta)\sigma_n}^{c'}(\alpha)=h^c_\beta(\alpha).
\end{equation}
Recall (arranging transversality if necessary) that the right hand side of~(\ref{eq:WantMarkov2}) is defined as the algebraic intersection of the diagonal $\widehat{\Lambda}_n^{\alpha,c}$ with the graph $\widehat{\Gamma}^{\alpha}_{\beta}$. Similarly, the left hand side of~(\ref{eq:WantMarkov2}) is the algebraic intersection of $\widehat{\Lambda}_{n+1}^{\alpha,c'}$ with $\widehat{\Gamma}^{\alpha}_{i_{c_n}(\beta)\sigma_n}.$ In order to relate these various spaces, consider the embedding $g \colon R_n^{\alpha,c} \times R_n^{\alpha,c} \to R_{n+1}^{\alpha,c'} \times R_{n+1}^{\alpha,c'}$ defined by
$$ (X_1,\ldots,X_n,Y_1,\ldots,Y_n) \mapsto (X_1,\ldots,X_n,Y_n,Y_1,\ldots,Y_n,Y_n).$$
One can check that $g(H_n^{\alpha,c}) \subset H_{n+1}^{\alpha,c'}$ and that $g$ commutes with the conjugation, thus giving rise to an embedding $\widehat{g} \colon \widehat{H}_n^{\alpha,c} \to \widehat{H}_{n+1}^{\alpha,c'}$. It can also be checked that $\widehat{g}(\widehat{\Lambda}_n^{\alpha,c})$ is contained in~$\widehat{\Lambda}_{n+1}^{\alpha,c'}$,
that $\widehat{g}(\widehat{\Gamma}_{\beta}^{\alpha})$ is contained in~$\widehat{\Gamma}_{i_{c_n}(\beta)\sigma_n}^{\alpha}$ and that $\widehat{g}(\widehat{\Lambda}_n^{\alpha,c} \cap \widehat{\Gamma}_{\beta}^{\alpha} )$ is equal to $\widehat{\Lambda}_{n+1}^{\alpha,c'} \cap \widehat{\Gamma}_{i_{c_n}(\beta)\sigma_n}^{\alpha}.$ Given $X=(X_1,\ldots,X_n)$ in $\Lambda_n^{\alpha,c} \cap \Gamma_\beta^\alpha$, the same arguments as in~\cite[page 346]{Lin} show that the intersection number of $\widehat{\Lambda}_{n+1}^{\alpha,c'}$ and~$\widehat{\Gamma}_{i_{c_n}(\beta)\sigma_n}^\alpha$ at $\widehat{g}(X,X)$ is equal to the intersection number of~$\widehat{\Lambda}_n^{\alpha,c}$ and $\widehat{\Gamma}_{\beta}^{\alpha,c}$ at~$(X,X)$. This proves~\eqref{eq:WantMarkov2} and concludes the proof of the proposition.
\end{proof}

Using the invariance under Markov moves, we now define the main invariant of this paper.

\begin{definition}
\label{def:MultivariableCassonLinLink}
Let $L$ be a $\mu$-colored link and fix $\alpha = (\alpha_1, \ldots, \alpha_{\mu}) \in (0, \pi)^\mu$. Assume that~$\Delta_L(\omega_\varepsilon) \neq 0$ for all $\omega_\varepsilon \in S(\alpha)$. The \emph{multivariable Casson-Lin invariant} of $L$ at $\alpha$ is 
defined~as
$$ h_L(\alpha):=h_\beta^c(\alpha), $$
where $\beta$ is any $(c,c)$-braid whose closure is $L$.
\end{definition}
\begin{remark}
\label{rem:HypothesisDefinition}
The multivariable Casson-Lin invariant $h_L$ can be defined for a larger subset of~$(0, \pi)^\mu$. More precisely, one can define $h_L$ on the subset $D_L$ of those~$\alpha \in (0, \pi)^\mu$ such that none of the abelian representations $\rho_{\omega_\varepsilon}$ (recall Section \ref{sub:MultivarAlex}) is a limit of irreducible representations, for any $\omega_\varepsilon \in S(\alpha)$. 
Indeed, looking at the proof of Proposition~\ref{prop:Compact}, this assumption is sufficient to guarantee that $\widehat{\Lambda}_n^{\al,c} \cap \widehat{\Gamma}_\beta^{\al}$ is compact in $\widehat{H}_n^{\al,c}$. 
In particular, note that Corollary~\ref{cor:AlexReducibleSU2} implies that $D_L$ contains the set $\lbrace \alpha \in (0,\pi)^\mu \ | \ \Delta_L(\omega_\varepsilon) \neq 0 \text{ for all } \omega_\varepsilon \in S(\alpha)  \rbrace$.
\end{remark}

We conclude this section with a last remark that was used in Remark~\ref{rem:Formula1Colored} of the introduction.
\begin{remark}
\label{rem:Diagonal}
Every ordered link $L$ has an underlying oriented link which we also denote by~$L$.
Observe that given $\alpha \in (0,\pi)$, the following equality holds provided the multivariable and single variable Casson-Lin invariants are defined:
$$h_L(\alpha,\ldots,\alpha)=h_L(\alpha).$$
Indeed, in both cases we are counting the irreducible $\SU(2)$-representations of $\pi_1(X_L)$ for which all meridional traces are fixed to be $2\operatorname{cos}(\alpha)$.
Notice furthermore since $\Delta_L(t,\ldots,t)=(t-1)\Delta_L(t)$ for a link $L$ with $n>2$ components, the multivariable Casson-Lin invariant is defined at~$(\alpha,\ldots,\alpha)$ if and only if the single-variable Casson-Lin invariant is defined at $\alpha$.
\end{remark}

\section{The colored Gassner matrices and the Potential function}
\label{sec:ColoredGassner}
This section is organized as follows. In Subsection~\ref{sub:ColoredGassner}, we recall the definition of the colored Gassner matrices, in Subsection~\ref{sub:Long}, we review a result due to Long, in Subsection~\ref{sub:Potential}, we recall the definition of the multivariable potential function. Finally, in Subsection~\ref{sub:Technical}, we prove a technical result which shall frequently be used in Section~\ref{sec:CrossingChange}. 

\subsection{The colored Gassner matrices} 
\label{sub:ColoredGassner}
In this subsection, we recall the definition of the colored Gassner matrices and of the reduced colored Gassner matrices which are multivariable generalizations of the (reduced) Burau matrices. Although references include~\cite{KirkLivingstonWang, CimasoniTuraev, ConwayEstier}, our conventions are closest to those of~\cite{ConwayTwistedBurau}. 
\medbreak
Let $F_n$ be the free group on $x_1,\ldots,x_n$. Recall from Subsection~\ref{sub:ColoredBraids} that the braid group~$B_n$ acts on $F_n$ from the right and that each $n$-stranded braid $\beta$ gives rise to an automorphism of~$F_n$ which is also denoted by $\beta$. Given a $\mu$-coloring $c=(c_1,\ldots,c_n)$, consider the map
$$\psi_c \colon F_n \rightarrow \Z^\mu=\langle t_1,\dots,t_\mu \rangle$$
which sends each~$x_i$ to $t_{c_i}$ and extend it to a homomorphism $\psi_c \colon \Z[F_n] \to \Lambda_\mu$. For later use, observe that if $\beta$ is a $(c,c)$-braid, then~$\psi_c(x_i)$ is equal to $\psi_c(x_i \beta)$ and in fact, both are equal to $t_{c_i}$. Next, consider the element $\frac{\partial (x_i \beta)}{\partial x_j}$ of the group ring $\Z[F_n]$, where $\frac{\partial}{\partial x_j} \colon \Z[F_n] \to \Z[F_n]$ denotes the Fox derivative associated to $x_i$ (see e.g.~\cite[Chapter 11]{Lickorish}). 

The main definition of this section is the following.
\begin{definition}
\label{def:Unreduced}
The (unreduced) \emph{colored Gassner matrix} of an $n$-stranded $(c,c)$-braid~$\beta$ is defined as the $n \times n$ matrix $\mathcal{B}^c_t(\beta)$ whose $i,j$-coefficient is $\psi_c \left( \frac{\partial (x_i \beta)}{\partial x_j} \right)$.  
 \end{definition}

The notation $\mathcal{B}^c_t(\beta)$ is meant to indicate that the coefficients of the colored Gassner matrix lie in $\Lambda_\mu=\Z[t_1^{\pm 1},\ldots,t_\mu^{\pm 1}]$ (i.e. $t$ is used as a shorthand for $(t_1,\ldots,t_\mu))$. When $\mu=1$, the colored Gassner matrices recover the usual matrices for the Burau representation of $B_n$. We refer the interested reader to~\cite{KirkLivingstonWang, CimasoniTuraev} for more intrinsic approaches and to~\cite[Example 3.5]{ConwayTwistedBurau} for $(c,c')$-braids. Instead, we note that the unreduced colored Gassner matrix of the generator~$\sigma_i \in B_n$, viewed as a $(c,c)$-braid, is given by
 \begin{equation}
\label{eq:GassnerFoxFurther}
\mathcal{B}^c_t(\sigma_i) = I_{i-1} \oplus
\begin{pmatrix} 1-t_{c_{i}} & t_{c_{i}} \\ 1 & 0 \end{pmatrix}
 \oplus I_{n-i-1}.
\end{equation}
Here, note that viewing $\sigma_i$ as a $(c,c)$-braid necessitates that $c_i=c_{i+1}$.
Next, following~\cite{Birman} and~\cite[Section 3 (c)]{ConwayTwistedBurau}, we deal with the reduced colored Gassner matrices. Instead of working with the free generators~$x_1,x_2\dots,x_n$ of~$F_n$, we consider the elements~$g_1, g_2, \dots,g_n,$ defined by~$g_i:=x_1 x_2 \cdots x_i$. 
 As~$g_n$ is always fixed by the action of the braid group, 
the matrix whose~$i,j$-coefficient is $\psi_c \left( \frac{\partial (g_i \beta)}{\partial g_j} \right)$ is equal to $\widetilde{B}^c_t(\beta):=\left( \begin{smallmatrix}
 \overline{\mathcal{B}}^c_t(\beta) & v \\ 0 & 1  \end{smallmatrix} \right) $ for some column vector $v$. This motivates the following definition.

\begin{definition}
\label{def:Reduced}
The \emph{reduced colored Gassner matrix} of an $n$-stranded $(c,c)$-braid $\beta$ is defined as the size $n-1$ matrix $\overline{\mathcal{B}}_t^c(\beta)$ whose $i,j$-coefficient is $\psi_c \left(\frac{\partial (g_i \beta)}{\partial g_j} \right)$.  
 \end{definition} 

When $\mu=1$, the reduced colored Gassner matrices recover matrices for the reduced Burau representation of the braid group $B_n$. We once again avoid the more intrinsic definition of the reduced colored Gassner \emph{representation} which involves homology and covering spaces, but instead refer the interested reader to~\cite{KirkLivingstonWang, CimasoniTuraev} and \cite[Theorem 1.2]{ConwayEstier}. 

We conclude this subsection with a technical lemma which will be needed in Section~\ref{sec:CrossingChange}.

\begin{lemma}
\label{lem:FixedPointGassner}
For any $(c,c)$-braid $\beta$, the submodule of fixed points of the unreduced colored Gassner matrix $\mathcal{B}_{\omega}^c(\beta)$ is generated by $\widetilde{g}_n=\left( \begin{smallmatrix} 1 & \omega_{c_1} & \omega_{c_1}\omega_{c_2} & \ldots &\omega_{c_1}\cdots \omega_{c_{n-1}} \end{smallmatrix} \right)$ whenever $\omega \in \T^\mu$ satisfies both~$\omega_{c_1}\cdots \omega_{c_n} \neq 1$ and $\Delta_{\widehat{\beta}}(\omega) \neq~0$.
\end{lemma}
\begin{proof}
We first translate the statement into the $g_1,\ldots,g_n$ basis of $F_n$. Namely, computing the change of basis matrix between $\mathcal{B}_{\omega}^c(\beta)$ and $\widetilde{\mathcal{B}}_{\omega}^c(\beta)$ (see~\eqref{eq:ChangeOfBasis} below), the statement is equivalent to the claim that the submodule of fixed points of $\widetilde{\mathcal{B}}_{\omega}^c(\beta)$ is freely generated by $x=\left( \begin{smallmatrix} 0  & \ldots &0 &1 \end{smallmatrix} \right)$. Here, our convention is that the Burau matrices act on the right on row vectors.

Since~$x$ is fixed by~$\widetilde{\mathcal{B}}_{\omega}^c(\beta)$, we suppose that $w=\left( \begin{smallmatrix} w_1  & \ldots &w_{n-1} &w_n \end{smallmatrix} \right)$ is fixed by $\widetilde{\mathcal{B}}_{\omega}^c(\beta)$ and wish to show that~$w$ lies in the span of $x$. Using Definition~\ref{def:Reduced}, the assumption on $w$ implies that the reduced colored Gassner matrix $\overline{\mathcal{B}}_{\omega}^c(\beta)$ must fix $w':=\left( \begin{smallmatrix} w_1  & \ldots &w_{n-1} \end{smallmatrix} \right)$ (recall that we are using right actions). 
This implies that $(\overline{\mathcal{B}}_{\omega}^c(\beta)-I_{n-1})w'=0$, and we therefore deduce that~$\det(\overline{\mathcal{B}}_{\omega}^c(\beta)-I_{n-1})=0$.
Using the relation between the multivariable Alexander polynomial and the colored Gassner representation (see e.g. Remark~\ref{rem:GassnerConway} below), we infer that~$(\omega_{c_1}\cdots \omega_{c_n}-1)\Delta_{\widehat{\beta}}(\omega)=0.$
This contradicts our assumptions on $\omega$ and concludes the proof of the lemma.
\end{proof}

\subsection{A result due to Long}
\label{sub:Long}
The goal of this subsection is to recall a theorem due to Long~\cite[Theorem 2.4]{Long}. In order to state this result, we use Long's conventions regarding automorphisms of the free group. As we observed in Remark~\ref{rem:ConventionsBraids}, these conventions match ours when dealing with the action of the braid group $B_n$ on $\SU(2)^n$.

\medbreak
For an automorphism $\theta \colon F_n \to F_n$ of the free group, consider the diffeomorphism $\theta^* \colon R(F_n) \to R(F_n), \rho \mapsto \rho \circ \theta^{-1}$. Picking free generators $x_1,\ldots,x_n$ of $F_n$ and identifying~$R(F_n)$ with $\SU(2)^n$, the diffeomorphism $\theta^*$ is described as $\theta^*(X_1,\ldots,X_n)=(\theta^{-1}X_1,\ldots,\theta^{-1}X_n).$ The assignment $\theta \mapsto \theta^*$ gives rise to a homomorphism $ \operatorname{Aut}(F_n) \to \operatorname{Diff}(\SU(2)^n)$. Fixing a subgroup~$H$ of $\operatorname{Aut}(F_n)$, the restriction of this assignment produces a homomorphism $H \to \operatorname{Diff}(\SU(2)^n)$. To get a linear representation of $H$, pick a function $f \colon (0,\pi)^\mu \to \SU(2)^n$ such that $h^*f(\alpha)=f (\alpha)$ for every $\alpha=(\alpha_1,\ldots,\alpha_\mu)$ in~$(0,\pi)^\mu$ and for every $h$ in $H$, and set
\begin{align*}
\rho_\alpha \colon &H \to \operatorname{Aut}(T_{f(\alpha)}\SU(2)^n) \\
& h \mapsto T_{f(\alpha)}(h^*).
\end{align*}
The fact that $\rho_\alpha$ is a representation follows from the chain rule~\cite[Theorem 2.3]{Long}. We now restrict to the colored braid group $H=B_c$ and, for $\theta \in (0,\pi)$, we set $\mathbf{e}^{i\theta}:=\left( \begin{smallmatrix} e^{i\theta} & 0 \\ 0 & e^{-i\theta} \end{smallmatrix}  \right)$. 
Recalling the notations and conventions discussed in Remark~\ref{rem:ConventionsBraids}, we observe  that for any $\mu$-tuple $\varepsilon=(\varepsilon_{1}, \ldots, \varepsilon_{\mu}) \in \lbrace \pm1\rbrace ^\mu$, the action of a $(c,c)$-braid $\beta$ on the $n$-tuple of matrices $f(\alpha):= (\mathbf{e}^{\varepsilon_{c_1}i\alpha_{c_1}},\ldots,\mathbf{e}^{\varepsilon_{c_n}i\alpha_{c_n}})$ satisfies $ f(\alpha)\beta =f(\alpha)$.
As a consequence, we obtain representations~$\rho_\alpha$ of $B_c$. Long~\cite[Theorem 2.4]{Long} proves the following result: 

\begin{proposition}
\label{prop:Long}
Let $c=(c_1,\ldots,c_n)$ be a $\mu$-coloring, let $\alpha=(\alpha_1,\ldots,\alpha_\mu)$ lie in $(0,\pi)^\mu$ and let $\varepsilon = (\varepsilon_1, \ldots, \varepsilon_\mu) \in \lbrace \pm1\rbrace^{\mu}$.
If one sets $\mathbf{a}_\varepsilon=(\mathbf{e}^{\varepsilon_{c_1} i \alpha_{c_1}},\ldots,\mathbf{e}^{\varepsilon_{c_n}i \alpha_{c_n}})$, then the representation $ \rho_{\alpha} \colon B_c \to \operatorname{Aut}(T_{\mathbf{a}_\varepsilon}\SU(2)^n)$ is a direct sum of a permutation representation with the colored Gassner matrix evaluated at $\omega_\varepsilon=~(e^{\varepsilon_1 2i\alpha_1}, \ldots, e^{\varepsilon_\mu 2i \alpha_\mu})$.
\end{proposition}

Note that Long proved this result for $\mu=1$~\cite[Theorem 2.4]{Long} and $\mu=n$~\cite[Theorem 2.5]{Long} but his proof goes through for arbitrary colored braid groups. In order to make some further remarks on Proposition~\ref{prop:Long}, we recall some known facts regarding the field $\mathbb{H}$ of quaternions.

\begin{remark}
\label{rem:Quaternions}
We think of $\mathbb{H}$ using the isomorphisms $\mathbb{H} \cong \mC \oplus \textbf{j}\mC \cong (\mR \oplus i\mR) \oplus (j\mR \oplus k\mR)$ and recall that a quaternion is \emph{pure} if its real part is zero. Matrices in $\SU(2)$ can be identified with unit quaternions via the map which sends $\bsm a&b\\-\bar{b}&\bar{a} \esm$ to $a+ \textbf{j} b$, for any $a,b \in \C$ which satisfy~$|a|^2+|b|^2=1$. On the Lie algebra level, for $r\in \R$ and $z \in \C$, matrices $\bsm ir&z\\-\bar{z}&-ir \esm$ in $\sudeux$ correspond to quaternions $ir+\mathbf{j}z$, and in particular $\sudeux$ splits as $i\mR \oplus \mathbf{j}\mC$. 
\end{remark}

Using Remark~\ref{rem:Quaternions} and working with the notations of Proposition~\ref{prop:Long}, Long's result shows that the restriction of the differential of $\beta \colon \SU(2)^n \to \SU(2)^n$ at $\mathbf{a}_\varepsilon$ to the complex summand of $\mathfrak{su}_2$ is $\mathcal{B}^c_{\omega_\varepsilon}(\beta)$ (i.e. the colored Gassner matrix evaluated~at~$\omega_\varepsilon$). In Section~\ref{sec:CrossingChange} however, we shall study the restriction of $\beta$ to $R_n^{\alpha,c}$. Since this latter space is homeomorphic to a product of 2-spheres $\mathbb{S}_{\alpha_j}$ which consist of those matrices with trace~$2\operatorname{cos}(\alpha_j)$, we adapt some observations from~\cite[Section 2.3]{HeusenerKroll} to the multivariable case.

\begin{remark}
\label{rem:MultiplicationByJ}
Matrices in $\SU(2) \setminus \pm I$ can be identified with pairs $(\theta,Q)$, where $\theta \in (0,\pi)$ and $Q=xi+yj+zk$ is a pure quaternion of norm~$1$. More explicitly, the quaternion $\cos(\theta)+\sin(\theta)Q$ associated to a pair $(\theta,Q)$ corresponds to the $\SU(2)$-matrix
$$X = \bsm \cos (\theta) + i x \sin (\theta) & (y+iz) \sin (\theta)\\ (-y+iz)\sin(\theta) & \cos(\theta)-ix\sin(\theta)\esm.$$
On the Lie algebra level, using $\mathbf{j}^2=-1$ and the identification of $\mathfrak{su}_2$ with $i\R \oplus \mathbf{j}\C$, multiplication by $-\mathbf{j}$ picks out the complex component $z$ of the matrix $\bsm 0&z\\-\bar{z}& 0 \esm \in \mathfrak{su}_2$.
 In particular, since~$R_n^{\alpha,c}$ is a product of~$\mathbb{S}_{\theta}$, Proposition~\ref{prop:Long} implies that the following diagram commutes:
$$ \xymatrix{ 
T_{\mathbf{a}}R_n^{\alpha,c} \ar[rr]^{(-\mathbf{j},\ldots,-\mathbf{j})} \ar[d]^{T_{\mathbf{a}}\beta}&& \C^n \ar[d]^{\mathcal{B}_\omega^c(\beta)} \\
T_{\mathbf{a}}R_n^{\alpha,c} \ar[rr]^{(-\mathbf{j},\ldots,-\mathbf{j})}&& \C^n.
} $$
\end{remark}

On the topological level, it is helpful to think of $\SU(2)$ as foliated by the spheres $\mathbb{S}_{\theta}$: indeed the quaternionic expression $\cos (\theta) + \sin (\theta) Q$, specifies a $2$-sphere $\mathbb{S}_{\theta}$ and a position $Q$ on this sphere. On the Lie algebra level, the complex lines are tangent to the leaves $\mS_{\theta}$ and the real lines are tangent to the transverse directions. 

\subsection{The potential function}
\label{sub:Potential}
In this section, we review some facts about the multivariable potential function. References include~\cite{ConwayJohn, Hartley, CimasoniPotential, Murakami}.
\medbreak
As we recalled in Section~\ref{sec:AlexanderAbelian}, the multivariable Alexander polynomial $\Delta_L$ of a $\mu$-colored link~$L$ is only well defined up to multiplication by units of $\Lambda_\mu$. The \emph{multivariable potential function} of a $\mu$-colored link~$L$ is a rational function $\nabla_L(t_1,\ldots,t_\mu)$ which satisfies 
\begin{equation}
\label{eq:ConwayVSAlexander}
 \nabla_L(t_1,\ldots,t_\mu)\stackrel{.}{=}
\begin{cases}
\frac{1}{t_1-t_1^{-1}}\Delta_L(t_1^2) & \mbox{if }  \mu=1, \\
\Delta_L(t_1^2,\ldots,t_\mu^2)   & \mbox{if } \mu>1. \\ 
 \end{cases} 
 \end{equation}
In this paper, we use a construction of the potential function which arises from the reduced colored Gassner representation~\cite[Theorem 1.1]{ConwayEstier}. The next remark briefly recalls this result.

\begin{remark}
\label{rem:GassnerConway}
Any $(c,c)$-braid $\beta$ can be decomposed into a product $\prod_{j=1}^{m} \sigma_{i_j}^{\varepsilon_{j}}$, where $\sigma_{i_j}$ denotes the $i_j$-th generator of the braid group (viewed as an appropriately colored braid) and each $\varepsilon_j$ is equal to $\pm 1$. For each $j$, use~$b_j$ to denote the color of the over-crossing strand in the generator $\sigma_{i_j}^{\varepsilon_{j}}$ and consider the Laurent monomial 
$$ \langle \beta \rangle:=\prod_{j=1}^m t_{b_j}^{-\varepsilon_j}.$$ 
Define $g \colon \Lambda_\mu \to \Lambda_\mu$ by extending $\Z$-linearly the group endomorphism of $\Z^\mu=\langle t_1,\ldots,t_\mu \rangle$ which sends $t_i$ to $t_i^2$. Given an $n$-stranded $\mu$-colored $(c,c)$-braid $\beta$,~\cite[Theorem 1.1]{ConwayEstier} shows that the multivariable potential function of the closure $\widehat{\beta}$ can be described as:
\begin{equation}
\label{eq:PotentialConway}
 \nabla_{\widehat{\beta}}(t_1,\ldots,t_\mu)=
 (-1)^{n+1} \cdot \frac{1}{t_{c_1} \cdots  t_{c_n} - t^{-1}_{c_1} \cdots  t^{-1}_{c_n}} \cdot \langle \beta \rangle \cdot g(\det(\overline{\mathcal{B}}^c_t(\beta) - I_{n-1})).
 \end{equation}
Note that in~\cite{ConwayEstier}, the matrices $\overline{\mathcal{B}}_t^c$ are the transposes of the ones used here (and in particular~\cite{ConwayEstier} deals with \emph{anti}-representations). Naturally, this does not affect~(\ref{eq:PotentialConway}).
\end{remark}

In the one-variable case, some care is needed with the terminology.

\begin{remark}
\label{rem:TerminologyPotential}
The expression $D_L(t):=\nabla(t)(t-t^{-1})$ is usually referred to as the \emph{Alexander-Conway polynomial} of $L$ and satisfies $D_L(t)\stackrel{.}{=}\Delta_L(t^2)$. On the other hand, some authors call~$D_L(\sqrt{t})$ the \emph{Conway-normalized Alexander polynomial}. For instance Heusener-Kroll use~$\Delta_L(t)$ to denote the Conway-normalized Alexander polynomial~\cite[Section 2.1]{HeusenerKroll}. These distinctions do matter: for a knot $K$ and $\omega \in S^1$, it is known that $D_K(\sqrt{\omega})$ is real, while this statement is incorrect for $\nabla_K$ and makes no sense for $\Delta_K$ (because of the indeterminacy).
\end{remark}

We conclude with some remarks on evaluations of $\nabla_L$ at elements of $\mathbb{T}^\mu=(S^1)^\mu$.

\begin{remark}
\label{rem:PotentialEvaluations}
The potential function $\nabla_L$ of an $n$-component $\mu$-colored link is known to be~$(-1)^n$-symmetric~\cite[Proposition 1]{CimasoniPotential}. Thus, for $\omega \in \T^\mu$, the evaluation $\nabla_L(\omega)$ need not be real. In fact, for $\omega \in \T^\mu$, the aforementioned symmetry property yields $ \overline{\nabla_L(\omega)}=\nabla_L(\overline{\omega})=(-1)^n \nabla_L(\omega)$, and therefore~$\nabla_L(\omega)$ belongs to~$\R$ (resp. $i\R$) if $n$ is even (resp. odd). In particular, if two $\mu$-colored links differ by a crossing change within a sublink, then the quotient of the two potential functions evaluated at $\omega \in \T^\mu$ is real (assuming the quotient is defined).
\end{remark}

\subsection{A technical proposition}
\label{sub:Technical}

The aim of this section is to prove the following multivariable generalization of~\cite[Lemma 4.4]{HeusenerKroll}. This result will be frequently used in Section~\ref{sec:CrossingChange}. 
Due to the technical nature of this result and its proof, we suggest the reader skip this subsection upon his first reading.

\begin{proposition}
\label{prop:TechnicalLemma}
Let $c$ be a $\mu$-coloring such that $c_1=c_2$ and let $\omega \in \T^\mu$.
Given a $(c,c)$-braid~$\beta$, use $\left( \begin{smallmatrix} A(\omega) & B(\omega)  \\ C(\omega) & D(\omega) \end{smallmatrix} \right)$ to denote the unreduced colored Gassner matrix of $\beta$ evaluated at~$\omega \in \mathbb{T}^\mu$, where $D(\omega)$ is a size $n-2$ square matrix. If we assume that $\omega_{c_1}^2 \neq 1$ and~$\nabla_{\widehat{\beta}}(\omega)\neq 0$, then~$\operatorname{det}(D(\omega)-I_{n-2}) \neq 0$.
\end{proposition}

The proof of Proposition~\ref{prop:TechnicalLemma} follows the strategy of~\cite[Lemma 4.4]{HeusenerKroll}. However several of the preliminary results require some additional work in the multivariable case. 

\subsubsection{Rows and columns of $\mathcal{B}_t^c(\beta)$}

We temporarily adopt the following conventions: given a matrix $\Psi$, we write $\Psi^i_j$ for the $(i,j)$-coefficient of~$\Psi$, instead of the more standard~$\Psi_{ij}$; apart if mentioned otherwise, $I$ denotes any identity matrix, regardless of its size.

The following lemma (which generalizes well known results for the Burau representation) describes the result of summing (linear combinations of) the rows and columns of the unreduced colored Gassner matrices.

\begin{lemma}
\label{lem:Fox}
Given a $(c,c)$-braid $\beta$, the rows and columns of the colored Gassner matrix satisfy the following properties:
\begin{enumerate}
\item For each $i$, one has $\sum_{j=1}^n (t_{c_j}-1) \mathcal{B}^c_t(\beta)_j^i=t_{c_i}-1$.
\item For each $j$, one has $\sum_{i=1}^n t_{c_1}\cdots t_{c_{i-1}} \mathcal{B}^c_t(\beta)^i_j=t_{c_1}\cdots t_{c_{j-1}}$.
\end{enumerate}

\end{lemma}
\begin{proof}
In order to prove both of these identities, we recall the so-called``fundamental lemma of Fox calculus"~\cite[Proposition 9.8, part c)]{BurdeZieschang}. Given a word $w$ in the free group $F_n$ on $x_1,\ldots, x_n$, the following identity holds in the group ring $\Z[F_n]$:
\begin{equation}
\label{eq:FundamentalFormula}
\sum_{j=1}^n \frac{\partial w}{\partial x_j}(x_j-1)=w-1.
\end{equation}
The first identity now follows by considering the word $w=x_i \beta$, applying $\psi_c$ to both sides of~(\ref{eq:FundamentalFormula}) and recalling that for a $(c,c)$-braid, both $\psi_c(x_i)$ and $\psi_c(x_i \beta)$ are equal to~$t_{c_i}$. To obtain the second formula, apply the Fox derivative $\frac{\partial}{\partial x_j}$ to both sides of the equality~$(x_1\cdots x_n)\beta=x_1 \cdots x_n$ and use the derivation property repeatedly.
\end{proof}

Taking advantage of our unconventional notation, observe that the $i$-th column of $\mathcal{B}^c_t(\beta)$ can be written as $\mathcal{B}^c_t(\beta)_i$, while the $i$-th line of $\mathcal{B}^c_t(\beta)$ can be written as $\mathcal{B}^c_t(\beta)^i$. In particular, Lemma~\ref{lem:Fox} implies that
$$ \sum_{i=1}^n (t_{c_i}-1)\mathcal{B}^c_t(\beta)_i=(T-1), \ \ \ \ \ \sum_{i=1}^n (t_{c_1}\cdots t_{c_{i-1}})\mathcal{B}^c_t(\beta)^i=v,$$
where $T-1$ denotes the size $n$ column vector whose $i$-th component is $t_{c_i}-1$ and $v$ denotes the size $n$ row vector whose $j$-th component is $t_{c_1}\cdots t_{c_{j-1}}$. 

\begin{example}
\label{ex:RealityCheck}
If $c=(1,\ldots,1)$, the first point of Lemma~\ref{lem:Fox} implies the following known fact: the sum of the coefficients within any line of the Burau matrix is $1$ (i.e. the Burau matrix is a ``right stochastic matrix"). For $\sigma_1^2 \in B_{(1,2)}=P_2$, the Gassner matrix is given by 
$$\mathcal{B}^{(1,2)}_{t_1,t_2}(\sigma_1^2)= \begin{pmatrix} 1-t_1+t_1t_2 & t_1(1-t_1) \\
1-t_2 & t_1
 \end{pmatrix}.$$
Now Lemma~\ref{lem:Fox} states in particular that $(1-t_1)(1-t_1+t_1t_2)+(1-t_2)(t_1(1-t_1)=1-t_1$ and $t_1(1-t_1)+t_1t_1=t_1$ which can indeed be verified.
\end{example}

\subsubsection{Computations with minors}

Given a square matrix $\Psi$ of size $n$, we use $\Psi_{i,j}$ to denote the size $(n-1)$ matrix obtained from $\Psi$ by deleting its $i$-th row and $j$-th column. We also use $\mathcal{B}^c_t(\beta,l,m)$ to denote $\det((\mathcal{B}^c_t(\beta)-I)_{l,m})$ (the notation $c_{l,m}^\beta$ is used in~\cite[Section 2.4]{HeusenerKroll}).

The following lemma is a multivariable generalization of~\cite[Lemma 2.2, part 1)]{HeusenerKroll}.

\begin{lemma}
\label{lem:HeusenerKrollLemme221}
Let $c$ be a $\mu$-coloring. Given an $n$-stranded $(c,c)$-braid $\beta$ and positive integers $1 \leq l,l',m,m' \leq n$, the following equality holds in $\Lambda_\mu$:
\begin{equation}
\label{eq:221}
(t_{c_{m'}}-1)(t_{c_1}\cdots t_{c_{l'-1}})\mathcal{B}^c_t(\beta,l,m)=(-1)^{m+m'+l+l'} (t_{c_m}-1) (t_{c_1}\cdots t_{c_{l-1}})\mathcal{B}^c_t(\beta,l',m').
\end{equation}
\end{lemma}
\begin{proof}
To prove the lemma, it suffices to prove~(\ref{eq:221}) when $l=l'$ and when $m=m'$. We therefore start by assuming that $l=l'$ and claim that
$$ (t_{c_{m'}}-1)\mathcal{B}^c_t(\beta,l,m)=(-1)^{m+m'} (t_{c_m}-1)\mathcal{B}^c_t(\beta,l,m').$$
Recall that $T-1$ denotes the size $n$ column vector whose $i$-th component is $t_{c_i}-1$ and assume that~$i$ differs from $m$. Using the first point of Lemma~\ref{lem:Fox}, a short computation shows that 
\begin{equation}
\label{eq:ColumnTrick}
(t_{c_i}-1)(\mathcal{B}^c_t(\beta)-I)_i=\left((T-1)-\sum_{k \neq i}(t_{c_k}-1)\mathcal{B}^c_t(\beta)_k\right)-(t_{c_i}-1)I_i=-\sum_{k\neq i}(t_{c_k}-1)(\mathcal{B}^c_t(\beta)-I)_k.
\end{equation}
We now use this identity to compute the determinant of the matrix $(\mathcal{B}^c_t(\beta)-I)_{l,m}$ obtained by removing the $l$-th row and the $m$-th column from~$\mathcal{B}^c_t(\beta)-I$. Multipliying the $i$-th column of~$\mathcal{B}^c_t(\beta)-I$ by $t_{c_i}-1$, using~(\ref{eq:ColumnTrick}), removing the $m$-th column of~$\mathcal{B}^c_t(\beta)-I$, invoking the multilinearity of the determinant and switching back the $i$-th column to its original place (this produces a sign $(-1)^{i+m-1}$ since we now have one column less), we obtain
$$ (t_{c_i}-1)\mathcal{B}^c_t(\beta,l,m)=(t_{c_m}-1)(-1)^{i+m}\mathcal{B}^c_t(\beta,l,i).$$
The claim now follows by taking $i=m'$. To prove~(\ref{eq:221}) for $m=m'$, one uses the second point of Lemma~\ref{lem:Fox} and follows the exact same steps as above with rows instead of columns. This concludes the proof of the lemma.
\end{proof}

Lemmas~\ref{lem:Fox} and~\ref{lem:HeusenerKrollLemme221} involve the colored Gassner matrices in the basis arising from the choice of generators $x_1,\ldots,x_n$ of the free group $F_n$. In order to work with the reduced colored Gassner matrices, we need the corresponding statements for the basis $g_1,\ldots,g_n$ of $F_n$. 

\begin{remark}
\label{rem:giBasis}
Use $\widetilde{\mathcal{B}}^c_t(\beta)$ to denote the unreduced colored Gassner matrix in the basis arising from the choice of generators $g_1,\ldots,g_n$ of the free group. Just as for the matrix~$\mathcal{B}^c_t(\beta)$, we set $\widetilde{\mathcal{B}}^c_t(\beta,l,m):=\det((\widetilde{\mathcal{B}}^c_t(\beta)-I)_{l,m})$. Using these notations, the following formula holds:
\begin{equation}
\label{eq:MinorsgiBasis}
-(t_{c_1}\cdots t_{c_{n-1}}-1)\widetilde{\mathcal{B}}^c_t(\beta,n,n-1)=(t_{c_1}\cdots t_{c_{n}}-1)\widetilde{\mathcal{B}}^c_t(\beta,n,n).
\end{equation}
The proof of~(\ref{eq:MinorsgiBasis}) is entirely analogous to the proof of Lemma~\ref{lem:HeusenerKrollLemme221}: it suffices to use the equality $\psi_c(g_i)=t_{c_1}\cdots t_{c_i}$ instead of $\psi_c(x_i)=t_{c_i}$. Finally note that~(\ref{eq:MinorsgiBasis}) can be rewritten using the reduced colored Gassner representation. Indeed, using Definition~\ref{def:Reduced}, we have~$\widetilde{\mathcal{B}}^c_t(\beta,n,n)=\det(\overline{\mathcal{B}}^c_t(\beta)-I_{n-1}),$ where $I_{n-1}$ denotes the size $n-1$ identity matrix.
\end{remark}

We now relate  $\det(\overline{\mathcal{B}}^c_t(\beta)-I_{n-1})$ and $\mathcal{B}^c_t(\beta,1,1)$, generalizing~\cite[Lemma 2.2.2]{HeusenerKroll}.

\begin{proposition}
\label{prop:HeusenerKrollLemma222}
Given an $n$-stranded $(c,c)$-braid~$\beta$, the following equation holds:
$$    \frac{t_{c_1}\cdots t_{c_n}-1}{t_{c_1}-1}   \mathcal{B}^c_t(\beta,1,1)=   \det(\overline{\mathcal{B}}^c_t(\beta)-I_{n-1}).  $$ 
\end{proposition}
\begin{proof}
It suffices to show that $\mathcal{B}^c_t(\beta,n,n)(t_{c_1}\cdots t_{c_n}-1)=t_{c_1}\cdots t_{c_{n-1}}(t_{c_n}-1)\widetilde{\mathcal{B}}^c_t(\beta,n,n)$: the conclusion will then follow from Remark~\ref{rem:giBasis} and Lemma~\ref{lem:HeusenerKrollLemme221} which imply respectively that~$\widetilde{\mathcal{B}}^c_t(\beta,n,n)=\det(\overline{\mathcal{B}}^c_t(\beta)-I_{n-1})$ and $(t_{c_{n}}-1)(t_{c_1}\cdots t_{c_{n-1}})\mathcal{B}^c_t(\beta,1,1)=(t_{c_1}-1)\mathcal{B}^c_t(\beta,n,n).$ A computation involving Fox calculus shows that the change of basis matrix from $\mathcal{B}^c_t(\beta)$ to~$\widetilde{\mathcal{B}}^c_t(\beta)$ is given by 
\begin{equation}
\label{eq:ChangeOfBasis}
P_n= \begin{pmatrix}
1 & 0 &0 & \cdots & 0 \\ 
1 & t_{c_1} & 0  & \cdots & 0 \\
1  &t_{c_1} & t_{c_1}t_{c_2} &  \ddots & \vdots \\ 
\vdots  & \vdots & \ddots& \ddots &0 \\ 
1 & t_{c_1} & \cdots &  t_{c_1} \cdots  t_{c_{n-2}} & t_{c_1} \cdots  t_{c_{n-2}}t_{c_{n-1}}
\end{pmatrix}.
\end{equation}
Given a matrix $M$, recall that we use $M_{n,n}$ to denote the matrix obtained by deleting the $n$-th row and $n$-th column of $M$. Until the end of this proof, we use $I$ to denote the size $n$ identity matrix. With this notation, observe that $\widetilde{\mathcal{B}}^c_t(\beta,n,n)=\det((P_n \mathcal{B}^c_t(\beta) P_n^{-1}-I)_{n,n})$. A~tedious computation now shows that 
$$\det((P_n \mathcal{B}^c_t(\beta) P_n^{-1}-I)_{n,n})=\det(P_{n-1}(\mathcal{B}_t^c(\beta))_{n,n}P_{n-1}^{-1}-I_{n-1})-\widetilde{\mathcal{B}}^c_t(\beta,n,n-1).$$
 Using the definition of $\widetilde{\mathcal{B}}^c_t(\beta)$ and the fact that the determinant is invariant under conjugation, this can be rewritten as $\widetilde{\mathcal{B}}^c_t(\beta,n,n)=\mathcal{B}^c_t(\beta,n,n)-\widetilde{\mathcal{B}}^c_t(\beta,n,n-1)$.
The conclusion then follows by using~(\ref{eq:MinorsgiBasis}). This concludes the proof of the proposition.
\end{proof}

\subsubsection{Relation to the potential function}

As in Remark~\ref{rem:GassnerConway},~$g \colon \Lambda_\mu \to \Lambda_\mu$ is defined by extending $\Z$-linearly the group endomorphism of $\Z^\mu=\langle t_1,\ldots,t_\mu \rangle$ which sends~$t_i$ to $t_i^2$. The following lemma expresses $\nabla_{\widehat{\beta}}$ using a minor of the \emph{unreduced} colored Gassner matrix.

\begin{lemma}
\label{lem:PotentialC11}
Given a $\mu$-colored $n$-stranded $(c,c)$-braid $\beta$, we have
$$(t_{c_1}^2-1) \nabla_{\widehat{\beta}}(t_1,\ldots,t_\mu)= (-1)^{n+1} \langle \beta \rangle \cdot t_{c_1}\cdots t_{c_n} \cdot g(\mathcal{B}^c_t(\beta,1,1)),$$
where $\langle \beta \rangle$ is the Laurent monomial described in Remark~\ref{rem:GassnerConway}. 
\end{lemma}
\begin{proof}
Using successively Remark~\ref{rem:GassnerConway} and  Proposition~\ref{prop:HeusenerKrollLemma222}, we obtain
\begin{align}
\label{eq:GassnerConway}
\nabla_{\widehat{\beta}}(t_1,\ldots,t_\mu)
&= \frac{(-1)^{n+1} \langle \beta \rangle}{t_{c_1}\cdots t_{c_n}-t_{c_1}^{-1}\cdots t_{c_n}^{-1}} g( \det(\overline{\mathcal{B}}^c_t(\beta)-I_{n-1}))\\
&=  \frac{(-1)^{n+1} \langle \beta \rangle}{t_{c_1}\cdots t_{c_n}-t_{c_1}^{-1}\cdots t_{c_n}^{-1}} \frac{t_{c_1}^2\cdots t_{c_n}^2-1}{t_{c_1}^2-1} g(\mathcal{B}^c_t(\beta,1,1)). \nonumber
\end{align}
This concludes the proof of the lemma.
\end{proof}

We need one last lemma in order to prove Proposition~\ref{prop:HeusenerKroll46}, namely we require a multivariable generalization of~\cite[Lemma 2.2, part 3]{HeusenerKroll}. For that purpose, we write the (unreduced) colored Gassner matrix of $\beta$ as $\left( \begin{smallmatrix} A & B \\ C & D \end{smallmatrix} \right)$, where $D$ is a square matrix of size $n-2$.

\begin{lemma}
\label{lem:HeusenerKrollLemma223}
Let $c$ be a $\mu$-coloring with $c_1=c_2$ and let $\beta$ be an $n$-stranded $(c,c)$-braid. If the generator $\sigma_1 \in B_n$ is viewed as a $(c,c)$-braid, then the following equality holds in $\Lambda_\mu$:
\begin{equation}
\label{eq:223}
\mathcal{B}^c_t(\sigma_1^2\beta,1,1)=t_{c_1}^2\mathcal{B}^c_t(\beta,1,1)+(t_{c_1}-1)\det(D-I_{n-2}).
\end{equation}
\end{lemma}
\begin{proof}
First, a short computation shows that $\mathcal{B}^c_t(\sigma_1^2)=\left( \begin{smallmatrix} 1-t_{c_1}+t_{c_1}^2 & t_{c_1}(1-t_{c_1}) \\ 1-t_{c_1} & t_{c_1} \end{smallmatrix} \right) \oplus I_{n-2}$, see also Example~\ref{ex:RealityCheck}. Next, recalling that we decomposed the colored Gassner matrix of $\beta$ as $\left( \begin{smallmatrix} A & B \\ C & D \end{smallmatrix} \right) $, we write the matrix~$A$ as $\left( \begin{smallmatrix} a_{11} & a_{12} \\ a_{21} & a_{22} \end{smallmatrix} \right)$, the matrix $B$ as~$\left( \begin{smallmatrix} b_1 \\ b_2  \end{smallmatrix} \right)$ where each $b_i$ is a size $n-2$ row vector and the matrix $C$ as $(c_1,c_2)$, where each $c_i$ is a size $n-2$ column vector. As~(\ref{eq:223}) does not involve the first lines and columns of the aforementioned matrices, we are reduced to proving
\begin{equation}
\label{eq:Goal}
\det \begin{pmatrix} (1-t_{c_1})a_{12}+t_{c_1} a_{22}-1 & (1-t_{c_1})b_1+t_{c_1}b_2 \\ c_2 & D-I  \end{pmatrix}
= t_{c_1}^2 \det \begin{pmatrix}
a_{22}-1 & b_2 \\ c_2 & D-I
\end{pmatrix}+ (t_{c_1}-1)\det(D-I),
\end{equation}
where we use $I$ as a shorthand for the identity matrix $I_{n-2}$. Expanding the left hand side of~(\ref{eq:Goal}) along the first row, we obtain
\begin{equation}
\label{eq:LHS}
((1-t_{c_1})a_{12}+t_{c_1}a_{22}-1)\det(D-I)
+\sum_{j=2}^{n-1} (-1)^{1+j}((1-t_{c_1})b_1^j+t_{c_1}b_2^j)\det(L_j),
\end{equation}
where, for $j$ greater than one, $L_j$ denotes the size $n-2$ square matrix obtained from $\left( \begin{smallmatrix} c_2 & D-I \end{smallmatrix} \right)$ by removing the $j$-th column. Keeping these notations in mind and expanding the determinant in the right hand side of~(\ref{eq:Goal}) along its first line, we obtain
\begin{equation}
\label{eq:RHS}
t_{c_1}^2 \left( (a_{22}-1)\det(D-I)+\sum_{j=2}^{n-1} (-1)^{1+j}b_2^j \det(L_j) \right) +(t_{c_1}-1) \det(D-I).
\end{equation}
Substracting~(\ref{eq:RHS}) from~(\ref{eq:LHS}) and simplifying the extraneous $t_{c_1}-1$ factors, we see that~(\ref{eq:Goal}) in fact reduces to proving the equation $-\mathcal{B}^c_t(\beta,2,1)-t_{c_1}\mathcal{B}^c_t(\beta,1,1)=0$. Since the latter equation holds thanks to Lemma~\ref{lem:HeusenerKrollLemme221}, the proof is concluded.
\end{proof}

\subsubsection{Conclusion of the proof}

\begin{proof}[Proof of Propostion~\ref{prop:TechnicalLemma}]
Let $\omega \in \T^\mu$ be such that~$\nabla_{\widehat{\beta}}(\omega)$ is non-zero. Our goal is to show that~$\operatorname{det}(D(\omega)-I_{n-2})$ is non-zero. Use $\mathcal{B}^c_\omega(\beta)$ to denote the unreduced colored Gassner matrix of $\beta$ evaluated at $\omega$. Assume by way of contradiction that $\det(D(\omega)-I_{n-2})$ vanishes. Using Lemma~\ref{lem:HeusenerKrollLemma223}, this implies that $\mathcal{B}^c_\omega(\sigma_1^2\beta,1,1)=\omega_{c_1}^2 \mathcal{B}_\omega^c(\beta,1,1)$. Combining this equality with Lemma~\ref{lem:PotentialC11} and the fact that $\langle \sigma_1^2 \beta \rangle =t_{c_1}^{-2} \langle \beta \rangle$,
 we get 
\begin{align}
\label{eq:SoonFinished}
(\omega_{c_1}^2-1)
 \nabla_{\widehat{\sigma_1^2\beta}}(\omega)
&= (-1)^{n+1} \langle \sigma_1^2\beta \rangle \cdot \omega_{c_1}\cdots \omega_{c_n} \cdot g(\mathcal{B}^c_\omega(\sigma_1^2\beta,1,1))   \nonumber  \\ 
&=\omega_{c_1}^2(\omega_{c_1}^2-1)
 \nabla_{\widehat{\beta}}(\omega).
\end{align}
Note that we slightly abused notations by thinking of $g$ as being defined on $\C$ and noting that $g(\mathcal{B}^c_\omega(\sigma_1^2\beta,1,1))=\omega_{c_1}^4 g(\mathcal{B}_\omega^c(\beta,1,1))$. Regardless of this fact, simplifying the extraneous terms, we obtain the equality $\nabla_{\widehat{\sigma_1^2\beta}}(\omega)=\omega_{c_1}^2\nabla_{\widehat{\beta}}(\omega)$. We let the reader verify that this conclusion also holds if $\omega_{c_1}^2\cdots \omega_{c_n}^2=1$.
 Since we assumed that $\nabla_{\widehat{\beta}}(\omega)\neq~0$, we deduce that~$\nabla_{\widehat{\sigma_1^2\beta}}(\omega) \neq 0$. As Remark~\ref{rem:PotentialEvaluations} implies that the quotient $\nabla_{\widehat{\sigma_1^2\beta}}(\omega)/\nabla_{\widehat{\beta}}(\omega)$ is real, we obtain a contradiction when $\omega_{c_1}^2$ is different from~$1$. This concludes the proof of Proposition~\ref{prop:TechnicalLemma}.
\end{proof}

Note that in their equivalent of~(\ref{eq:SoonFinished}), Heusener and Kroll work with the Conway-normalized Alexander polynomial which they denote $\Delta_K(t)$ (recall Remark~\ref{rem:TerminologyPotential}). This explains why they obtain the equality $\Delta_{k'}(\omega)=\omega \Delta_k(\omega)$~\cite[last equation of p.494]{HeusenerKroll}, while we have a $\omega_{c_1}^2$ factor.

\section{The multivariable Casson-Lin invariant and crossing changes}
\label{sec:CrossingChange}

The goal of this section is to understand the behavior of the multivariable Casson-Lin invariant under a crossing change within a sublink. In Subsection~\ref{sub:ReductionPillowcase}, we reduce this analysis to a computation in a space $\widehat{H}_2^{\alpha_j}$, in Subsection~\ref{sub:PillowCaseComputations}, we perform calculations in $\widehat{H}_2^{\alpha_j}$ which are then reformulated in Subsection~\ref{sub:CrossingChangeCassonLin} in terms of the multivariable potential function.

\subsection{Reduction to a ``pillowcase-like" space.}
\label{sub:ReductionPillowcase}

Let $c$ be a $\mu$-coloring such that $c_1=c_2=j$. Let $\beta$ be an $n$-stranded $(c,c)$-braid and view the generator $\sigma_1 \in B_n$ as a $(c,c)$-braid. Let $\alpha$ be an element of $(0,\pi)^\mu$.
$$S_j(\alpha) = \lbrace (e^{\varepsilon_1 2i\alpha_1}, \ldots, e^{\varepsilon_\mu 2i\alpha_\mu}) \in S(\alpha) \ | \ \varepsilon_j=1 \rbrace.$$ 
This set contains $2^{\mu-1}$ elements and once again its elements are written as $\omega_\varepsilon$ with $\varepsilon$ in~$\lbrace \pm 1\rbrace^{\mu-1}$.
Although this fact is not needed in the sequel, observe that $S_j(\alpha)$ is in bijection with the set of conjugacy classes of abelian representations of $\pi_1(M_L)$ where the meridional traces of the sublink~$L_k$ are fixed to~$2 \cos (\alpha_{k})$ for $k=1,\ldots, \mu$. 
To see this, first simultaneously diagonalize these meridional matrices, yielding $\bsm e^{\varepsilon_k2i\alpha_k}&0\\0&e^{-\varepsilon_k2i\alpha_k} \esm$, and then use one extra conjugation to fix $\varepsilon_j=1$.

Assume that $\Delta_{\widehat{\beta}}(\omega_\varepsilon), \Delta_{\widehat{\sigma_1^2 \beta}}(\omega_\varepsilon) \neq 0$ for all $\omega_\varepsilon \in S_j(\alpha)$.
In order to understand the effect of a single crossing change within a sublink on the multivariable Casson-Lin invariant~$h_L$, we will study
\begin{equation}
\label{eq:CrossingChangeInvariant}
h^{c}_{\sigma_1^2\beta}(\alpha)-h^{c}_\beta(\alpha).
\end{equation}
Indeed the links $L:=\widehat{\beta}$ and $\widehat{\sigma_1^2\beta}$ differ by a single crossing change within the sublink $L_j$ and any such (negative to positive) crossing change within a colored link can be realized in this way, see the proof of Proposition~\ref{prop:CassonLinCrossingChange} below for further details. The first step in understanding~(\ref{eq:CrossingChangeInvariant}) is to consider the following set:
$$ V_n^{\alpha,c} =\lbrace (A_1,\ldots,A_n,B_1,\ldots,B_n) \in H_n^{\alpha,c} \ | \ A_i=B_i  \text{ for } i=3,\ldots,n \rbrace. $$
Use $c'$ to denote $(c_3,\ldots,c_n)$ so that $c=(c_1,c_2,c')$. 
Observe that $V_n^{\alpha,c}$ is homeomorphic to~$H_2^{\alpha_j} \times \Lambda_{n-2}^{\alpha,c'}$ and set $\widehat{V}_n^{\alpha,c}:=(V_n^{\alpha,c} \setminus S_n^{\alpha,c}) /\SO(3)$. Using Lemma~\ref{lem:DimensionIrred}, we deduce that this latter space is a smooth submanifold of $\widehat{H}_n^{\alpha,c}$ whose dimension is $2n-2$.
We then consider the projection $p \colon V_n^{\alpha,c} \to H_2^{\alpha_j}$ given by the following map:$$ p(X_1,X_2,X_3,\ldots,X_n,Y_1,Y_2,Y_3,\ldots,Y_n)=(X_1,X_2,Y_1,Y_2).$$ 
In order to obtain an induced map $\widehat{p}$ on (a subset of) $\widehat{V}_n^{\alpha,c}$, we introduce some further notations.
Namely,  we consider the subset $W_n^{\alpha,c} =p^{-1}(S_2^{\alpha_j})$ of $V_n^{\alpha, c}$ that projects onto the abelian representations in $H_2^{\alpha_j}$. 
We additionally set $\widehat{W}_n^{\alpha,c} := (W_n^{\alpha,c}\setminus S_n^{\alpha,c})/\SO(3)$: this way $p$ induces a well defined map 
$$\widehat{p} \colon \widehat{V}_n^{\alpha,c} \setminus \widehat{W}_n^{\alpha, c} \to \widehat{H}_2^{\alpha_j}.$$
 Arguing as in~\cite{Lin}, and perturbing $\widehat{\Gamma}^\alpha_\beta$ if necessary, we can assume that $\widehat{V}_n^{\alpha,c}$ and $\widehat{\Gamma}_\beta^\alpha$ intersect transversally in a properly embedded one-dimensional submanifold of $\widehat{H}_n^{\alpha,c}$.
 One can then further assume that $\widehat{\Gamma}_\beta^\alpha \cap \widehat{W}_n^{\alpha,c} = \emptyset$. 
 Now~(\ref{eq:CrossingChangeInvariant}) can be computed by considering curves inside the 2-dimensional space~$\widehat{H}_2^{\alpha_j}$. More precisely, using~$\langle-,-\rangle$ to denote the algebraic intersection number, the same arguments as in~\cite[Lemma 2.3]{Lin} and~\cite[Equation (4)]{HeusenerKroll}) show that
\begin{align}
\label{eq:IntersectionDifferenceCocycle}
h^{c}_{\sigma_1^2\beta}(\alpha)-h^{c}_\beta(\alpha)
&=\langle \widehat{\Gamma}_{\sigma_1^2}^{\alpha_j}-\widehat{\Lambda}_2^{\alpha_j},\widehat{p}(\widehat{V}_n^{\alpha,c} \cap \widehat{\Gamma}_\beta^\alpha) \rangle_{\widehat{H}_2^{\alpha_j}}.
\end{align}
As we will see in Proposition~\ref{prop:Curve} below, $\widehat{p}(\widehat{V}_n^{\alpha,c} \cap \widehat{\Gamma}_\beta^\alpha) $ consists of \emph{several} arcs, each approaching the same pair of punctures of $\widehat{H}_2^{\alpha_j}$.
Each of these arcs will contribute, or not, to the intersection number~(\ref{eq:IntersectionDifferenceCocycle}) in a way that we make precise in Proposition~\ref{prop:HeusenerKroll46} below. 
This contrasts with the situation described in~\cite{Lin, HeusenerKroll}, where $\widehat{p}(\widehat{V}_n^{\alpha,c} \cap \widehat{\Gamma}_\beta^\alpha) $ consists of a \emph{single} arc.

Note that we are adopting the following convention: we are writing $\widehat{\Gamma}_{\sigma_1^2}^{\alpha_j}, \widehat{\Lambda}_2^{\alpha_j}$ and $\widehat{H}_2^{\alpha_j}$ instead of $\widehat{\Gamma}_{\sigma_1^2}^{(\alpha_j,\alpha_j)}, \widehat{\Lambda}_2^{(\alpha_j,\alpha_j),(c_1,c_2)}$ and $\widehat{H}_2^{(\alpha_j,\alpha_j),(c_1,c_2)}$ which would be more coherent with the previous notation. Summarizing,~(\ref{eq:IntersectionDifferenceCocycle}) shows that the difference of the multivariable Casson-Lin invariants (which are defined via algebraic intersections in $\widehat{H}_n^{\alpha,c}$) can be understood in the more manageable space $\widehat{H}_2^{\alpha_j}$ by studing intersections with the \emph{difference cycle} $\widehat{\Gamma}_{\sigma_1^2}^{\alpha_j}-\widehat{\Lambda}_2^{\alpha_j}$. 

\subsection{Computations in $\widehat{H}_2^{\alpha_j}$}
\label{sub:PillowCaseComputations}

The goal of this subsection is to understand whether the projection $\widehat{p}(\widehat{V}_n^{\alpha,c} \cap \widehat{\Gamma}_\beta^\alpha)$ intersects the difference cycle $\widehat{\Gamma}_{\sigma_1^2}^{\alpha_j}-\widehat{\Lambda}_2^{\alpha_j}$: using~(\ref{eq:IntersectionDifferenceCocycle}), this will provide a formula for the difference $h^{c}_{\sigma_1^2\beta}(\alpha)-h^{c}_\beta(\alpha)$.
\medbreak

We first recall the parametrization of $\widehat{H}_2^{\alpha_j}= \lbrace (X_1,X_2,Y_1,Y_2) \in \SU(2)^4 \ | \ \Tr(X_i) = \Tr(Y_i) = 2\cos (\al_j), X_1X_2=Y_1Y_2 \rbrace$ which was obtained by Lin for $\alpha_j=\pi/2$~\cite[Lemma 2.1]{Lin} and by Heusener-Kroll for $\alpha_j \neq \pi/2$~\cite[Lemma 4.1]{HeusenerKroll}. Although the proofs may be found in the aforementioned references, we provide an outline of the arguments in order to introduce some notation which we shall use throughout the section.

\begin{lemma}
\label{lem:HKLemma41}
 Given $\alpha_j \in (0,\pi)$, the space $\widehat{H}_2^{\alpha_j}$ is homeomorphic to
 \begin{enumerate}
 \item a $2$-sphere with four points deleted if $\alpha_j=\pi/2$,
 \item a $2$-sphere with three points deleted if $\alpha_j \neq \pi/2$.
 \end{enumerate}
\end{lemma}

\begin{proof}
For $X,Y \in \SU(2)$, consider the $\operatorname{SU}(2)$-invariant distance on $\SU(2)$ given by $d(X,Y) := \arccos \left( \frac {\Tr(X^{-1}Y)}{2} \right)$. Notice that this distance realizes the distance induced by the standard spherical metric on $\mS^3$.
Let $(X_1,X_2,Y_1,Y_2)$ lie in $\widehat{H}_2^{\alpha_j}$. Up to conjugacy, one can assume that $$X_1 = \bm \cos (\al_j)+ i \sin (\al_j) \cos (\theta_1) & \sin (\al_j) \sin (\theta_1)\\ -\sin (\al_j) \sin (\theta_1) & \cos (\al_j)- i \sin (\al_j) \cos (\theta_1) \ema, \ \ \ X_2 = \bm e^{i\al_j} & 0 \\ 0 & e^{-i\al_j} \ema$$
for some $\theta_1 \in [0,\pi]$. As the distance $d$ is invariant, the matrices $X_1$ and $Y_1$ lie on a (possibly degenerate) circle given by the intersection of the sphere $\mS_{\al_j}(1)= \lbrace X \in \SU(2) \ | \ d(1,X)=\al_j \rbrace$ with the sphere $\mS_{\al_j}(X_1X_2) = \lbrace X \in \SU(2) \ | \ d(X,X_1X_2) = \al_j \rbrace$, see Figure \ref{Circle}. We denote by~$\theta_2 \in [0, 2\pi]$ the oriented angle between $X_1$ and $Y_1$ on this circle.
\begin{figure}[!htb]
\begin{center}
\def\svgwidth{0.5\columnwidth}
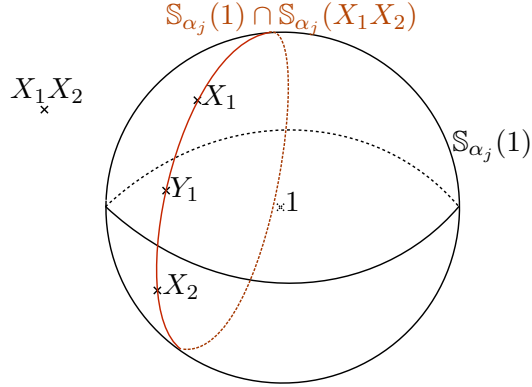
\caption{\label{Circle} The points $X_1$ and $Y_1$ on the red circle $\mS_{\al_j}(1)\cap \mS_{\al_j}(X_1X_2)$.
The angle $\theta_2$ is given by the euclidean angle between $X_1$ and $Y_1$ on this circle.}
\end{center}
\end{figure}
Two cases must be treated according to whether $\alpha_j=\pi/2$ or $\alpha_j \neq \pi /2$. These cases are respectively discussed in~\cite{Lin} and~\cite{HeusenerKroll}, but here is short outline.
\benu
\item \label{pi/2}
First suppose that $\al_j=\frac \pi 2$. In this case, the space $\widehat{H}^{\al_j}_2$ is parametrized by the two coordinates $\theta_1 \in [0,\pi]$ and $\theta_2 \in [0, 2\pi]$, with the identifications $(0, \theta_2) \sim (0, 2\pi-\theta_2)$, $(\pi,\theta_2) \sim (\pi,2\pi -\theta_2)$ and $(\theta_1, 0) \sim (\theta_1, 2\pi)$~\cite[Lemma 2.1]{Lin}. Let us briefly justify the appearance of these identifications.

When $\theta_1 = 0$, one has $X_1=X_2 = \bsm i&0\\0&-i \esm$ and therefore $X_1X_2=-1$. As a consequence, using the definition of $d$ and the fact that $\alpha_j=\pi/2$, the spheres $\mS_{\frac \pi 2}(1)$ and $\mS_{\frac \pi 2}(X_1X_2)$ coincide. Since $X_1=X_2$ is diagonal, after conjugating by a diagonal matrix, one can write $Y_1(\theta_2) = \bsm i\cos(\theta_2) & \sin (\theta_2)\\-\sin (\theta_2) & -i\cos(\theta_2) \esm$. We then notice that $Y_1(2\pi-\theta_2) = \bsm i&0\\0&-i\esm Y_1(\theta_2) \bsm i&0\\0&-i\esm^{-1}$, whence the announced identification.

If $\theta_1 = \pi$, then $X_1X_2 = 1$ and the same argument holds. Finally when $\theta_2=0$ and $\theta_2=2\pi$, we see that $Y_1=X_1$ which also leads to the claimed identifications. To conclude the proof of the first assertion, note that removing the abelian representations corresponds to removing the four points $A =(0,0), A'=(0,\pi), B=(\pi, 0), B'=(\pi, \pi)$.
\item
Next, assume that $\al_j \neq \frac \pi 2$. In this case, the parametrization is given by $\theta_1 \in [0,\pi]$ and $ \theta_2 \in [0, 2\pi]$ with identifications $(0, \theta_2) \sim (0,0), (\theta_1,0)\sim ( \theta_1,2\pi)$ and $(\pi, \theta_2)\sim (\pi, 2\pi-\theta_2)$~\cite[Lemma 4.1]{HeusenerKroll}. We  once again briefly justify the appearance of these identifications which are illustrated in Figure~\ref{pillow}.

When $\theta_1=0$, we have $X_1=X_2$ and the spheres $\mS_{\al_j}(1)$ and $\mS_{\al_j}(X_1X_2)$ are tangent, with intersection point $X_1=X_2=Y_1=Y_2$ (i.e. the red circle is ``degenerate": it is a unique point). This proves the identification $(0,0) = (0,\theta_2)$. The remaining identifications follow from the same argument as in the $\alpha_j=\pi/2$ case. Finally, removing the abelian representations corresponds to removing the three points $A=(0,0), B=(\pi,0), B' = (\pi,\pi)$.
\eenu
\begin{figure}[!htb]
\begin{center}
\def\svgwidth{0.6\columnwidth}
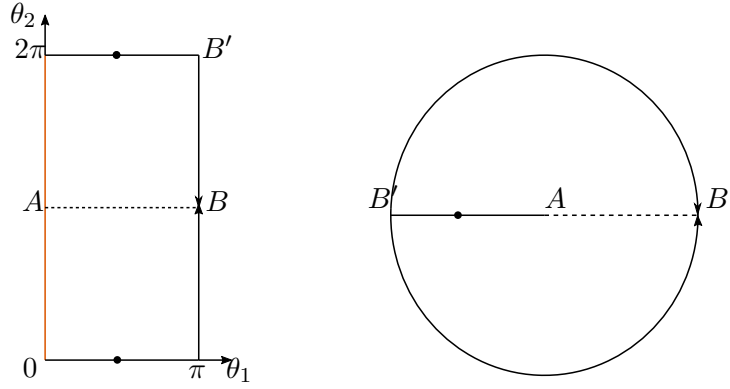
\caption{\label{pillow} The space $\widehat{H}^2_{\al_j}$, for $\al_j \neq \frac \pi 2$. 
On the left hand side: the left vertical edge (in red) is collapsed onto the point $A$, both the top horizontal edge and the bottom horizontal edge (marked with a dot) are identified, as are the right vertical edges above and below the point $B$, with orientations described by arrows.
On the right hand side: the result of the aforementioned identifications; gluing the two boundary segments joining B' to B produces the desired sphere.}
\end{center}
\end{figure}
This concludes our outline of the description of $\widehat{H}_2^{\alpha_j}$ and therefore the proof of the lemma.
\end{proof}

\begin{remark}
Since we aim to consider the algebraic intersection of $\widehat{\Gamma}_\beta^\al$ with the difference cycle $\widehat{\Gamma}_{\sigma_1^{-2}}^\al -\widehat{\Lambda}_n^{\al,c}$, it is worth mentioning that we lose nothing by working in $\widehat{V}_n^{\al,c}$, which is a strict subset of $(V_n^{\al,c} \setminus S_n^{\al,c} ) / \SO(3)$, see \cite[Lemma 5.2]{HarperSaveliev}. 
\end{remark}

Working in the space $\widehat{H}_2^{\alpha_j}$, we will now observe that near the puncture $A=(\textbf{e}^{i\alpha_j},\textbf{e}^{i\alpha_j},\textbf{e}^{i\alpha_j},\textbf{e}^{i\alpha_j})$ (which was also described in Lemma~\ref{lem:HKLemma41}), the projection $\widehat{p}(\widehat{V}_n^{\alpha,c} \cap \widehat{\Gamma}_\beta^\alpha)$ is a family of $2^{\mu-1}$ curves indexed by the elements of the set $S_j(\alpha)$.
\begin{proposition}
\label{prop:Curve}
Let $c$ be a $\mu$-coloring such that $c_1=c_2=j$ and let $\alpha \in (0,\pi)^\mu$. 
If $\omega:=\omega_\varepsilon \in S_j(\alpha)$ is
such that $\omega_j^2 \neq 1$ and $\beta$ is a $(c,c)$-braid such that~$\nabla_{\widehat{\beta}}(\omega) \neq 0,$ then, in a neighborhood of $A$ in $\widehat{H}_2^{\alpha_j}$, the projection $\widehat{p}(\widehat{\Gamma}_\beta^\alpha \cap \widehat{V}_n^{\alpha,c})$ is a family of $2^{\mu-1}$ curves.
\end{proposition}

\begin{proof}
Given $\theta$ in $(0,\pi)$, we use $\mathbf{e}^{i\theta}$ to denote the matrix~$\left( \begin{smallmatrix} e^{i\theta} &0 \\ 0 & e^{-i\theta} \end{smallmatrix} \right)$. 
For each $\varepsilon \in \lbrace \pm 1 \rbrace^{\mu-1}$, observe that $\textbf{a}:=\textbf{a}_{\varepsilon} = (\mathbf{e}^{i\alpha_j},\mathbf{e}^{i\alpha_j},\mathbf{e}^{\varepsilon_{c_3}i\alpha_{c_3}},\ldots,\mathbf{e}^{\varepsilon_{c_n}i\alpha_{c_n}})$ is
 an element in the subset $S_n^{\alpha,c}$ of abelian representations. Next, we consider the following subspace of $R_n^{\alpha,c} \times R_n^{\alpha,c}$:
$$ {\Lambda_n'}^{\alpha,c}=\lbrace (X_1,X_2,X_3,\ldots,X_n,Y_1,Y_2,X_3,\ldots,X_n) \in R_n^{\alpha,c} \times R_n^{\alpha,c} \rbrace.  $$
Since ${\Lambda_n'}^{\alpha,c}$ is $(2n+4)$-dimensional, $\Gamma_\beta^\alpha$ is $2n$-dimensional and $R_n^{\alpha,c} \times R_n^{\alpha,c}$ is $4n$-dimensional, we deduce that the dimension of the vector space $T_{(\mathbf{a},\mathbf{a})} {\Lambda_n'}^{\alpha,c} \cap T_{(\mathbf{a},\mathbf{a})} \Gamma_\beta^\alpha$ is at least $4$.
\begin{claim}
The dimension of $T_{(\mathbf{a},\mathbf{a})} {\Lambda_n'}^{\alpha,c} \cap T_{(\mathbf{a},\mathbf{a})} \Gamma_\beta^\alpha$ is equal to $4$. In particular, the manifolds $\Lambda_n'$ and~$\Gamma_\beta^\alpha$ intersect transversally at $(\mathbf{a},\mathbf{a})$.
\end{claim}
\begin{proof}
Using Remark~\ref{rem:MultiplicationByJ}, the tangent map of $\beta|_{R_n^{\alpha,c}}$ at $\mathbf{a}$ can be canonically identified with the unreduced colored Gassner matrix $\mathcal{B}^c_\omega(\beta)=\left( \begin{smallmatrix} A(\omega) & B(\omega) \\ C(\omega) &  D(\omega) \end{smallmatrix} \right)$. Since the tangent space to a graph is the graph of the corresponding derivative, the space $T_{(\mathbf{a},\mathbf{a})} {\Lambda_n'}^{\alpha,c} \cap T_{(\mathbf{a},\mathbf{a})} \Gamma_\beta^\alpha$ is isomorphic to the space~$X$ of $n$-tuples $v=(v_1,\ldots,v_n)$ which satisfy
\begin{equation}
\label{eq:TangentSpaceIntersection}
\begin{pmatrix} A(\omega) & B(\omega) \\ C(\omega) &  D(\omega)\end{pmatrix} \begin{pmatrix} v_1\\ v_2 \\v_3\\ \vdots \\ v_n \end{pmatrix}=\begin{pmatrix} * \\ *  \\v_3\\ \vdots \\ v_n \end{pmatrix}.
\end{equation}
The claim is therefore equivalent to the assertion that the real dimension of $X$ is~$2$. Since we assumed that $\nabla_{\widehat{\beta}}(\omega) \neq 0$ and $\omega_j^2 \neq 1$, Proposition~\ref{prop:TechnicalLemma} ensures that $\det(I_{n-2}-D(\omega))\neq~0$. Thus we deduce from~(\ref{eq:TangentSpaceIntersection}) that the last $n-3$ components of $v$ are equal to $(I_{n-2}-D(\omega))^{-1}C(\omega)\left( \begin{smallmatrix} v_1 \\ v_2 \end{smallmatrix} \right)$, finishing the proof of the first assertion; the second assertion follows immediately by recalling the respective dimensions of ${\Lambda_n'}^{\alpha,c}$ and $\Gamma_\beta^\alpha$. This concludes the proof of the claim.
\end{proof}

The claim implies that in a neighborhood of each of the $2^{\mu-1}$ points $(\mathbf{a}_\varepsilon,\mathbf{a}_\varepsilon)$, the space~$ {\Lambda_n'}^{\alpha,c} \cap~\Gamma_\beta^\alpha$ is a manifold of dimension $4$. 
Since $V_n^{\alpha,c} \cap \Gamma_\beta^{\alpha}$ is equal to ${\Lambda_n'}^{\alpha,c} \cap \Gamma_\beta^\alpha$, the same conclusion holds for this former space. 
Since each of the $(\mathbf{a}_\varepsilon,\mathbf{a}_\varepsilon)$ projects to $A$, after quotienting by $\SO(3)$ (and perturbing if necessary), we deduce that the projection $\widehat{p}(\widehat{V}_n^{\alpha,c} \cap \widehat{\Gamma}_\beta^\alpha )$ consists of at least~$2^{\mu-1}$ curves near $A$.

It remains to show that $\widehat{p}(\widehat{V}_n^{\alpha,c} \cap \widehat{\Gamma}_\beta^\alpha )$ consists precisely of $2^{\mu-1}$ curves (and not more). Reformulating, we assert that the arcs of $\widehat{p}(\widehat{V}_n^{\alpha,c} \cap \widehat{\Gamma}_\beta^\alpha )$ that approach $A$ are precisely parametrized by the $(\mathbf{a}_\varepsilon,\mathbf{a}_\varepsilon)$.
To see this, we must understand how the fiber above $A$ interacts with $\widehat{V}_n^{\alpha,c} \cap \widehat{\Gamma}_\beta^\alpha $.
 Observe that $(\textbf{a}_\varepsilon,\textbf{a}_\varepsilon) \in p^{-1}(A)$ (i.e. the fiber contains all the $(\mathbf{a}_\varepsilon,\mathbf{a}_\varepsilon)$) and $p^{-1}(A) \subset W_n^{\alpha,c}$: the former is clear while for the latter we use that $W_n^{\alpha,c}=p^{-1}(S_2^{\alpha_j})=p^{-1}(\lbrace A,B,B' \rbrace).$
Since~$\widehat{\Gamma}_\beta^{\alpha,c} \cap \widehat{W}_n^{\alpha,c}= \emptyset$, the assertion follows readily.
This concludes the proof of the proposition.

\end{proof}
From now on, we use $\mathcal{C}_{\varepsilon}$ to denote the arc of $\widehat{p}(\widehat{V}_n^{\alpha,c} \cap \widehat{\Gamma}_\beta^\alpha)$ corresponding to $\omega_\varepsilon \in S_j(\alpha)$, as described Proposition~\ref{prop:Curve}.
Perturbing if necessary, we can arrange that these $2^{\mu-1}$ arcs  intersect transversally.
 In particular, (\ref{eq:IntersectionDifferenceCocycle}) turns into
\begin{equation}
\label{eq:Intersection}
h^{c}_{\sigma_1^2\beta}(\alpha)-h^{c}_\beta(\alpha)
=\langle \widehat{\Gamma}_{\sigma_1^2}^{\alpha_j}-\widehat{\Lambda}_2^{\alpha_j},\widehat{p}(\widehat{V}_n^{\alpha,c} \cap \widehat{\Gamma}_\beta^\alpha) \rangle_{\widehat{H}_2^{\alpha_j}}
 = \sum\limits_{\varepsilon \in \lbrace \pm 1 \rbrace^{\mu-1}} \langle \widehat{\Gamma}_{\sigma_1^2}^{\alpha_j}-\widehat{\Lambda}_2^{\alpha_j},\mathcal{C}_\varepsilon \rangle_{\widehat{H}_2^{\alpha_j}}.
\end{equation}

Observe that $\mathcal{C}_{\varepsilon}$ lifts to a curve in $\widehat{V}_n^{\alpha,c} \cap \widehat{\Gamma}_\beta^\alpha$ approaching $\mathbf{a}_\varepsilon$.
Proposition~\ref{prop:Curve} shows that the question of whether $\mathcal{C}_\varepsilon$ intersects the difference cycle strongly depends on the position of this curve near $A$.

\begin{remark}
\label{rem:Graph}
As $\Gamma_\beta^{\alpha}$ is the graph of a function,
each component $\mathcal{C}_\varepsilon$ of the projection~$\widehat{p}(\widehat{V}_n^{\alpha,c} \cap~\widehat{\Gamma}_\beta^\alpha)$ is the graph of a function in $\widehat{H}_2^{\alpha_j}$, of the form $\theta_2=g_\varepsilon(\theta_1)$; recall the left hand side of Figure~\ref{pillow}.
\end{remark} 
 
Next, we let $\gamma \colon (-\delta, \delta) \to \widehat{H}_2^{\al_j}$ be (a parameterization of) a curve such that $\gamma(t)$ approaches~$A$ as~$t$ goes to $0$. Slightly abusing notations, we sometimes write $\gamma(0)=A$. The example to keep in mind is (a perturbation of) $\mathcal{C}_\varepsilon$. Recalling the definition and parametrization of $\widehat{H}_2^{\alpha_j}$, we write 
\begin{equation}
\label{eq:GammaCoordinates}
\gamma(t) = (X_1(t), X_2(t), Y_1(t), Y_2(t)).
\end{equation}
and follow~\cite{HeusenerKroll} by introducing the \emph{velocity} $\theta_1^0 = \frac d {dt} \theta_1 \vert_{t=0}$ and the \emph{angle} $\theta_2^0= \frac d {dt} \theta_2 \vert_{t=0}$ of such a curve $\gamma$. Still following~\cite{HeusenerKroll}, we define 
$$s(\theta_2^0) :=\frac{\cos(\al_j + \frac{\theta_2^0}{2})}{\cos(\al_j)} e^{\frac{i\theta_2^0}{2}}$$
 and observe that $2\arg s(\theta_2^0) = \theta_2^0$. The following remark is used implicitly in~\cite{HeusenerKroll}.

\begin{remark}\label{rem:derivative}
If the curve $\gamma$ is non constant, then we can choose $\theta_1(t)$ such that $\theta_1^0 \neq 0$. 
Assume by way of contradiction that $\theta_1^0=0$. Since $\gamma(0)=A$, this implies that $\gamma$ is tangent to the vertical axis $\lbrace \theta_1=0 \rbrace$ (recall Figure~\ref{pillow}). As this whole axis is collapsed to the point~$A$, the curve $\gamma$ must be constant, a contradiction.
Note also that when $\gamma=\mathcal{C}_\varepsilon$, this is a consequence of Remark~\ref{rem:Graph}: since $\mathcal{C}_\varepsilon$ is the graph of a function (in the $(\theta_1,\theta_2)$ coordinates), the derivative~$\theta_1'(t)$ cannot vanish.
\end{remark}

From now on, we consider the $2^{\mu-1}$ paths $\gamma_{\beta,\varepsilon}$ given by (a perturbation of) $\mathcal{C}_\varepsilon$ where~$\varepsilon$ lies in~$\lbrace \pm 1 \rbrace^{\mu-1}$; we then write $\theta_{1,\varepsilon}^0,\theta_{2,\varepsilon}^0$ for the corresponding velocity and angle, although at times, we will drop the $\varepsilon$ from the notation.
Using Remark~\ref{rem:derivative}, we suppose that $\theta_{1,\varepsilon}^0 = \frac 1 {\sin(\al_j)}$. 
As in the proof of Proposition~\ref{prop:Curve}, we write the unreduced colored Gassner matrix evaluated at~$\omega$ as $\mathcal{B}^c_\omega(\beta)=\left( \begin{smallmatrix} A(\omega) & B(\omega) \\ C(\omega) &  D(\omega) \end{smallmatrix} \right)$. The following lemma relates the angle $\theta_{2,\varepsilon}^0$ to this matrix.

\begin{lemma}
\label{claim:SlopeBurau}
Let $c$ be a $\mu$-coloring, let $\alpha$ be an element of $(0,\pi)^\mu$ and let $\omega:=\omega_\varepsilon \in S_j(\alpha)$ be such that $\omega_j^2 \neq 1$. Let $\beta$ be a $(c,c)$-braid which satisfies $\nabla_{\widehat{\beta}}(\omega) \neq 0$ and additionally set $v:=v_{\varepsilon}=(1-D(\omega))^{-1}C(\omega) \left( \begin{smallmatrix} 1 \\ 0 \end{smallmatrix} \right)$.
Then $s_\beta:=s_{\beta,\varepsilon} = s(\theta_{2,\varepsilon}^0)$ satisfies
$$ \mathcal{B}^c_\omega(\beta) \begin{pmatrix} 1\\0 \\ v  \end{pmatrix}=\begin{pmatrix} s_\beta \\ \overline{\omega}_j(1-s_\beta) \\ v \end{pmatrix}. $$
\end{lemma}
\begin{proof}
Write $\omega \in S_j(\alpha)$ as $\omega=(e^{\varepsilon_1 2i\alpha_1},\ldots, e^{2i\alpha_j}, \ldots, e^{\varepsilon_\mu 2i\alpha_\mu})$.
For $\theta \in (0,\pi)$, we write $\mathbf{e}^{i\theta}:=\bsm e^{i\theta}&0\\0 &e^{-i\theta} \esm$ and~$\textbf{a}:=\textbf{a}_{\varepsilon}= \left(  \mathbf{e}^{i\alpha_{j}}, \mathbf{e}^{i\alpha_{j}}, \mathbf{e}^{\varepsilon_{c_3}i \alpha_{c_3}},\ldots,\mathbf{e}^{\varepsilon_{c_n}i \alpha_{c_n}} \right)\in \SU(2)^n$.
Remark~\ref{rem:MultiplicationByJ} shows that the derivative~$T_\textbf{a}\beta$ of the map $\beta : R_n^{\al,c} \to R_n^{\al,c}$ at $\textbf{a}$ is given by the action of the colored Gassner matrix $\mathcal{B}^c_\omega(\beta)$ on~$\mC^n$ (after multiplication by $-\mathbf{j}$). Using the definition of~$v$, the proposition will follow once we show that
\begin{equation}
\label{eq:GassnersbetaGoal}
\mathcal{B}_\omega^c(\beta)|_{\C^2}  \begin{pmatrix} 1 \\ 0 \end{pmatrix} = \begin{pmatrix} s_\beta \\ \overline{\omega}_j(1-s_\beta) \end{pmatrix}.
\end{equation}
Writing the curve $\gamma_\beta(t):=\gamma_{\beta,\varepsilon}(t)$ as $(X_1(t),X_2(t),Y_1(t),Y_2(t))$, we first compute $\gamma'_\beta(0)$. Recalling the parametrization of $\widehat{H}_2^{\alpha_j}$ which was described in Lemma~\ref{lem:HKLemma41}, we see that $X_2'(0)=~0$.
Additionally using that $\theta_1(0)=0$ and that our parametrization satisfies $\theta_1^0:=\theta_{1,\varepsilon}^0=1/\operatorname{sin}(\alpha_j)$, we also get $X_1'(0) = \bsm 0 & \sin (\al_j) \cos(\theta_1(0)) \theta_1^0\\ -\sin (\al_j) \cos(\theta_1(0)) \theta_1^0&0 \esm = \bsm 0&1\\-1&0\esm$. Next, using~\cite[page~492, Equation~(5)]{HeusenerKroll} a computation shows that $Y_1'(0)=\left( \begin{smallmatrix} 0 & \theta_1^0 \sin(\alpha_j) s_\beta  \\ -\theta_1^0 \sin(\alpha_j) s_\beta & 0  \end{smallmatrix} \right)= \left( \begin{smallmatrix} 0 & s_\beta \\ -s_\beta & 0  \end{smallmatrix} \right)$.  Finally, since  $Y_2(t)=Y_1(t)^{-1}X_1(t)X_2(t)$, we also deduce that $Y_2'(0)$ is given by the matrix $\bsm 0& e^{-2i\al_j}(1-s_\beta) \\-e^{-2i\al_j}(1-s_\beta) & 0 \esm$. Summarizing and recalling Remark~\ref{rem:MultiplicationByJ}, we obtain
$$\gamma'_\beta(0) (-\textbf{j}) = (1,0,s_\beta, \overline{\omega}_j(1-s_\beta)).$$
Since $\gamma_\beta$ is a parametrization of the curve $\mathcal{C}_\varepsilon$ and since the tangent space of a graph can be described using the graph of the derivative, we deduce that~(\ref{eq:GassnersbetaGoal}) holds, concluding the proof of the proposition.
\end{proof}

Using Lemma~\ref{claim:SlopeBurau}, we make a first observation on the angles $\theta_{2,\varepsilon}^0$ near $A$. We stress the fact that each $s_{\beta,\varepsilon}$ depends on the corresponding $\omega:=\omega_\varepsilon \in S_j(\alpha)$.

\begin{proposition}
\label{prop:Lin24HeusenerKroll45}
The angle $\theta_{2,\varepsilon}^0$ of each $\mathcal{C}_\varepsilon$ is not equal to $0$ or $-4\alpha_j$.
\end{proposition}
\begin{proof}
We argue by means of contradiction. First assume that  $\theta_{2,\varepsilon}^0$ is equal to $0$. The definition of $s_{\beta,\varepsilon}$ implies that $s_{\beta,\varepsilon}=1$. Using Lemma~\ref{claim:SlopeBurau} (and its notations), this implies that the vector $w:=(1 \ 0 \ v_3 \  \ldots \ v_n)^t$ is fixed by the colored Gassner matrix. This is a contradiction since Lemma~\ref{lem:FixedPointGassner} shows that fixed vectors of the colored Gassner matrix do not contain a zero in their second coordinate.

Next, assume that $\theta_{2,\varepsilon}^0=-4\alpha_j$. In this case, $s_{\beta,\varepsilon}$ is equal to $\overline{\omega}_j=e^{-2i\alpha_j}$ (recall that $\omega=\omega_\varepsilon$ lies in $\mathcal{S}_j(\alpha)$). 
Applying Lemma~\ref{claim:SlopeBurau} to~$ \sigma_1^2\beta$, we obtain $\mathcal{B}^c_\omega(\sigma_1^2\beta)w=(s_{\sigma_1^2\beta,\varepsilon} \ \  \overline{\omega}_j(1-s_{\sigma_1^2\beta,\varepsilon}) \ \  v)^t$, where $v:=(v_3, \ \ldots \ v_n)$ is as in Lemma~\ref{claim:SlopeBurau}. Using the multiplicativity of the colored Gassner matrix, applying Lemma~\ref{claim:SlopeBurau} to~$\beta$ and recalling the Gassner matrix for $\sigma_1^2$ which was described in~Example~\ref{ex:RealityCheck}, we get
\begin{equation}
\label{eq:UsedInOtherProofDontDelete}
\begin{pmatrix} 1+\omega_j^2-\omega_j & -\omega_j^2+\omega_j \\ 1-\omega_j & \omega_j \end{pmatrix} \begin{pmatrix} s_{\beta,\varepsilon} \\ \overline{\omega}_j(1-s_{\beta,\varepsilon}) \end{pmatrix}= \begin{pmatrix} s_{\sigma_1^2\beta,\varepsilon} \\ \overline{\omega}_j(1-s_{\sigma_1^2\beta,\varepsilon}) \end{pmatrix}. 
\end{equation}
Since $s_{\beta,\varepsilon}=\overline{\omega}_j$, the left hand side of~(\ref{eq:UsedInOtherProofDontDelete}) is equal to $\left( \begin{smallmatrix} 1 \\ 0 \end{smallmatrix} \right)$. As a consequence, we obtain $s_{\sigma_1^2\beta,\varepsilon}=1$, contradicting the first paragraph of the proof. This concludes the proof of the proposition.
\end{proof}

We now build on~\cite[Lemma 4.6]{HeusenerKroll} in order to understand how the various $s_{\beta,\varepsilon}$ control the behavior of the multivariable Casson-Lin invariant under a crossing change within a given sublink.
 Since the argument is nearly the same as in~\cite{HeusenerKroll}, we only indicate the necessary modifications (we exceptionally chose to use the notation ${\omega_\varepsilon}_j$ to refer to the $j$-th component of $\omega_\varepsilon$; even though ${\omega_\varepsilon}_j=e^{2i\alpha_j}$ for each $\varepsilon \in \lbrace \pm \rbrace^{\mu-1}$).
\begin{proposition}
\label{prop:HeusenerKroll46}
Let $c$ be a $\mu$-coloring for which $c_1=c_2=j$, let $\alpha \in (0,\pi)^\mu$, and let $\beta$ be a $(c,c)$-braid whose induced permutation $\overline{\beta}$ satisfies $\overline{\beta}(1) \neq 1$ and~$\overline{\beta}(2) \neq 2$. 
If all $\omega_\varepsilon \in S_j(\alpha)$ satisfy ${\omega^2_\varepsilon}_j \neq 1,\nabla_{\widehat{\beta}}(\omega_\varepsilon) \neq~0$ and $\nabla_{\widehat{\sigma_1^2\beta}}(\omega_\varepsilon) \neq 0$, then the following equality holds:
$$ h^{c}_{\sigma_1^2\beta}(\alpha)-h^{c}_\beta(\alpha)= \# \left\lbrace \omega_\varepsilon \in S_j(\alpha) \ \vert \ \frac{{\omega_\varepsilon}_j s_{\beta,\varepsilon}-1}{s_{\beta,\varepsilon}-1}<0 \right\rbrace. $$
\end{proposition}
\begin{proof}
Recall from~(\ref{eq:Intersection}) that $h^{c}_{\sigma_1^2\beta}(\alpha)-h^{c}_\beta(\alpha)$ can be understood by studying the intersection of each $\mathcal{C}_\varepsilon$ with the difference cycle $ \widehat{\Gamma}_{\sigma_1^2}^{\alpha_j}-\widehat{\Lambda}_2^{\alpha_j}$ inside ${\widehat{H}_2^{\alpha_j}}$. We also know from Proposition~\ref{prop:Curve} that each $\mathcal{C}_\varepsilon$ approaches $A$.
Since each $\mathcal{C}_\varepsilon$ is the graph of a function (recall Remark~\ref{rem:Graph}), the~$\mathcal{C}_\varepsilon$ cannot be loops at $A$.
the conclusion now depends on the behavior near $B$~and~$B'$.
\begin{claim}
There is a neighborhood of $B'$ in $\widehat{H}_2^{\alpha_j}$ which is disjoint from $\widehat{p}(\widehat{V}_n^{\alpha,c} \cap \widehat{\Gamma}_\beta^\alpha)$.
\end{claim}
\begin{proof}
Suppose this not to be the case and recall that $B' \in \widehat{H}_2^{\alpha_j}$ is the point $(\mathbf{e}^{i\alpha_j},\mathbf{e}^{-i\alpha_j}, \mathbf{e}^{i\alpha_j},\mathbf{e}^{-i\alpha_j})$ and that $\widehat{p}$ is induced by $(A_1,A_2,\ldots,A_n,B_1,B_2,\ldots B_n) \mapsto (A_1,A_2,B_1,B_2)$. Observe that $B'=\widehat{p}(\mathbf{A},\mathbf{A})$, where $\mathbf{A}=(\mathbf{e}^{i\alpha_j},\mathbf{e}^{-i\alpha_j},A_3,\ldots,A_n).$ Using this notation, we deduce that there is a point in $\widehat{V}_n^{\alpha,c} \cap \widehat{\Gamma}_\beta^\alpha $ which is represented by the pair $(\mathbf{A},\mathbf{A})$.  There are now two cases both of which lead to contradictions. If $(\mathbf{A},\mathbf{A})$ represents an irreducible point, then we obtain the same contradiction as in~\cite[page 491]{HeusenerKroll}. On the other hand, if $(\mathbf{A},\mathbf{A})$ represents a reducible point, then the same argument as~\cite[page 496]{HeusenerKroll} shows that $\mathbf{A}$ is a fixed point of~$\beta|_{R_n^{\alpha,c}}$. Looking back to the definition of $\mathbf{A}$, this contradicts our assumption that $\overline{\beta}(1) \neq 1$ and $\overline{\beta}(2) \neq 2$, concluding the proof of the claim. 
\end{proof}
\begin{center}
\begin{figure}[!htb]
\def\svgwidth{1\columnwidth}
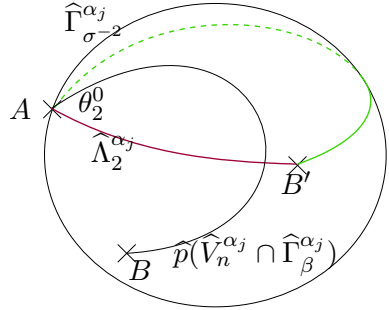
\caption{\label{disconnected} The union of the curves $\widehat{\Lambda}_2^{\al_j}$ (in red) and $\widehat{\Gamma}^{\alpha_j}_{\sigma_1^{-2}}$ (in green) form a circle that disconnects the punctured sphere $\widehat{H}_2^{\al_j}$.}
\end{figure}
\end{center}

Using the claim, we know that $\widehat{p}(\widehat{V}_n^{\alpha,c} \cap \widehat{\Gamma}_\beta^\alpha)$, and hence each curve $\mathcal{C}_\varepsilon$, must approach the points~$A$ and~$B$. The intersection properties of each of those curves now depend on the angle~$\theta_{2,\varepsilon}^0$. If $\theta_{2,\varepsilon}^0$ lies between~$0$ and $-4\al_j$, then the curve $\mathcal{C}_\varepsilon$ starts off in the connected component of~$\widehat{H}_2^{\al_j}\setminus (\widehat{\Lambda}_2^{\al_j} \cup~\widehat{\Gamma}_{\sigma_1^{-2}}^{\al_j})$ which does not contain~$B$. Since this curve eventually reaches $B$, it must intersect algebraically once positively the circle $ \widehat{\Lambda}_2^{\al_j} \cup \widehat{\Gamma}_{\sigma_1^{-2}}^{\al_j} $, such a situation is depicted in Figure \ref{disconnected}.

Similarly, if $\theta_{2,\varepsilon}^0$ is not between $0$ and $-4\alpha_j$, then the algebraic intersection of $\mathcal{C}_\varepsilon$ with the difference cycle will be zero. Heusener-Kroll~\cite[Section 4]{HeusenerKroll} now prove that these two situations correspond respectively to the cases $ \frac{{\omega_{\varepsilon}}_j s_{\beta,\varepsilon}-1}{s_{\beta,\varepsilon}-1}<0$ and $ \frac{{\omega_{\varepsilon}}_js_{\beta,\varepsilon}-1}{s_{\beta,\varepsilon}-1}>0$.
  The proposition now follows from~\eqref{eq:Intersection}.
\end{proof}

\subsection{The behavior under crossing changes and the Alexander polynomial}
\label{sub:CrossingChangeCassonLin}

In this subsection, we express the behavior of the multivariable Casson-Lin invariant under crossing changes in terms of the multivariable potential function.
\medbreak

Following closely the proof of~\cite[Lemma 2.7]{Lin}, the next result generalizes~\cite[Lemma 4.7]{HeusenerKroll}. In this latter reference, the authors use the Conway-normalized Alexander polynomial instead of the potential function: this explains the slight difference in their formula, see Remark~\ref{rem:TerminologyPotential}. 

\begin{proposition}
\label{prop:SlopeAlexander}
Let $c$ be a $\mu$-coloring for which $c_1=c_2=j$, let $\alpha \in (0,\pi)^\mu$, and let $\omega:=\omega_\varepsilon \in S_j(\alpha)$ be such that $\omega_j^2 \neq 1$ and $\omega_{c_1}\cdots \omega_{c_n} \neq 1$.
If $\beta$ is a $(c,c)$-braid such that $\nabla_{\widehat{\beta}}(\omega) \neq~0$ and $\nabla_{\widehat{\sigma_1^2\beta}}(\omega) \neq~0$, and if we write $s_\beta:=s_{\beta,\varepsilon}$ then we have
$$ \frac{\nabla_{\widehat{\beta}}(\omega)}{\nabla_{\widehat{\sigma_1^2\beta}}(\omega)}= \frac{s_{\beta}-1}{\omega_j^2 s_{\beta}-1}. $$
\end{proposition}
\begin{proof}
Let $\xi$ be an $n$-stranded $(c,c)$-braid such that $\nabla_{\widehat{\xi}}(\omega) \neq~0$. As in~Subsection~\ref{sub:Technical}, we use~$\mathcal{B}^c_\omega(\xi)$ to denote the colored Gassner matrix of $\xi$ evaluated at $\omega$. Setting $v:=(1,0,v_3,\ldots,v_n)^t$ and $x_\xi:=(s_\xi, \overline{\omega}_j(1-s_\xi),v_3,\ldots,v_n)^t$, Lemma~\ref{claim:SlopeBurau} implies that~$\mathcal{B}^c_\omega(\xi) v=x_\xi$. Writing $\mathcal{B}^c_\omega(\xi) $ as~$\left( \begin{smallmatrix} A(\omega) & B(\omega)  \\ C(\omega)  & D(\omega)  \end{smallmatrix} \right)$, we know from Proposition~\ref{prop:TechnicalLemma} that $I_{n-2}-D(\omega) $ is invertible. Using this fact to isolate the last $n-3$ vectors in the equation $\mathcal{B}^c_\omega(\xi) v=x_\xi$, we deduce that
$$ \begin{pmatrix}
s_\xi \\ \overline{\omega}_j(1-s_\xi)
\end{pmatrix}=(A(\omega) +B(\omega) (I_{n-2}-D(\omega) )^{-1}C(\omega) ) \begin{pmatrix}
1 \\ 0
\end{pmatrix}. $$
We can therefore write $A(\omega)+B(\omega)(I_{n-2}-D(\omega))^{-1}C(\omega)$ as $\left( \begin{smallmatrix} s_\xi & a \\ \overline{\omega}_j (1-s_\xi) & b   \end{smallmatrix} \right)$ for some~$a$ and $b$.
Next, we set $T=\left( \begin{smallmatrix} 1 & 0 \\ 0 & t \end{smallmatrix} \right)$, where $t$ is some indeterminate. Since $I_{n-2}-D(\omega)$ is invertible, a short computation using the formula $\det \bsm W&X\\ Y&Z \esm=\det(W-XZ^{-1}Y)\det(Z)$  for the determinant of a $(2 \times 2)$ block matrix $\bsm W&X\\ Y&Z \esm$ where $Z$ is invertible shows that
\begin{equation}
\label{eq:FirstStepPsi-T}
\det \left( \begin{pmatrix}T & 0 \\ 0 & I_{n-2} \end{pmatrix} -\mathcal{B}^c_\omega(\xi) \right)=\det \begin{pmatrix}
1-s_\xi & -a \\ \overline{\omega}_j(s_\xi-1) & t-b \end{pmatrix}  \det (I_{n-2}-D(\omega)). 
\end{equation}
We now compute the left hand side of~(\ref{eq:FirstStepPsi-T}). Let $M_1,\ldots,M_n$ be the columns of $M:=I_n-\mathcal{B}^c_\omega(\xi)$ and let $E_2$ be the column vector whose only non-zero entry is in the second position and is equal to $1$. Using these notations, the second column of  $\left( \begin{smallmatrix}T & 0 \\ 0 & I_{n-2} \end{smallmatrix} \right)-\mathcal{B}^c_\omega(\xi)$ is equal to $(t-1)E_2+M_2$.
Using the linearity of the determinant in its second column, we get
\begin{equation}
\label{eq:LinearityDeterminant}
 \det \left( \begin{pmatrix}T & 0 \\ 0 & I_{n-2} \end{pmatrix} -\mathcal{B}^c_\omega(\xi) \right)=\det(M)+(t-1)\det(M_1, E_2, M_3, \ldots, M_n).
\end{equation}
The first summand vanishes: the matrix $M=\mathcal{B}^c_\omega(\xi)-I_n$ has a nontrivial kernel since the colored Gassner matrix has fixed vectors. Recall from Section~\ref{sec:ColoredGassner} that we use $\mathcal{B}^c_\omega(\xi,l,m)$ to denote the determinant of the size $(n-1)$ matrix obtained by deleting the $l$-th row and $m$-th column of $M$. Expanding the second summand in~(\ref{eq:LinearityDeterminant}) along the second column, we obtain
$$\det \left(  \begin{pmatrix}T & 0 \\ 0 & I_{n-2} \end{pmatrix}-\mathcal{B}^c_\omega(\xi) \right)=(t-1)\mathcal{B}^c_\omega(\xi,2,2). $$
On the other hand, Lemma~\ref{lem:HeusenerKrollLemme221} gives $\omega_{c_1}(\omega_{c_2}-1)\mathcal{B}^c_\omega(\xi,1,1)=(\omega_{c_1}-1)\mathcal{B}^c_\omega(\xi,2,2)$, while Lemma~\ref{prop:HeusenerKrollLemma222} ensures that $\frac{\omega_{c_1}\cdots \omega_{c_n}-1}{\omega_{c_1}-1} \mathcal{B}^c_\omega(\xi,1,1)=   \det(\overline{\mathcal{B}}^c_\omega(\xi)-I_{n-1}).$ 
Using~(\ref{eq:FirstStepPsi-T}), and recalling that we assume $c_1=c_2$, we obtain
\begin{equation}
\label{eq:LinMistake?}
\det \begin{pmatrix}
1-s_\xi & -a \\ \overline{\omega}_j(s_\xi-1) & t-b \end{pmatrix}  \det (I_{n-2}-D(\omega))=\frac{\omega_{c_1}(\omega_{c_1}-1)}{\omega_{c_1}\cdots \omega_{c_n}-1} (t-1) \det(\overline{\mathcal{B}}^c_\omega(\xi)-I_{n-1}).
\end{equation}
We now set $t=1$ in~(\ref{eq:LinMistake?}) so that its right hand side vanishes. Since $\det(I_{n-2}-D(\omega))$ is non-zero, the leftmost determinant must vanish. But as $s_\xi$ cannot be equal to $1$ (recall the proof of Proposition~\ref{prop:Lin24HeusenerKroll45}),
we deduce that $a=\omega_j(1-b)$. A straightforward computation now shows that the leftmost determinant of~(\ref{eq:LinMistake?}) is equal to $(t-1)(1-s_\xi)$. Simplifying the $(t-1)$ terms, we have therefore obtained
\begin{equation}
\label{eq:LinMistake2?}
(1-s_\xi) \det (I_{n-2}-D(\omega))=\frac{\omega_{c_1}(\omega_{c_1}-1)}{\omega_{c_1}\cdots \omega_{c_n}-1} \det(\overline{\mathcal{B}}^c_\omega(\xi)-I_{n-1}).
\end{equation}
In order to apply~(\ref{eq:LinMistake2?}) to $\beta$ and $\sigma_1^2 \beta$, notice that the ``D blocks" of the colored Gassner matrices of $\beta$ and $\sigma_1^2\beta$ are equal: multiplying by $\mathcal{B}^c_\omega(\sigma_1^2)$ only affects the $A$ and~$B$ submatrices. Consequently, substituting $\xi$ with $\beta$ and $\sigma_1^2\beta$ in~(\ref{eq:LinMistake2?}) and taking quotients, we obtain
\begin{equation}
\label{eq:GassnerInsteadOfPotential}
\frac{1-s_\beta}{1-s_{\sigma_1^2\beta}}=
 \frac{\det(\overline{\mathcal{B}}^c_\omega(\beta)-I_{n-1})}{\det(\overline{\mathcal{B}}^c_\omega(\sigma_1^2\beta)-I_{n-1})}.
\end{equation}
A short computation using~(\ref{eq:UsedInOtherProofDontDelete}) shows that $s_{\sigma_1^2\beta}-1=\omega_j(\omega_js_\beta-1)$. To conclude the proof of the proposition, it thus only remains to express~(\ref{eq:GassnerInsteadOfPotential}) using the potential function instead of the reduced colored Gassner matrices. Since $\langle \sigma_1^2 \beta \rangle=t_{c_1}^{-2}\langle \beta \rangle$, Remark~\ref{rem:GassnerConway} and~(\ref{eq:GassnerInsteadOfPotential}) yield 
$$ \frac{\nabla_{\widehat{\beta}}(\omega)}{\nabla_{\widehat{\sigma_1^2\beta}}(\omega)}
=\omega_j^2\frac{\det(\overline{\mathcal{B}}^c_{\omega^2}(\beta)-I_{n-1})}{\det(\overline{\mathcal{B}}^c_{\omega^2}(\sigma_1^2\beta)-I_{n-1})}=\omega_j^2\frac{1-s_\beta}{\omega_j^2(\omega_j^2s_\beta-1)}.$$ 
Simplifying the $\omega_j^2$ terms concludes the proof of the proposition.
\end{proof}

We can now express the effect of an intra-component crossing change on the multivariable Casson-Lin invariant~$h_L$ in terms of the multivariable potential function. 

\begin{proposition}
\label{prop:CassonLinCrossingChange}
Let $L$ be a $\mu$-colored link and assume that $L_+$ is obtained from $L$ by changing a negative crossing within a connected component of $L_j \subset L$. 
Let $\alpha \in (0,\pi)^\mu$ be such that all~$\omega \in S_j(\alpha)$ satisfy $\omega_j^2 \neq 1$ and $\omega_{c_1}\cdots \omega_{c_n} \neq 1$. 
If $\nabla_{L}(\omega^{1/2}) \neq~0$ and~$\nabla_{L_+}(\omega^{1/2}) \neq~0$ for all $\omega \in S_j(\alpha)$, then the multivariable Casson-Lin invariants of $L$ and~$L_+$ satisfy
\begin{equation}
\label{eq:CrossingChangeh}
 h_{L_+}(\alpha)-h_{L}(\alpha)= \# \left\lbrace \omega \in S_j(\alpha) \vert \ \frac{\nabla_{L_+}(\omega^{1/2})}{\nabla_{L}(\omega^{1/2})}<0 \right\rbrace.
\end{equation}
\end{proposition}
\begin{proof}
Since $\nabla_L(\omega^{1/2}) \neq 0$ and $\nabla_{L_+}(\omega^{1/2}) \neq 0$ for all $\omega \in S_j(\alpha)$, we deduce from~(\ref{eq:ConwayVSAlexander}) that $\Delta_L(\omega) \neq 0$ and $\Delta_{L_+}(\omega) \neq 0$ for all $\omega \in S_j(\alpha)$.
Using the symmetry of the Alexander polynomial, the same assertion holds for all $\omega \in S(\alpha)$.
Therefore the multivariable Casson-Lin invariants $h_L(\alpha)$ and $h_{L_+}(\alpha)$ are defined. 
Assume that the crossing change occurs within a $j$-colored knot $K \subset L$.
Arguing as in~\cite[Remark 2.1]{ConwayEstier}, we can then assume that $L=\widehat{\beta}$ and~$L_+=\widehat{\sigma_1^2 \beta}$, where $\beta$ and $\sigma_1^2$ are $\mu$-colored $(c,c)$-braids with $c_1=c_2=j$, see Figure~\ref{CrossingChange}.
\begin{figure}[!htb]
\centering
\includegraphics[scale=0.9]{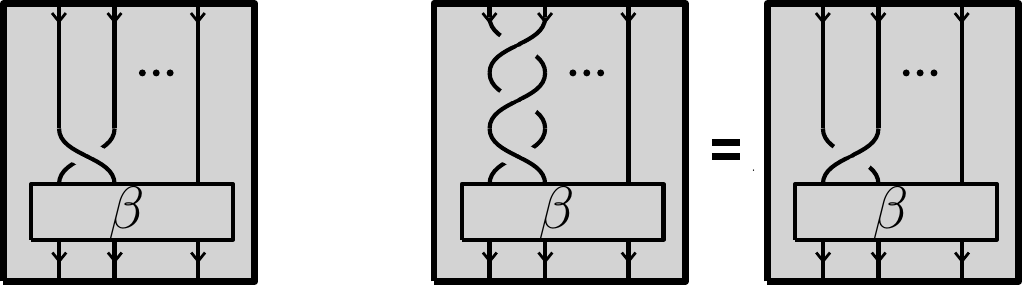}
\centering
\caption{On the left hand side, the braid $\beta$; on the right hand side, the braid $\sigma_1^2\beta$.}
\label{CrossingChange}
\end{figure}

Since the crossing change takes place within $K$, the permutation $\overline{\beta}$ induced by $\beta$ cannot satisfy $\overline{\beta}(1)=1$ and $\overline{\beta}(2)=2$.
We can therefore apply Proposition~\ref{prop:HeusenerKroll46} to deduce that
 $$ h^{c}_{\sigma_1^2\beta}(\alpha)-h^{c}_\beta(\alpha)= \# \left\lbrace \omega \in S_j(\alpha) \ \vert \ \frac{{\omega}_j s_{\beta,\varepsilon}-1}{s_{\beta,\varepsilon}-1}<0 \right\rbrace. $$
 Applying Proposition~\ref{prop:SlopeAlexander}, this equation can be rewritten as in~\eqref{eq:CrossingChangeh}. 
This concludes the proof of the proposition. 
\end{proof}
In~\cite{HeusenerKroll, Lin}, the condition in~(\ref{eq:CrossingChangeh}) is expressed as a product of polynomials instead of a quotient. Since these authors work with knots, the Conway-normalized Alexander polynomial evaluated at $\omega \in S^1$ is real (recall Remark~\ref{rem:TerminologyPotential}) and so the two formulations are in fact equivalent. The next remark describes the situation in the multivariable case.

\begin{remark}
\label{rem:QuotientProduct}
If $L$ and $L_+$ are $n$-component $\mu$-colored links as in Proposition~\ref{prop:CassonLinCrossingChange} and~$\omega \in~\T^\mu$, then the sign of $\frac{\nabla_{L_+}(\omega^{1/2})}{\nabla_{L}(\omega^{1/2})}$ is equal to the sign of $\nabla_{L_+}(\omega^{1/2})\nabla_{L}(\omega^{1/2})$ up to 
$(-1)^n$. Indeed, the quotient and the product agree up to multiplication by $\nabla_{L}(\omega^{1/2})^2$, and recalling Remark~\ref{rem:PotentialEvaluations}, this latter quantity equals $(-1)^n \nabla_{L}(\omega^{1/2})\nabla_{L}(\omega^{-1/2})=(-1)^n |\nabla_{L}(\omega^{1/2})|^2$.
\end{remark}

\section{The relation to the multivariable signature}
\label{sec:Signature}

In this section, we prove the main results of this paper. In more details, Subsection~\ref{sub:MultivariableSignature} gathers some facts about the multivariable signature, Subsection~\ref{sub:ProofSignatureCassonLin} proves Theorem~\ref{thm:Intro}, Subsection~\ref{sub:LocallyConstant} shows that $h_L$ is locally constant and Subsection~\ref{sub:Deformation} proves Theorem~\ref{thm:DeformationIntro}.

\subsection{The multivariable signature}
\label{sub:MultivariableSignature}

In this subsection, we briefly recall the definition of the multivariable signature, the main references being~\cite{Cooper} and \cite{CimasoniFlorens}.
\medbreak
A \emph{$C$-complex} for a~$\mu$-colored link~$L=L_1\cup\dots\cup L_\mu$ is a union~$S=S_1\cup\dots\cup S_\mu$ of surfaces in~$S^3$ which is connected, and such that:
\begin{enumerate}
\item for all~$i$,~$S_i$ is a Seifert surface for the sublink~$L_i$,
\item for all~$i\neq j$,~$S_i\cap S_j$ is either empty or a union of clasps,
\item for all~$i,j,k$ pairwise distinct,~$S_i\cap S_j\cap S_k$ is empty.
\end{enumerate}
The existence of a~$C$-complex for arbitrary colored links was established in~\cite[Lemma 1]{CimasoniPotential}. Given a sequence~$\varepsilon=(\varepsilon_1,\dots,\varepsilon_\mu)$ of signs~$\pm 1$, let~$i^\varepsilon\colon H_1(S)\to H_1(S^3\setminus S)$ be defined as follows. Any homology class in~$H_1(S)$ can be represented by an oriented cycle~$x$ which behaves as
illustrated in~\cite[Figure 2]{CimasoniFlorens} whenever crossing a clasp. Define~$i^\varepsilon([x])$ as
the class of the~$1$-cycle obtained by pushing~$x$ in the~$\varepsilon_i$-normal direction off~$S_i$ for~$i=1,\dots,\mu$. Next, consider the bilinear form
$$ \alpha^\varepsilon\colon H_1(S)\times H_1(S)\to\Z,\quad(x,y)\mapsto\ell k(i^\varepsilon(x),y)\,, $$
where~$\ell k$ denotes the linking number.
Fix a basis of~$H_1(S)$ and denote by~$A^\varepsilon$ the matrix of~$\alpha^\varepsilon$. Note that for all~$\varepsilon$, these \emph{generalized Seifert matrices} satisfy~$A^{-\varepsilon}=(A^\varepsilon)^T$. Using this fact, one easily checks that for any~$\mathbf{\omega}=(\omega_1,\dots,\omega_\mu)$ in the~$\mu$-dimensional 
torus~$\mathbb{T}^\mu$, the matrix
$$ H(\omega)=\sum_\varepsilon\prod_{i=1}^\mu(1-\overline{\omega}_i^{\varepsilon_i})\,A^\varepsilon $$
is Hermitian. Since this matrix vanishes when one of the coordinates of~$\omega$ is equal to~$1$, we restrict ourselves to the subset~$\mathbb{T}_*^\mu=(S^1\setminus\{1\})^\mu$
of~$\mathbb{T}^\mu$.

\begin{definition}
\label{def:MultivariableSignature}
The {\em multivariable signature and nullity} of a~$\mu$-colored link~$L$ are the maps $ \sigma_L,\eta_L \colon\mathbb{T}_*^\mu\to\Z \,,$ where~$\sigma_L(\omega)$ is the signature of $H(\omega)$ and~$\eta_L(\omega)$ its nullity.
\end{definition}

The multivariable signature and nullity are independent of the choice of the~$C$-complex~\cite{CimasoniFlorens}. Note furthermore that when $\mu=1$, a C-complex is nothing but a Seifert surface and~$\sigma_L$ recovers the Levine-Tristram signature of the oriented link.

\subsection{The multivariable signature and the multivariable Casson-Lin invariant}
\label{sub:ProofSignatureCassonLin}

The goal of this subsection is to relate the multivariable Casson-Lin invariant $h_L$ to $\sigma_L$ when $L$ is a 2-component ordered link with linking number 1, proving Theorem~\ref{thm:Intro} from the introduction.
\medbreak
The following lemma describes the parity of the multivariable signature and its behavior under crossing changes within a sublink. 

\begin{lemma}
\label{lem:SignatureLemma}
The multivariable signature satisfies the following properties:
\begin{enumerate}
\item If a $\mu$-colored link $L$ has $\nu$ components and $\omega \in \T_*^\mu$ is not a root of $\Delta_L$, then 
$$ \sigma_L(\omega) \equiv  \nu+\sum_{k<j} \ell k(L_k,L_j)-\operatorname{sign}(i^\nu \nabla_L(\omega^{1/2}))  \ \ \ \operatorname{ mod } 4.$$
In particular, if $L$ is a 2-component ordered link with linking number $1$ and $\omega \in \T_*^2$ is not a root of~$\Delta_L$, then $\sigma_L(\omega)$ is even and 
$$
\sigma_L(\omega) \equiv 
 \begin{cases}
0 \  \operatorname{ mod } 4 & \text{if } \nabla_L(\omega^{1/2})>0, \\
2 \  \operatorname{ mod } 4 & \text{if } \nabla_L(\omega^{1/2})<0.
\end{cases}$$
\item Assume that $L_+$ is obtained from $L$ by changing a unique negative crossing within a given sublink. If $\omega \in \T_*^\mu$ is neither a root of $\Delta_{L_+}$ nor a root of $\Delta_L$, then
$$ \sigma_{L_+}(\omega)-\sigma_L(\omega) \in \lbrace 0,-2 \rbrace. $$
\end{enumerate}
\end{lemma}
\begin{proof}
The first statement is contained in~\cite[Lemma 5.7]{CimasoniFlorens} and directly implies the claim about 2-component links with linking number $1$ (here $\nabla_L(\omega)$ is real since~$L$ has 2 components, see Remark~\ref{rem:PotentialEvaluations}).
Here, note that we can apply~\cite[Lemma 5.7]{CimasoniFlorens} since we assumed that $\Delta_L(\omega) \neq 0$: this hypothesis is equivalent to the assumption $\eta_L(\omega)=0$ made in~\cite[Lemma 5.7]{CimasoniFlorens}.
 We now prove the second statement. Pick C-complexes $S_+$ and $S$ for $L_+$ and~$L$ which only differ at the crossing under consideration. Since the crossing change occurs within a sublink, there are bases for $H_1(S_+)$ and $H_1(S)$ such that the resulting generalized Seifert matrix $A_+^\varepsilon$ only differs from $A^\varepsilon$ at one diagonal entry which is reduced by $1$. As a consequence, the Hermitian matrix~$H_+(\omega)$ is the same as $H(\omega)$ except for one diagonal entry which is reduced by the positive real number $\sum_{i=1}^\mu (2-\omega_i-\overline{\omega_i})$. Since only one eigenvalue can change and since we assumed both Alexander polynomials to be non-zero (i.e. there are no zero eigenvalues in~$H_+(\omega)$ and~$H(\omega)$), the result follows.
\end{proof}

Reformulating Lemma~\ref{lem:SignatureLemma}, we immediately obtain the following result.
\begin{lemma}
\label{lem:SignatureCrossingChange}
Let $L$ be a 2-component ordered link with linking number $1$ and assume that~$L_+$ is obtained from $L$ by a unique crossing change within a component of $L$. If $\omega \in \mathbb{T}_*^2$ is such that $\nabla_L(\omega^{1/2}) \neq 0$ and $\nabla_{L_+}(\omega^{1/2}) \neq 0$, then
$$ \sigma_{L_+}(\omega)-\sigma_L(\omega)=
\begin{cases}
0              & \mbox{if }  \nabla_{L_+}(\omega^{1/2})\nabla_L(\omega^{1/2}) >0, \\
-2              & \mbox{if }  \nabla_{L_+}(\omega^{1/2})\nabla_L(\omega^{1/2}) <0.
 \end{cases} 
$$
\end{lemma}

For 2-component links with linking number 1, we can now relate the multivariable Casson-Lin invariant to the multivariable signature, proving Theorem~\ref{thm:Intro} from~the~introduction.

\begin{theorem}
\label{thm:CassonLinSignature}
Let $L=K_1 \cup K_2$ be a $2$-component ordered link with $\ell k(K_1,K_2)=~1$, let~$(\alpha_1, \alpha_2) \in (0, \pi)^2$, and set $(\omega_1,\omega_2)=(e^{2i\alpha_1},e^{2i\alpha_2})$. 
If the multivariable Alexander polynomial satisfies $\Delta_L(\omega_1^{\varepsilon_1}, \omega_2^{\varepsilon_2}) \neq 0$ for all~$(\varepsilon_1,\varepsilon_2) \in \lbrace \pm 1 \rbrace^2$, then the following equality holds:
\begin{equation}
\label{eq:WantedTheoremProof} h_L(\alpha_1,\alpha_2)=\frac{-1}{2}\left(\sigma_L(\omega_1,\omega_2)+\sigma_L(\omega_1,\omega_2^{-1}) \right).
\end{equation}
\end{theorem}
\begin{proof}
Recall that we write $S(\alpha)=\lbrace (\omega_1^{\varepsilon_1},\omega_2^{\varepsilon_2}) \ | \ (\varepsilon_1,\varepsilon_2) \in \lbrace \pm 1 \rbrace^2 \rbrace$.
We first prove the theorem when $\alpha \in (0,\pi)^\mu$
is such that all $\omega \in S(\alpha)$ satisfy the following conditions: $\operatorname{arg}(\omega_{j})$ is transcendental, $\omega_ 1\omega_2 \neq 1$ and~$\omega_j^2 \neq 1$ for $j=1,2$.
 Since $L$ has $2$ components and linking number $1$, the Torres formula (which reads $\Delta_{K_1 \cup K_2}(t_1,1)\stackrel{.}{=}(t_1^{\ell_{12}}-1)/(t_1-1)\Delta_{K_1}(t_1)$, where $\ell_{12}=\ell k(K_1,K_2)$) shows that $|\Delta_L(1,1)|=1$. Thus $\Delta_L$ is not identically zero and therefore the multivariable Casson-Lin invariant $h_L(\alpha)$ is well defined whenever $\Delta_L(\omega) \neq 0$ for all $\omega \in S(\alpha)$. Since the fundamental group of the complement of the Hopf link $J$ is abelian,~$h_J$ vanishes identically. The same conclusion holds for the multivariable signature $\sigma_J$, as $J$ admits a contractible C-complex.

Since $\operatorname{arg}(\omega_j)$ is transcendental for $j=1,2$, it follows that $\nabla_L(\omega^{1/2}) \neq 0$ for all $L$ as in the statement of the theorem. 
The equality~\eqref{eq:WantedTheoremProof} is obtained by induction: both sides of this equation vanish on the (positive) Hopf link, and the next paragraph will show that they behave identically under crossing changes within components.
Notice that since
the links have linking number one, the Torres formula guarantees that such crossing changes do not make the Alexander polynomial vanish (consequently if~$h_L$ is defined for $L$, then it remains defined after performing such a crossing change).

To prove the induction step, we assume that~$L_+$ differs from $L=K_1 \cup K_2$ by a unique negative crossing within a component. 
Assume that this component is~$K_1$ (the reasoning is similar for~$K_2$).
Since we are working with 2-component links, Proposition~\ref{prop:CassonLinCrossingChange} and Remark~\ref{rem:QuotientProduct} imply that 
\begin{equation}
\label{eq:CrossingChangeProof}
 h_{L_+}(\alpha)-h_{L}(\alpha)= \# \left\lbrace \omega \in S_1(\alpha_1,\alpha_2) \vert \ \nabla_{L_+}(\omega^{1/2})\nabla_{L}(\omega^{1/2})<0 \right\rbrace.
\end{equation}
By definition, $S_1(\alpha_1,\alpha_2)$ contains two elements: $(\omega_1,\omega_2)$ and $(\omega_1,\omega_2^{-1})$. 
Let $\delta_{(\omega_1,\omega_2)}$ be $1$ or $0$ according to whether or not $\nabla_{L_+}(\omega^{1/2})\nabla_{L}(\omega^{1/2})<0$.
Using~\eqref{eq:CrossingChangeProof} and Lemma~\ref{lem:SignatureCrossingChange}, we deduce that 
\begin{align*}
 h_{L_+}(\alpha)-h_{L}(\alpha) 
 &= \delta_{(\omega_1,\omega_2)}+\delta_{(\omega_1,\omega_2^{-1})} \\
 &=\frac{-1}{2}\left(\sigma_{L_+}(\omega_1,\omega_2)-\sigma_{L}(\omega_1,\omega_2) \right)+\frac{-1}{2} \left( \sigma_{L_+}(\omega_1,\omega_2^{-1})-(\sigma_{L}(\omega_1,\omega_2^{-1}) \right)  \\
&=\frac{-1}{2}\left(\sigma_{L_+}(\omega_1,\omega_2)+\sigma_{L_+}(\omega_1,\omega_2^{-1})-(\sigma_L(\omega_1,\omega_2)+\sigma_L(\omega_1,\omega_2^{-1}))\right).
\end{align*}
This concludes the proof of the theorem for the
$\alpha \in (0,\pi)^2$
which were described above, since the linking number is a complete link homotopy invariant for 2-component links~\cite{Milnor}.

We conclude. View the right hand side of~\eqref{eq:WantedTheoremProof} as a function on $(0,\pi)^2$.
The result will follow if we prove that both sides of~\eqref{eq:WantedTheoremProof} are locally constant on $\lbrace \alpha \in (0,\pi)^2 \vert \  \Delta_L(\omega_1^{\pm1}, \omega_2^{\pm1}) \neq 0 \rbrace$: for the multivariable signature, this follows from~\cite[Corollary 4.2]{CimasoniFlorens}, while for $h_L$, the result is proved in Proposition~\ref{prop:locallyconst} below. This concludes the proof of the theorem.
\end{proof}

The sign appearing in Theorem~\ref{thm:CassonLinSignature} depends on some conventions which we briefly discuss.
\begin{remark}
\label{rem:CrossingChangeSignature}
 Given a knot $K$ obtained as the closure of a braid $\beta$, Lin writes $K_+=\widehat{\sigma_1^2\beta}$, while Heusener and Kroll write $K_-=\widehat{\sigma_1^2\beta}$. As a consequence, while these authors agree on the sign of $h_{\widehat{\sigma_1^2\beta}}(\alpha)-h_{\widehat{\beta}}(\alpha)$, comparing~\cite[Theorem 2.9]{Lin} with~\cite[Proposition 4.8]{HeusenerKroll} shows that the meaning of this sign differs: it depends on the conventions adopted for the generators of the braid group. We follow Lin's conventions (recall Figures~\ref{fig:CompositionThesisColor} and~\ref{CrossingChange}). On the other hand, assuming that~$K_+$ is obtained from~$K_-$ by changing a single negative crossing, Lin states that $0 \leq  \sigma_{K_+}(\omega)-\sigma_{K_-}(\omega) \leq~2$~\cite[page 356]{Lin}, while Heusener-Kroll state that $0 \leq  \sigma_{K_-}(\omega)-\sigma_{K_+}(\omega) \leq~2$~\cite[page 497]{HeusenerKroll}. With our notations, the proof of Lemma~\ref{lem:SignatureLemma} (as well as~\cite[proof of Lemma 2.1]{NagelOwens} and~\cite[proof of Lemma 2.2]{Garoufalidis}) yields the latter result. Summarizing, the sign differences in~\cite{Lin} and~\cite{HeusenerKroll} cancel out (explaining why these authors obtain ``$h_K=\sigma_K/2$") while our conventions account for the minus sign in Theorem~\ref{thm:CassonLinSignature}.
\end{remark}

\subsection{The multivariable Casson-Lin invariant is locally constant}
\label{sub:LocallyConstant}
Recall from Remark~\ref{rem:HypothesisDefinition} that $h_L$ is defined on the set
$$D_L =\lbrace \alpha \in (0,\pi)^2 \ | \ \rho_\omega \text{ is not a limit of irreducible representations for all } \omega \in S(\alpha) \rbrace.$$
 Since the inclusion $\lbrace \alpha \in (0,\pi)^2 \vert \  \Delta_L(\omega_1^{\pm1}, \omega_2^{\pm 1}) \neq 0 \rbrace \subset D_L$ was also observed in Remark~\ref{rem:HypothesisDefinition}, the following proposition concludes the proof of Theorem~\ref{thm:CassonLinSignature}.
\begin{proposition}
\label{prop:locallyconst}
Given a $\mu$-colored link $L$, the multivariable Casson-Lin invariant is locally constant on $D_L$. Namely, if~$\alpha^0$ and~$\alpha^1$ lie in the same connected component of $D_L$, then the following equality holds:
$$h_L(\alpha^0) = h_L(\alpha^1).$$
\end{proposition}

We first describe the strategy of the proof which is inspired by \cite[Proposition 3.8]{HeusenerKroll}. Let $\al \in (0,\pi)^\mu$. Given $\varepsilon >0$, we denote by $B(\al, \varepsilon)$ the ball of radius $\varepsilon$ centered at~$\al$. We will show that if $\varepsilon$ is small enough, then $h_L(\alpha)$ coincides with $h_L(\alpha')$ for any~$\al'$ in~$B(\al, \varepsilon)$. Writing~$L$ as the closure of an $n$-stranded $(c,c)$-braid $\beta$, this will be carried out by constructing a cobordism which joins $\widehat{\Lambda}_n^{\al,c} \cap \widehat{\Gamma}_\beta^{\al}$ to $\widehat{\Lambda}_n^{\al',c} \cap \widehat{\Gamma}_\beta^{\al'}$. This cobordism will take place in an ambient space whose description requires us to introduce the following spaces:
\begin{align*}
R_{n,2n}^c&= \lbrace (A_1, \ldots, A_n, B_1, \ldots, B_n) \in \SU(2)^{2n} \vert \ \tr(A_i)=\tr(B_i)=\tr(A_j)=\tr(B_j) \text{ if } c_i=c_j \rbrace, \\
 H_n^c& = \lbrace(A_1, \ldots, A_n, B_1, \ldots, B_n) \in R_{n,2n}^c  \ \vert \ \prod_{i=1}^n A_i = \prod_{i=1}^n B_i \rbrace.
\end{align*}
Recalling the notations from Section~\ref{sec:MultivariableCassonLin}, observe that we have the inclusions $R_{n}^{\alpha,c} \times R_{n}^{\alpha,c} \subset R_{n,2n}^c$ and $H_n^{\alpha,c} \subset H_n^c$.
Just as in Section~\ref{sec:MultivariableCassonLin}, we then define $S_n^c$ as the space of abelian representations in~$R_{n,2n}^c$ (i.e. we impose the same relations as in~\eqref{eq:AbelianLocus}) 
and define $\widehat{H}_n^c$ by removing $S_n^c \cap H_n^c$ from $H_n^c$
 and moding out by the action of $\SO(3)$. The next lemma is an analogue of Lemma~\ref{lem:DimensionIrred}; we also refer to~\cite[Corollary 3.2]{HeusenerOrientation} where a similar statement is made.

\begin{lemma}
\label{lem:Hnsubmanifolds}
The space $\widehat{H}_n^c$ is a smooth open manifold which contains $\widehat{H}_n^{\al,c}$ as a codimension~$\mu$ submanifold. Furthermore, the normal bundle of $\widehat{H}_n^{\al,c}$ inside $\widehat{H}_n^c$ is trivial.
\end{lemma}
\begin{proof}
The proof of the first statement is the same as in Lemma~\ref{lem:DimensionIrred}. Namely, the map $f_n \colon R_{n,2n}^c   \to\SU(2)$ defined~by $f_n(A_1, \ldots, A_n, B_1, \ldots, B_n) = A_1\cdots A_n B_n^{-1} \cdots B_1^{-1}$ restricts to a submersion~${f_n}_|$ on $H_n^c \setminus S_n^c$ and therefore $H_n^c \setminus S_n^c={f_n}_|^{-1}(\Id)$ is a smooth manifold whose dimension is equal to $\dim (R_{n, 2n}^c) - \dim(\SU(2))=4n+\mu-3$. Since $\SO(3)$ acts freely on~$H_n^c \setminus S_n^c$, the quotient $\widehat{H}_n^c$ is a smooth open manifold of dimension $4n-6+\mu$.
It is clear that $\widehat{H}_n^{\al,c}$ has codimension $\mu$ in $\widehat{H}_n^c$ because that many traces are fixed.

We now show that $\widehat{H}_n^{\al,c}$ has trivial normal bundle in $\widehat{H}_n^c$. Recall that for any $\theta \in (0,\pi)$, the $2$-sphere $\mS_\theta = \lbrace A \in \SU(2) \ \vert \ \Tr(A) = 2\cos(\theta) \rbrace$ has trivial normal bundle in $\SU(2)$: the Lie algebra $\su(2)$ splits as $\mC \oplus \mR$, the complex line being mapped onto the tangent space of $\mS_\al$ at~$A$ by the tangent map of multiplication by~$A$ and the real direction is spanned by the tangent map of the trace function $\Tr \colon \SU(2) \setminus \lbrace \pm \Id \rbrace \to (-2,2)$ at $A$. Denoting by~$(R_{n, 2n}^c)^*$ the subspace of $R_{n,2n}^c$ with none of its coordinates equal to $\pm \Id$,  and by $i_1, \ldots, i_\mu$ some preimages of $1, \ldots, \mu$ by the coloring $c$, the following map is thus a submersion:
\begin{align*}
\Tr_\mu \colon (R_{n,2n}^c)^* &\to (-2,2)^\mu \\
(A_1,\ldots, A_n, B_1, \ldots, B_n) &\mapsto (\Tr(A_{i_1}), \ldots, \Tr(A_{i_\mu})).
\end{align*}
Fiberwise, the normal bundle of $R_n^{\alpha,c} \times R_n^{\alpha,c}$ in $R_{n,2n}^c$ is given by $T_x R_{n,2n}^c /T_x (R_n^{\alpha,c} \times R_n^{\alpha,c})$, for any $x$ in $R_n^{\alpha,c} \times R_n^{\alpha,c}$. As a consequence, using $\mathcal{N}((R_n^{\alpha,c} \times R_n^{\alpha,c})/R_{n,2n}^{c})$ to denote the normal bundle of $R_n^{\alpha,c} \times R_n^{\alpha,c}$ inside $R_{n,2n}^c$, the map $\Tr_\mu $ induces a fiberwise isomorphism $\mathcal{N}((R_n^{\alpha,c} \times R_n^{\alpha,c})/R_{n,2n}^{c}) \to T (-2,2)^\mu$.
Since this latter bundle is trivial, so is the former. The statement now descends to the normal bundle of $H_n^{\alpha,c}$ inside $H_n^c$: indeed $H_n^{\al,c} \setminus S_n^{\al,c}$ (resp. $R_n^{\al,c} \times R_n^{\al,c}$) is a submanifold of codimension $\mu$ in $H_n^c\setminus S_n^c$ (resp. $R_{n,2n}^c$). This concludes the proof of the lemma.
\end{proof}

Using Lemma~\ref{lem:Hnsubmanifolds}, we can now prove  Proposition~\ref{prop:locallyconst} which asserts that $h_L$ is locally constant on $D_L$. The main idea is inspired by the proof of Ehresmann's fibration theorem~\cite{Ehresmann}.

\begin{proof}[Proof of Proposition~\ref{prop:locallyconst}]
Let $\al \in (0,\pi)^\mu$, fix $\varepsilon >0$ and use $B(\al, \varepsilon)$ to denote the ball of radius $\varepsilon$ centered in $\al$. We want to show that if $\varepsilon$ is small enough, then $h_L(\al')$ coincides with $h_L(\alpha)$  for any $\alpha' \in B(\al, \varepsilon)$. Pick an isotopy $F  \colon \widehat{H}_n^{\al,c} \times [0,1] \to \widehat{H}_n^{\al,c}$ which makes the intersection $\widehat{\Lambda}_n^{\al,c} \cap \widehat{\Gamma}_\beta^{\al}$ transverse in $\widehat{H}_n^{\al,c}$. Choose a path $\al \colon [0,1] \to (0,\pi)^\mu$ joining~$\al$ to $\al'$ 
In order to build a cobordism joining $\widehat{\Lambda}_n^{\al,c} \cap \widehat{\Gamma}_\beta^{\al}$ to $\widehat{\Lambda}_n^{\al',c} \cap \widehat{\Gamma}_\beta^{\al'}$, we will prove that~$F$ can be ``transported" along~$\al(t)$ so that for each $t$, the intersection $\widehat{\Lambda}_n^{\al(t),c} \cap \widehat{\Gamma}_\beta^{\al(t)}$ becomes transverse in $\widehat{H}_n^{\al(t),c}$.

Let $\mathcal{N}(\widehat{H}_n^{\al,c}/\widehat{H}_n^{c})$ denote the normal bundle of $\widehat{H}_n^{\al,c}$ inside of $\widehat{H}_n^{c}$. Since Lemma~\ref{lem:Hnsubmanifolds} ensures that this bundle is trivial, we can pick a nowhere vanishing normal vector field $X \colon \widehat{H}_n^{\al,c} \to \mathcal{N}(\widehat{H}_n^{\al,c}/\widehat{H}_n^{c})$ whose flow we denote by $\phi_X^t \colon \widehat{H}_n^c \to \widehat{H}_n^c$. Since the intersection $\widehat{\Lambda}_n^{\al(t),c}\cap \widehat{\Gamma}_\beta^{\al(t)}$ is compactly supported for each $t$, there is a compact set $K_0 \subset \widehat{H}_n^{\al,c}$ containing $\widehat{\Lambda}_n^{\al,c}\cap \widehat{\Gamma}_\beta^{\al}$ and such that for each $t$, the compact set $K_t =\phi_X^t(K_0)$ is a subset of  $\widehat{H}_n^{\al(t),c}$ containing $\widehat{\Lambda}_n^{\al(t),c} \cap \widehat{\Gamma}_\beta^{\al(t)}$. It can in fact safely be assumed that $K_0$ is a manifold. Let $\lbrace U_i \ \vert \ i \in I \rbrace$ be an open cover of $\widehat{H}_n^{\al,c}$, with finite subcover $\lbrace U_i \vert i= 1, \ldots, k \rbrace$ of $K_0$. Refining this sub-cover if necessary, one can assume that each open set $U_i\subset \widehat{H}_n^{\al,c}$ verifies the following property: for some $t \in [0,1]$, the set $\phi_X^t(U_i)$ contains only one component of the non-transverse intersection $\widehat{\Lambda}_n^{\al(t),c} \cap \widehat{\Gamma}_\beta^{\al(t)}$ in $\widehat{H}_n^{\al(t),c}$ (there are finitely number such components because we are dealing with (semi-)algebraic sets).


Since there are only finitely many non-transverse intersections, it is enough to show that for one such $U \subset \widehat{H}_n^{\al,c}$, one can transport the isotopy $F$ so that, for the corresponding $t$, the non-transverse intersection point of $\widehat{\Lambda}_n^{\al(t),c} \cap \widehat{\Gamma}_\beta^{\al(t)}$ in $\widehat{H}_n^{\al(t),c}$ is perturbed to a transverse one. To make this possible, consider the isotopy 
$$(\phi_X^t)^*F  \colon \phi_X^t(U) \times [0,1] \to \phi_X^t(U)$$
$$(p,s) \mapsto \phi_X^t\circ F(\phi_X^{-t}(p),s).$$
As vector fields $X$ such that the isotopy $(\phi_X^t)^* F$ makes this intersection transverse are generic in the set of normal vector fields, this procedure can always be carried out. 

We now conclude the proof. Pick $\varepsilon$ small enough so that each $\phi_X^t \colon U_i \to \widehat{H}_n^c$ is an embedding. The set $K = \bigcup\limits_{t \in [0,1]} \phi_X^t(K_0)$ is therefore a compact submanifold
of~$\widehat{H}_n^c$. The previous construction now ensures that $\bigcup\limits_{t \in [0,1]} \phi_X^t(\widehat{\Lambda}_n^{\al,c})$ and $\bigcup\limits_{t\in  [0,1]} \phi_X^t(\widehat{\Gamma}_\beta^{\al})$ can be assumed to intersect transversally in a one dimensional submanifold of $K$. This latter submanifold realizes the desired cobordism between  $\widehat{\Lambda}_n^{\al,c} \cap \widehat{\Gamma}_\beta^{\al}$ and $\widehat{\Lambda}_n^{\al',c} \cap \widehat{\Gamma}_\beta^{\al'}$. As a consequence, the corresponding intersections numbers are equal and therefore the proposition is proved.
\end{proof}

\subsection{Deformations of $\SU(2)$ representations of link groups}
\label{sub:Deformation}
The goal of this subsection is to prove Theorem~\ref{thm:DeformationIntro} from the introduction.
\medbreak

Recall that for a $\mu$-colored link $L$, the multivariable Alexander polynomial $\Delta_L(t_1^{\pm1},\ldots, t_\mu^{\pm1})$ restricts to a polynomial on the $\mu$-dimensional torus $\T^\mu$. 
Our interest lies in the zero locus 
$$V(\Delta_L)= \lbrace(\omega_1, \ldots, \omega_\mu) \in \T_*^\mu \vert \  \Delta_L(\omega_1,\ldots,\omega_\mu) = 0\rbrace.$$
We can now prove Theorem~\ref{thm:DeformationIntro} from the introduction. 

\begin{theorem}
\label{thm:DeformationTheorem}
Let $L$ be a 2-component ordered link with linking number $1$.
 Let $(\omega_1,\omega_2) \in~\T_*^2$ be such that $\Delta_L(\omega_1,\omega_2)=0$ and $\Delta_L(\omega_1,\omega_2^{-1}) \neq 0$. 
Assume that for any open subset $U \subset~\T_*^2$ containing~$(\omega_1,\omega_2)$, the multivariable signature $\sigma_L$ is not constant on~$U \setminus (V(\Delta_L) \cap U)$.
Then the abelian representation $\rho_{(\omega_1,\omega_2)}$  is a limit of irreducible representations.
\end{theorem}
\begin{proof}
Set $\omega:=(\omega_1,\omega_2)$.
By way of contradiction, assume that $\rho_\omega$ is not a limit of irreducible representations. 
Recall from Remark~\ref{rem:HypothesisDefinition} that the invariant $h_L$ is defined on the set
$$D_L =\lbrace \alpha \in (0,\pi)^2 \ | \ \rho_\omega \text{ is not a limit of irreducible representations for all } \omega \in S(\alpha) \rbrace.$$
Consider the continuous map
$\Pi \colon \T_\ast^2 \to (0,\pi)^2$ defined by $$\Pi(\omega_1, \omega_2) = \left(\frac 1 2 \arccos\left(\frac{\omega_1+\omega_1^{-1}}2\right), \frac 1 2 \arccos \left(\frac{\omega_2+\omega_2^{-1}}2\right) \right).$$
Since $\Pi$ is continuous and, by Proposition~\ref{prop:locallyconst}, $h_L \colon D_L \to \Z$ is locally constant, we see that the composition $h_L \circ \Pi$ defines a locally constant function on $\Pi^{-1}(D_L)$. 
Since Remark~\ref{rem:HypothesisDefinition} implies that $V(\Delta_L) \subset \Pi^{-1}(D_L),$ we can apply $h_L \circ \Pi$ to $(\omega_1,\omega_2)$.

Combining these facts, there is a small open neighborhood $U \subset  \Pi^{-1}(D_L)$ containing $(\omega_1,\omega_2)$ such that $h_L \circ \Pi$ is constant on $U$.
In particular $h_L \circ \Pi$ is constant on $U \setminus V(\Delta_L) \cap U$. 
Writing $\omega_i+\omega_i^{-1}=2 \operatorname{cos}(2 \alpha_i)$ and applying Theorem~\ref{thm:CassonLinSignature}, we deduce that 
\begin{align*}
h_L \circ\Pi(\omega_1, \omega_2) 
= -\frac{1}{2} (\sigma_L(\omega_1, \omega_2) + \sigma_L(\omega_1, \omega_2^{-1})).
\end{align*}
Since we established that $h_L \circ\Pi$ is constant on $U \setminus V(\Delta_L) \cap U$, the same holds for $(\omega_1,\omega_2) \mapsto -\frac{1}{2} (\sigma_L(\omega_1, \omega_2) + \sigma_L(\omega_1, \omega_2^{-1})).$
 Now observe that~$(\omega_1,\omega_2) \mapsto \sigma_L(\omega_1,\omega_2^{-1})$ is locally constant around $(\omega_1, \omega_2)$ because $\Delta_L(\omega_1, \omega_2^{-1}) \neq 0$~\cite[Corollary 4.2]{CimasoniFlorens}.
These facts imply that $\sigma_L$ is constant in a neighborhood of $(\omega_1,\omega_2) $ in $U \setminus (V(\Delta_L) \cap U)$.
This contradicts the hypothesis of the theorem, concluding the proof.
\end{proof}

\bibliography{biblioCassonLin}
\bibliographystyle{plain}

\end{document}

%% file: Circle.pdf_tex
\begingroup%
  \makeatletter%
  \providecommand\color[2][]{%
    \errmessage{(Inkscape) Color is used for the text in Inkscape, but the package 'color.sty' is not loaded}%
    \renewcommand\color[2][]{}%
  }%
  \providecommand\transparent[1]{%
    \errmessage{(Inkscape) Transparency is used (non-zero) for the text in Inkscape, but the package 'transparent.sty' is not loaded}%
    \renewcommand\transparent[1]{}%
  }%
  \providecommand\rotatebox[2]{#2}%
  \ifx\svgwidth\undefined%
    \setlength{\unitlength}{327.76015015bp}%
    \ifx\svgscale\undefined%
      \relax%
    \else%
      \setlength{\unitlength}{\unitlength * \real{\svgscale}}%
    \fi%
  \else%
    \setlength{\unitlength}{\svgwidth}%
  \fi%
  \global\let\svgwidth\undefined%
  \global\let\svgscale\undefined%
  \makeatother%
  \begin{picture}(1,0.63571743)%
    \put(0,0){\includegraphics[width=\unitlength,page=1]{Circle.pdf}}%
    \put(0.46219996,0.29375048){\color[rgb]{0,0,0}\makebox(0,0)[lb]{\smash{$1$}}}%
    \put(0,0){\includegraphics[width=\unitlength,page=2]{Circle.pdf}}%
    \put(0.32045965,0.47569664){\color[rgb]{0,0,0}\makebox(0,0)[lb]{\smash{$X_1$}}}%
    \put(0,0){\includegraphics[width=\unitlength,page=3]{Circle.pdf}}%
    \put(0.25512155,0.15270439){\color[rgb]{0,0,0}\makebox(0,0)[lb]{\smash{$X_2$}}}%
    \put(0,0){\includegraphics[width=\unitlength,page=4]{Circle.pdf}}%
    \put(0.26744951,0.31913169){\color[rgb]{0,0,0}\makebox(0,0)[lb]{\smash{$Y_1$}}}%
    \put(-0.00253258,0.4867918){\color[rgb]{0,0,0}\makebox(0,0)[lb]{\smash{$X_1X_2$}}}%
    \put(0,0){\includegraphics[width=\unitlength,page=5]{Circle.pdf}}%
    \put(0.74577387,0.3967978){\color[rgb]{0,0,0}\makebox(0,0)[lb]{\smash{$\mS_{\al_j}(1)$}}}%
    \put(0.26128549,0.61253689){\color[rgb]{0.70588235,0.2627451,0}\makebox(0,0)[lb]{\smash{$\mS_{\al_j}(1)\cap \mS_{\al_j}(X_1X_2)$}}}%
  \end{picture}%
\endgroup%

%% file: pillow.pdf_tex
\begingroup%
  \makeatletter%
  \providecommand\color[2][]{%
    \errmessage{(Inkscape) Color is used for the text in Inkscape, but the package 'color.sty' is not loaded}%
    \renewcommand\color[2][]{}%
  }%
  \providecommand\transparent[1]{%
    \errmessage{(Inkscape) Transparency is used (non-zero) for the text in Inkscape, but the package 'transparent.sty' is not loaded}%
    \renewcommand\transparent[1]{}%
  }%
  \providecommand\rotatebox[2]{#2}%
  \ifx\svgwidth\undefined%
    \setlength{\unitlength}{369.4556366bp}%
    \ifx\svgscale\undefined%
      \relax%
    \else%
      \setlength{\unitlength}{\unitlength * \real{\svgscale}}%
    \fi%
  \else%
    \setlength{\unitlength}{\svgwidth}%
  \fi%
  \global\let\svgwidth\undefined%
  \global\let\svgscale\undefined%
  \makeatother%
  \begin{picture}(1,0.5188349)%
    \put(0,0){\includegraphics[width=\unitlength,page=1]{pillow.pdf}}%
    \put(0.30392935,0.00281493){\color[rgb]{0,0,0}\makebox(0,0)[lb]{\smash{$\theta_1$}}}%
    \put(-0.0008987,0.50782786){\color[rgb]{0,0,0}\makebox(0,0)[lb]{\smash{$\theta_2$}}}%
    \put(0.01734109,0.00475969){\color[rgb]{0,0,0}\makebox(0,0)[lb]{\smash{$0$}}}%
    \put(0.25075667,0.00225205){\color[rgb]{0,0,0}\makebox(0,0)[lb]{\smash{$\pi$}}}%
    \put(0.00553061,0.45767077){\color[rgb]{0,0,0}\makebox(0,0)[lb]{\smash{$2\pi$}}}%
    \put(0.01616971,0.24062802){\color[rgb]{0,0,0}\makebox(0,0)[lb]{\smash{$A$}}}%
    \put(0.27486282,0.24062802){\color[rgb]{0,0,0}\makebox(0,0)[lb]{\smash{$B$}}}%
    \put(0.2693945,0.45814655){\color[rgb]{0,0,0}\makebox(0,0)[lb]{\smash{$B'$}}}%
    \put(0,0){\includegraphics[width=\unitlength,page=2]{pillow.pdf}}%
    \put(0.75233304,0.24489277){\color[rgb]{0,0,0}\makebox(0,0)[lb]{\smash{$A$}}}%
    \put(0.50331803,0.24489277){\color[rgb]{0,0,0}\makebox(0,0)[lb]{\smash{$B'$}}}%
    \put(0.97969458,0.24489277){\color[rgb]{0,0,0}\makebox(0,0)[lb]{\smash{$B$}}}%
    \put(0,0){\includegraphics[width=\unitlength,page=3]{pillow.pdf}}%
  \end{picture}%
\endgroup%

%% file: disconnect.pdf_tex
\begingroup%
  \makeatletter%
  \providecommand\color[2][]{%
    \errmessage{(Inkscape) Color is used for the text in Inkscape, but the package 'color.sty' is not loaded}%
    \renewcommand\color[2][]{}%
  }%
  \providecommand\transparent[1]{%
    \errmessage{(Inkscape) Transparency is used (non-zero) for the text in Inkscape, but the package 'transparent.sty' is not loaded}%
    \renewcommand\transparent[1]{}%
  }%
  \providecommand\rotatebox[2]{#2}%
  \ifx\svgwidth\undefined%
    \setlength{\unitlength}{1052.04729614bp}%
    \ifx\svgscale\undefined%
      \relax%
    \else%
      \setlength{\unitlength}{\unitlength * \real{\svgscale}}%
    \fi%
  \else%
    \setlength{\unitlength}{\svgwidth}%
  \fi%
  \global\let\svgwidth\undefined%
  \global\let\svgscale\undefined%
  \makeatother%
  \begin{picture}(1,0.25930854)%
    \put(0,0){\includegraphics[width=\unitlength,page=1]{disconnect.pdf}}%
    \put(0.19507917,0.16224378){\color[rgb]{0,0,0}\makebox(0,0)[lb]{\smash{}}}%
    \put(0.19445607,0.16033379){\color[rgb]{0,0,0}\makebox(0,0)[lb]{\smash{$A$}}}%
    \put(0.4241306,0.09657798){\color[rgb]{0,0,0}\makebox(0,0)[lb]{\smash{$B'$}}}%
    \put(0.29508272,0.02283633){\color[rgb]{0,0,0}\makebox(0,0)[lb]{\smash{$B$}}}%
    \put(0.26358888,0.1242311){\color[rgb]{0,0,0}\makebox(0,0)[lb]{\smash{$\widehat{\Lambda}_2^{\al_j}$}}}%
    \put(0.24131273,0.23561175){\color[rgb]{0,0,0}\makebox(0,0)[lb]{\smash{$\widehat{\Gamma}^{\al_j}_{\sigma^{-2}}$}}}%
    \put(0.33348986,0.03819918){\color[rgb]{0,0,0}\makebox(0,0)[lb]{\smash{$\widehat{p}(\widehat{V}_n^{\al_j}\cap\widehat{\Gamma}_\beta^{\al_j})$}}}%
    \put(0.2528349,0.16647893){\color[rgb]{0,0,0}\makebox(0,0)[lb]{\smash{$\theta_2^0$}}}%
    \put(-0.00142022,0.14266652){\color[rgb]{0,0,0}\makebox(0,0)[lb]{\smash{$ \ $}}}%
  \end{picture}%
\endgroup%